\documentclass[10pt,a4paper]{amsart}
\usepackage{a4wide}
\usepackage{geometry} 
\newgeometry{tmargin=2.5cm, bmargin=3cm, lmargin=2.5cm, rmargin=2.5cm}
\usepackage{amssymb}
\usepackage{enumerate}
\usepackage{esint}
\usepackage[colorlinks=true, urlcolor=red, linkcolor=red, citecolor=blue]{hyperref}

\theoremstyle{plain}
\newtheorem{theorem}{Theorem}[section]
\newtheorem{lemma}[theorem]{Lemma}
\newtheorem{corollary}[theorem]{Corollary}
\newtheorem{proposition}[theorem]{Proposition}

\theoremstyle{definition}
\newtheorem{definition}[theorem]{Definition}
\newtheorem{remark}[theorem]{Remark}

\numberwithin{equation}{section}

\DeclareMathOperator{\supp}{supp}
\DeclareMathOperator*{\osc}{osc}
\DeclareMathOperator*{\essosc}{ess\,osc}
\DeclareMathOperator*{\esssup}{ess\,sup}
\DeclareMathOperator*{\essinf}{ess\,inf}
\DeclareMathOperator*{\essliminf}{ess\,liminf}
\DeclareMathOperator{\DIV}{div}
\DeclareMathOperator{\dist}{dist}
\newcommand{\R}{{\mathbb{R}}}
\newcommand{\rn}{{\mathbb{R}^n}}
\newcommand{\cA}{{\mathcal{A}}}
\newcommand{\cF}{\mathcal{F}}
\newcommand{\cS}{\mathcal{S}}
\newcommand{\cW}{\mathcal{W}}
\newcommand{\dx}{\,{\mathrm{d}}x}
\newcommand{\dy}{\,{\mathrm{d}}y}
\newcommand{\dt}{\,\mathrm{d}t}
\newcommand{\dmu}{\,\mathrm{d}\mu}
\newcommand{\dtau}{\,\mathrm{d}\tau}

\makeatletter
\@namedef{subjclassname@2020}{\textup{2020} Mathematics Subject Classification}
\makeatother

\title[Orlicz Potential Theory]{Orlicz Potential Theory:\\ Balayage, Riesz Measures, and Very Weak Solutions}

\author{Iwona Chlebicka}
\address{Institute of Applied Mathematics and Mechanics, University of Warsaw,  Warsaw 2-097, Poland}
\email{i.chlebicka@mimuw.edu.pl}

\author{Minhyun Kim}
\address{Department of Mathematics \& Research Institute for Natural Sciences, Hanyang University, 04763 Seoul, Republic of Korea}
\email{minhyun@hanyang.ac.kr}

\author{Ying Li}
\address{School of Mathematics, Harbin Institute of Technology, Harbin 150001, China}
\email{lymath@hit.edu.cn}

\author{Chao Zhang}
\address{School of Mathematics and Institute for Advanced Study in Mathematics, Harbin Institute of Technology, Harbin 150001, China}
\email{czhangmath@hit.edu.cn}


\begin{document}

\begin{abstract}
We develop a nonlinear potential theory for elliptic equations with Orlicz growth under general monotonicity and growth conditions, without any homogeneity or scaling assumptions.

The lack of scaling invariance prevents the use of many classical tools from nonlinear potential theory. To overcome this difficulty, we establish a new framework that includes global H\"older regularity for obstacle problems, a balayage theory, the construction and analysis of Riesz measures associated with superharmonic functions, the identification of capacitary potentials, capacitary estimates for polar sets, and the quasicontinuity of superharmonic functions.

As an application of this theory, we prove that the classes of superharmonic functions and renormalized solutions to elliptic measure data problems coincide. This extends the classical equivalence theory from the homogeneous $p$-growth setting to general Orlicz growth and is new even for power-growth operators without homogeneity assumptions.
\end{abstract}

\maketitle

\section{Introduction}

\subsection{Objectives}
\label{ssec:obj}
Nonlinear potential theory studies elliptic equations of the form
\begin{equation}
    \label{eq:intro}
-\DIV \cA(x,\nabla u)=\mu \quad\text{in }\Omega\subset\R^n,
\end{equation}
with possibly rough data $\mu$ through the fine properties of their solutions, with central roles played by superharmonic functions, capacities, balayage, and Riesz measures. In the classical setting of $p$-Laplace type equations, nonlinear potential theory has become a fundamental tool for the analysis of singularities and measure data problems; see, for instance, \cite{HKM06,Kil11super,KM92,KM94,KM14}. At the same time, equations with Orlicz growth have attracted considerable attention as a natural extension of the standard $p$-growth framework; see, e.g.,~\cite{Bar15,Chl20,Chl23,CGZG19,CKW,CYZG,CM17,Lie91}. While substantial progress has been made in the variational and regularity theory of such equations, the corresponding potential theory remains far less developed. In particular, the classical connection between superharmonic functions and renormalized solutions, established in the homogeneous $p$-growth setting in \cite{Kil11super} under the assumption
\begin{equation}\label{eq-homogeneity}
\cA(x, \lambda \xi) = \lambda |\lambda|^{p-2} \cA(x, \xi),
\end{equation}
has no counterpart in the Orlicz framework. Many classical tools of nonlinear potential theory rely on scaling, whereas in the Orlicz setting scalar multiples of superharmonic functions (or solutions) generally fail to remain superharmonic (or solutions).

The purpose of this paper is to develop a scaling-free potential theory, where instead of a homogeneous $p$-growth operator from~\eqref{eq-homogeneity}, we consider the one that satisfies $\cA(x,\xi)\cdot\xi\approx G(\xi)$ for a doubling function~$G$, while also providing new insights into $p$-growth problems (see Remark~\ref{rem:p} for more details). Beyond their intrinsic interest, these potential-theoretic results play a central role in proving the equivalence between renormalized solutions to measure data problems and the corresponding generalized superharmonic functions. Main results are presented in Section~\ref{ssec:main-results}.

\subsection{Historical remarks}

The study of elliptic measure data problems has a long history, particularly in the standard $p$-growth setting. When the datum is merely a Radon measure, the classical weak formulation is generally inadequate, while distributional solutions may be pathological~\cite{Serrin-pat}. Foundational approaches to address this issue were developed in \cite{BBGGPV95,BG89,BG92,Dal96}. A different perspective comes from nonlinear potential theory, where superharmonic functions, defined with the use of comparison principle, associated with $p$-Laplace type equations provide a natural framework for singular right-hand sides. The connection between the two approaches was established in \cite{KM92,KM94} through sharp potential estimates for superharmonic solutions with nonnegative measure data. This viewpoint is compatible with approximation-based notions such as SOLA whenever both formulations are available; see \cite{BG89,BG92,KM92,KLP13}. We refer to \cite{HKM06} for the classical theory and to \cite{KM14} for an overview of further potential estimates.

The variational approach led to the notion of renormalized solutions. Originating in kinetic theory \cite{DL89}, this notion was later adapted to nonlinear elliptic and parabolic equations with rough data \cite{BM97,BGDM93,DMOP99}. Its key feature is the use of truncations of the solution, which makes it possible to formulate the equation even when the solution does not belong to the natural energy space. In the elliptic measure data setting, a systematic theory was developed in \cite{DMOP99}, based on the capacitary decomposition
$\mu=\mu_0+\mu_s$, where $\mu_0$ is diffuse and $\mu_s$ is concentrated on a set of zero capacity. Further results can be found in \cite{BMMP02,Blanchard13MA,Petitta08AMPA,Petitta11jee}.

The relation between the potential-theoretic and variational viewpoints was exploited in \cite{Kil11super} in the homogeneous $p$-growth setting. Superharmonic functions and renormalized solutions were shown to be locally equivalent: superharmonic functions satisfy the renormalized formulation with respect to their Riesz measures, while renormalized solutions admit superharmonic representatives.  

Let us now turn to equations with Orlicz growth. The study of elliptic equations with non-power growth conditions goes back to \cite{gos,Talenti} and was systematically developed in \cite{Lie91}. For measure data problems with Orlicz or more general growth, several notions of generalized solutions have been investigated. Existence, uniqueness, and regularity results for solutions obtained via approximation were established in \cite{Chl20,CGZG19,CM17}. Renormalized solutions constitute another important part of the theory. For $L^1$ data, they were studied in generalized Orlicz spaces in \cite{Gwiazda18JDE,Gwiazda12JDE}. For general bounded measures, a systematic theory was developed in \cite{Chl23}, where renormalized and approximable solutions were shown to exist and coincide. Thus, the approximation-based theory, usually phrased in terms of SOLA in the standard $p$-growth setting, has a natural counterpart in the Orlicz framework through approximable solutions.

The potential-theoretic side of the Orlicz theory has also developed in parallel. Generalized superharmonic functions associated with strongly nonlinear operators were introduced and studied in~\cite{Che22Gen}. Pointwise estimates in terms of nonlinear potentials were obtained in~\cite{Bar15,Benyaiche23pa,CGZG24}; in particular, two-sided Wolff potential estimates for superharmonic functions were established in \cite{CGZG24}. Fine potential-theoretic aspects of generalized Orlicz spaces, such as capacities, Wiener-type criteria, removable sets, and Kellogg-type properties, were studied in~\cite{BHH18,CK21remov,HJ22,LL21}. Related developments have also been made for systems~\cite{CKW,Chlebicka23cvpde,CYZG}, anisotropic problems~\cite{ACCZG19}, and nonlocal or mixed operators~\cite{CSYZ,kim25wolff,kim23}.

Nonlinear potential theory in the Orlicz-growth setting remains far less developed than in the homogeneous case, mainly due to the lack of scaling. Consequently, several fundamental tools, including balayage, Riesz measures, capacitary estimates for polar sets, and quasicontinuity of superharmonic functions, are unavailable in this setting. The aim of this paper is to develop a scaling-free potential theory and apply it to establish the equivalence between renormalized solutions and generalized superharmonic functions for measure data problems with Orlicz growth.

\subsection{Main results}\label{ssec:main-results}

We provide new potential-theoretic toolkit for the study on the measure data problem within the framework of Orlicz growth. Specifically, we consider the equation
\begin{equation}\label{eq:main}
-\DIV \cA(x,\nabla u)=\mu \quad\text{in }\Omega,
\end{equation}
where $\Omega \subset \rn$, $n \geq 2$, is a bounded open set, $\mu$ is a nonnegative Radon measure, and $\cA:\Omega \times \rn \to \rn$ is a Carath\'eodory vector field. Here, the growth of $\cA$ is governed by a function $G:[0, \infty) \to [0, \infty)$ defined by
\begin{equation}\label{eq:G}
G(t)=\int_0^t g(\tau)\dtau,
\end{equation}
where $g: [0, \infty) \to [0, \infty)$ is a strictly increasing continuous function.

Throughout the paper, we always assume that $G$ satisfies
\begin{equation}\label{eq:Gg}
pG(t) \le tg(t) \leq qG(t) \quad\text{for all }t \geq 0,
\end{equation}
with constants $1<p \leq q<\infty$, and that the vector field $\cA$ satisfies the strict monotonicity
\begin{equation}\label{eq-monotonicity}
(\cA(x,\xi)-\cA(x,\zeta))\cdot(\xi-\zeta)>0\quad\text{for every }\xi, \zeta \in \rn \text{ with }\xi \neq \zeta,
\end{equation}
and the coercivity and growth conditions
\begin{equation}\label{ass-op}
\begin{cases}
c_1^\cA G(|\xi|)\leq \cA(x,\xi)\cdot\xi, \\
|\cA(x,\xi)|\leq c_2^\cA g(|\xi|),
\end{cases}
\end{equation}
for some constants $c_1^\cA,c_2^\cA>0$. We emphasize that we do not impose any homogeneity assumption on $\cA$ in the gradient variable. In particular, even in the power case $G(t)=t^p/p$, we do not assume \eqref{eq-homogeneity}.\newline

Understanding the properties of generalised superharmonic functions is fundamental in nonlinear potential theory and the analysis of related quasilinear elliptic equations with measure data. As the natural generalization of classical superharmonic functions, they underpin key tools such as nonlinear potentials, capacities, and comparison principles. These properties are essential for existence, regularity, and pointwise estimates of solutions when standard Sobolev theory breaks down. This motivates the need for broader notions of generalised superharmonic functions, namely $\cA$-superharmonic function, through lower semicontinuity and comparison with $\cA$-harmonic functions (given by Definition~\ref{def-supersolution}).
\begin{definition}\label{def:Ash}
A lower semicontinuous function $u:\Omega\to \mathbb{R}\cup\{\infty\}$ is said to be \emph{$\cA$-superharmonic} in $\Omega$ (denoted $u \in \cS_{\cA}(\Omega)$) if $u\not\equiv\infty$ in each component of $\Omega$, and for every open set $U\Subset \Omega$ and every $h\in C(\overline{U})$ that is $\cA$-harmonic in $U$, the inequality $u\geq h$ on $\partial U$ implies $u\geq h$ in $U$.
\end{definition}

An important role in nonlinear potential theory is played by the exceptional sets on which superharmonic functions may blow up. \begin{definition}\label{def-polar}
A set $E \subset \rn$ is said to be \emph{$\cA$-polar} if there exists an open neighborhood $\mathcal{O}$ of $E$ and an $\cA$-superharmonic function $u$ in $\mathcal{O}$ such that $E \subset \{x \in \mathcal{O}: u(x) = \infty\}$.
\end{definition}
We show that such singular sets are negligible from the capacitary point of view; see Section~\ref{sec2} for the definition of the associated $G$-capacity.
\begin{theorem}\label{th-1}
Every $\cA$-polar set has $G$-capacity zero.
\end{theorem}
The proof relies on the development of several potential-theoretic tools in Sections~\ref{sec4-1}--\ref{sec4-6}. Specifically, we first establish global H\"older regularity for obstacle problems under the structural assumptions \eqref{eq:G}--\eqref{ass-op}; see Theorem~\ref{thm-bdry-reg} in Section~\ref{sec4-1}. This regularity serves as a key ingredient in the construction and approximation of balayages. We then develop a scale-free balayage theory and prove its connection with obstacle problems (Theorem~\ref{thm-obs-open}); see Sections~\ref{sec4-3} and \ref{sec4-4}. Using balayage, we establish the integrability needed to define Riesz measures associated with arbitrary $\cA$-superharmonic functions; see Section~\ref{sec4-5}. We further identify the $\cA$-potential of a compact set as the capacitary potential and derive quantitative estimates for the Riesz measures of balayages corresponding to different constant obstacles; see Section~\ref{sec4-6}. These estimates imply that every $\cA$-polar set has zero $G$-capacity; see Section~\ref{sec4-7}.

Theorem~\ref{th-1} extends a fundamental principle of nonlinear potential theory to the non-homogeneous Orlicz setting and provides the basis for the fine analysis of $\cA$-superharmonic functions. In particular, it leads to the following regularity property proven also in Section~\ref{sec4-7}. See \eqref{G-cap} in Section~\ref{sec2} for a relevant notion of capacity and Definition~\ref{def-quasicont} introducing quasicontinuity.

\begin{theorem}\label{thm-quasicont}
If $u \in \cS_{\cA}(\Omega)$, then $u$ is $G$-quasicontinuous in $\Omega$.
\end{theorem}

Having established that $\cA$-superharmonic functions are continuous outside sets of arbitrarily small $G$-capacity, we now have the fine regularity framework needed in the Orlicz setting. This property is essential for proving their equivalence with renormalized solutions, as it ensures that integrals against diffuse measures are well defined independently of the chosen representative. We therefore turn to the main application of this framework: the equivalence between $\cA$-superharmonic functions and renormalized solutions.

The formulation of renormalized solutions intrinsically relies on the decomposition of the measure $\mu = \mu_0 + \mu_s$, where $\mu_0$ is absolutely continuous with respect to the $G$-capacity, and $\mu_s$ is the singular part concentrated on a set of zero $G$-capacity. See Section \ref{sec2-4} for the discussion on this decomposition.

In what follows, for $k>0$ and $t\in \mathbb{R}$, the symmetric truncation operator $T_k$ is defined by
\begin{equation}\label{eq-truncation}
T_{k}(t):=\max\{-k,\min\{t, k\}\}.
\end{equation}
It is known in~\cite[Lemma~2.1]{BBGGPV95} that if $u$ is a function such that $T_k(u)\in W^{1,G}_{\rm loc}(\Omega)$ for every $k>0$, then there exists a unique measurable function $Z_u:\Omega\to\rn$ such that
\begin{equation}\label{eq-generalized-grad}
\nabla T_{k}(u)=\chi_{\{|u|<k\}}Z_u\quad \text{a.e.\ in }\Omega\text{ for every } k>0.
\end{equation}
With a slight abuse of notation, we write $\nabla u$ for this generalized gradient $Z_u$, which allows us to define one more notion of a solution to~\eqref{eq:main}.

\begin{definition}\label{def-ren}
Let $\mu$ be a nonnegative Radon measure on $\Omega$. We say that a measurable function $u: \Omega \to \overline{\R}$ is a \emph{renormalized solution} to \eqref{eq:main} if
\begin{equation}\label{eq:renormalizedpro-intro}
T_{k}(u) \in W^{1,G}_{\rm loc}(\Omega) \quad\text{for all $k>0$,} \quad g(|\nabla u|) \in L_{\rm loc}^{1}(\Omega),
\end{equation}
and
\begin{equation}\label{eq-renormalized-intro}
\int_{\Omega}\cA(x,\nabla u)\cdot\nabla (h(u)\varphi)\dx = \int_{\Omega}h(u)\varphi\dmu_0+ h(\infty)\int_{\Omega}\varphi\dmu_s
\end{equation}
for all $\varphi \in C^{\infty}_{c}(\Omega)$ and all $h\in W^{1,\infty}(\mathbb{R})$ such that $h'$ has compact support. Here $h(\infty)=\lim_{t \to \infty}h(t)$.
\end{definition}

The split form of the right-hand side in \eqref{eq-renormalized-intro} reflects a basic obstruction. Sobolev functions are defined only up to sets of zero capacity, while the singular part $\mu_s$ is concentrated precisely on such a set. Hence the integral $\int_\Omega h(u)\varphi\dmu_s$ is not canonically defined in terms of Sobolev representatives. The renormalized formulation avoids this ambiguity by assigning the asymptotic value $h(\infty)$ to the singular part of the measure.

We establish the exact equivalence between $\cA$-superharmonic functions and renormalized solutions to \eqref{eq:main}, extending \cite{Kil11super} to equations with Orlicz growth. The first implication shows that every variational solution admits a precise potential-theoretic representative.

\begin{theorem}\label{thm-ren-sup}
Let $\mu$ be a nonnegative Radon measure in $\Omega$, and let $u$ be a renormalized solution to \eqref{eq:main}. Then there exists an $\cA$-superharmonic function $\tilde{u}$ in $\Omega$ such that $\tilde{u}=u$ almost everywhere in $\Omega$.
\end{theorem}

For the converse implication, we recall the definition of the Riesz measure associated with an $\cA$-superharmonic function. For any $\cA$-superharmonic function $u$ in $\Omega$, there exists a unique nonnegative Radon measure $\mu[u]$, called the \emph{Riesz measure} of $u$, such that $u$ solves \eqref{eq:main} in the sense of distributions, i.e.,
\begin{equation*}
\int_{\Omega}\cA(x,\nabla u)\cdot \nabla \varphi\dx =\int_{\Omega} \varphi\dmu[u] \quad\text{for all $\varphi\in C^{\infty}_c(\Omega)$};
\end{equation*}
see Proposition~\ref{prop-Riesz}. The second implication shows that this distributional formulation automatically has the renormalized structure.

\begin{theorem}\label{thm-sup-ren}
Let $u$ be an $\cA$-superharmonic function in $\Omega$, and let $\mu[u]$ be its Riesz measure. Then $u$ is a renormalized solution to \eqref{eq:main} with datum $\mu[u]$.
\end{theorem}

Consequently, $\cA$-superharmonic functions and renormalized solutions coincide, up to the choice of the lower semicontinuous representative. In particular, the variational and potential-theoretic formulations describe the same class of solutions to \eqref{eq:main}.  We point out that Theorem~\ref{thm-sup-ren} is new even in the standard power-growth case $  G(t)=t^p/p  $, when the operator $  \cA  $ is not assumed to satisfy the homogeneity condition.

\begin{remark}\label{rem:p} Although the equivalence between renormalized solutions and $\cA$-superharmonic functions in~\cite{Kil11super} does not explicitly list~\eqref{eq-homogeneity} in assumptions and its proof does not directly use it, it relies on potential-theoretic facts such as the capacitary negligibility of polar sets from~\cite{HKM06} and~\cite{Mikkonen}. These results are deeply rooted in the classical nonlinear potential theory, which makes extensive use of scaling arguments. Without homogeneity, scalar multiples of $  \cA  $-superharmonic functions are not necessarily $  \cA  $-superharmonic, and the standard scaling techniques become unavailable. Thus, our result provides a genuinely scale-free counterpart to the classical equivalence theory even in the power-growth case.
\end{remark}
Finally, let us comment that Theorems~\ref{thm-ren-sup} and~\ref{thm-sup-ren} are of particular interest when the function $G$ does not grow too fast at infinity.
\begin{remark}
 The sharp Orlicz--Sobolev embedding (see \cite{Cian96}), which corresponds to the subcritical exponent regime $p\leq n$ in the standard $p$-growth case, requires the condition
\begin{equation*}
\int^{\infty}\left(\frac{t}{G(t)}\right)^{\frac{1}{n-1}}\dt = \infty.
\end{equation*}
If $G$ exhibits such fast growth at infinity that this condition is violated, we are equipped with the continuous embedding $W_0^{1,G}(\Omega)\hookrightarrow C(\overline{\Omega})$, provided that $\partial\Omega$ is sufficiently smooth (see \cite{Cian-pisa}). In this case, the Riesz measure $\mu[u]$ automatically belongs to the dual space $(W^{1,G}_0(\Omega))'$. Consequently, the problem reduces to the classical weak formulation. In this regime, the weak, renormalized, and $\cA$-superharmonic formulations coincide in a rather direct way. See also~\cite{CM17}.
\end{remark}

\subsection{Methods}
The main contribution of this paper is twofold: the development of a nonlinear potential theory for elliptic equations with non-homogeneous Orlicz growth (in particular under no counterpart of assumption~\eqref{eq-homogeneity}), and its application to the equivalence between renormalized solutions and $\cA$-superharmonic functions. The implication from renormalized solutions to $\cA$-superharmonic functions is established in Section~\ref{sec3}. The converse implication is substantially more delicate and relies on the potential-theoretic framework developed in Section~\ref{sec4}; the proof is completed in Section~\ref{sec5}. Let us summarize the main ideas and difficulties.

\medskip
\noindent\textbf{From renormalized solutions to $\cA$-superharmonic functions.}
In Section~\ref{sec3}, we prove Theorem~\ref{thm-ren-sup}. The main difficulty arises from the limited a priori regularity of renormalized solutions, so constructing an $\cA$-superharmonic representative requires proving that the solution does not blow down to $-\infty$. Inspired by the ideas from \cite{Kil11super,KM92,KM94,kim25wolff}, we study the negative part $u_-$ through suitable nonlinear test functions in the renormalized formulation. This leads to a Caccioppoli-type estimate and a De Giorgi iteration scheme adapted to the Orlicz setting, which yields local boundedness of $u_-$. We can then pass to the lower semicontinuous regularizations of the truncations and apply the Harnack convergence theorem.

\medskip
\noindent\textbf{Potential theory without scaling.}
The converse implication requires a detailed understanding of $\cA$-superharmonic functions. In the homogeneous $p$-growth setting, many fine properties of superharmonic functions rely on scaling and homogeneity, for instance through rescaling arguments for solutions or capacitary potentials. Such tools are not available in the Orlicz setting, where scalar multiples of solutions or superharmonic functions generally fail to preserve the structure of the equation. To overcome this difficulty, in Section~\ref{sec4} we develop a novel nonlinear potential theory adapted to non-homogeneous Orlicz growth. Starting from obstacle problems, we establish global H\"older regularity, develop a scale-free balayage theory and its relation to obstacle problems, prove integrability properties and the Riesz measure theorem for $\cA$-superharmonic functions, and derive capacitary estimates. These results imply that $\cA$-polar sets have zero $G$-capacity and yield the $G$-quasicontinuity of $\cA$-superharmonic functions.

\medskip
\noindent\textbf{From $\cA$-superharmonic functions to renormalized solutions.}
Let $u$ be an $\cA$-superharmonic function with Riesz measure $\mu=\mu_0+\mu_s$. In Section~\ref{sec5}, we prove Theorem~\ref{thm-sup-ren} by combining the potential theory of Section~\ref{sec4} with a localization argument near the singular set of $\mu_s$. The key ingredients are the $G$-quasicontinuity of $\cA$-superharmonic functions and the concentration property $\mu_s(\{u<\infty\})=0$, which allow us to identify the renormalized right-hand side with integration against the full Riesz measure. After introducing an auxiliary comparison problem and suitable truncation and cutoff procedures, we establish the required distributional identity and conclude the renormalized formulation.

\medskip
\noindent\textbf{A potential-theoretic heart.}
We emphasize that Theorems~\ref{thm-ren-sup} and \ref{thm-sup-ren} are not the sole contributions of this work. A substantial part of the paper develops a nonlinear potential theory for non-homogeneous Orlicz growth. Since this theory comprises several interconnected results rather than a single theorem, we outline its main components in Section~\ref{ssec:main-results}. We regard this framework as a central contribution of the paper: it is essential for proving Theorem~\ref{thm-sup-ren} and of independent interest in its own right.
We expect the potential-theoretic framework developed here to provide a basis for further study of fine properties of solutions and measure data problems in non-homogeneous settings.

\section{Preliminaries}\label{sec2}

In this section, we recall the definitions of growth functions, the relevant function spaces, $\cA$-supersolutions with their relation to $\cA$-superharmonic functions, capacities, and properties of renormalized solutions. We also collect several auxiliary results used in the sequel. Throughout the paper, we always assume that $\Omega \subset \rn$, $n \geq 2$, is a bounded open set, and that the structural assumptions \eqref{eq:G}--\eqref{ass-op} hold. Finally, the letters $c$ and $C$ denote positive constants that may change from line to line.

\subsection{Growth functions and relevant function spaces}

A function $G:[0,\infty)\to [0,\infty)$ is called an \emph{$N$-function} if $G(t)>0$ for $t>0$, satisfy~\eqref{eq:G} for some nondecreasing, left-continuous function $g:[0,\infty)\to [0,\infty)$, and
\begin{equation*}
\lim_{t\to \infty}\frac{G(t)}{t}=\lim_{t\to 0^+}\frac{t}{G(t)}=\infty.
\end{equation*}
As mentioned before, we assume throughout the paper that $G$ is an $N$-function satisfying $G \in C^1([0, \infty))$ and \eqref{eq:Gg}, with a strictly increasing function $g=G'$. We define $\bar{g}(t):=\frac{G(t)}{t}$ and $\bar{g}(0):=0$. Then, the assumption \eqref{eq:Gg} reads as
\begin{equation}\label{eq-barg-comp}
p\bar{g}(t)\leq g(t)\leq q\bar{g}(t).
\end{equation}
Moreover, $\bar g$ is strictly increasing on $[0,\infty)$, and hence invertible, since
\begin{equation*}
\bar{g}'(t) = \frac{tg(t)-G(t)}{t^2} \geq (p-1)\frac{G(t)}{t^2}>0 \quad\text{for all }t>0.
\end{equation*}
A key advantage of $\bar{g}$ is that it satisfies
\begin{equation}\label{eq-barg-der}
(p-1)\bar{g}(t)\leq t\bar{g}'(t)\leq (q-1)\bar{g}(t).
\end{equation}
Note that this property with $\bar{g}$ replaced by $g$ does not follow from \eqref{eq:Gg}. Let us summarize several basic algebraic inequalities involving the growth functions $G$, $g$, and $\bar{g}$. See, for instance, \cite[Section~2]{kim23}.

\begin{lemma}\label{lem-barg}
Let $c=p^{1/(q-1)}/q^{1/(p-1)}$. Under the assumptions on $G$ and $g$ described above, the following properties hold for all $t>0$.
\begin{enumerate}[\textup{(}i\textup{)}]
\item
For all $\lambda \geq 1$, it holds that
\begin{alignat*}{2}
&\lambda^p G(t) \leq G(\lambda t) \leq \lambda^q G(t),
&\quad&\lambda^{1/q} G^{-1}(t) \leq G^{-1}(\lambda t) \leq \lambda^{1/p} G^{-1}(t), \\
&\lambda^{p-1} \bar{g}(t) \leq \bar{g}(\lambda t) \leq \lambda^{q-1} \bar{g}(t),
&&\lambda^{1/(q-1)} \bar{g}^{-1}(t) \leq \bar{g}^{-1}(\lambda t) \leq \lambda^{1/(p-1)} \bar{g}^{-1}(t), \\
&\tfrac{p}{q} \lambda^{p-1} g(t) \leq g(\lambda t) \leq \tfrac{q}{p} \lambda^{q-1} g(t),
&&c \lambda^{1/(q-1)} g^{-1}(t) \leq g^{-1}(\lambda t) \leq c^{-1} \lambda^{1/(p-1)} g^{-1}(t).
\end{alignat*}
\item
For all $0<\lambda \leq 1$, it holds that
\begin{alignat*}{2}
&\lambda^q G(t) \leq G(\lambda t) \leq \lambda^p G(t),
&\quad&\lambda^{1/p} G^{-1}(t) \leq G^{-1}(\lambda t) \leq \lambda^{1/q} G^{-1}(t), \\
&\lambda^{q-1} \bar{g}(t) \leq \bar{g}(\lambda t) \leq \lambda^{p-1} \bar{g}(t),
&&\lambda^{1/(p-1)} \bar{g}^{-1}(t) \leq \bar{g}^{-1}(\lambda t) \leq \lambda^{1/(q-1)} \bar{g}^{-1}(t), \\
&\tfrac{p}{q} \lambda^{q-1} g(t) \leq g(\lambda t) \leq \tfrac{q}{p} \lambda^{p-1} g(t),
&&c \lambda^{1/(p-1)} g^{-1}(t) \leq g^{-1}(\lambda t) \leq c^{-1} \lambda^{1/(q-1)} g^{-1}(t).
\end{alignat*}
\item
For all $a, b \geq 0$ and $\delta>0$,
\begin{align*}
g(a)b \leq \delta g(a)a + g(b/\delta)b, \\
\bar{g}(a)b \leq \delta \bar{g}(a)a + \bar{g}(b/\delta)b.
\end{align*}
\end{enumerate}
\end{lemma}

The \emph{conjugate function} of an $N$-function $H$ is defined by \[H^\ast(\tau) := \sup_{t \geq 0} {(t\tau-H(t))}\quad\text{for all}~ \tau \geq 0.\]
Directly from this definition, one can infer the following \emph{Young's inequality}
\begin{equation}\label{eq-Young}
t\tau \leq H(t) + H^\ast(\tau) \quad\text{for all}~ t, \tau \geq 0.
\end{equation} 
The following lemma will also be useful, and its proof can be found in \cite[Lemma 2.2]{kim23}.

\begin{lemma}\label{lem-H}
Let $a \in (0, p]$ and set $H(t) := G(t^{1/a})$. Then $c\, G(t) \leq H^\ast(G(t)/t^a) \leq C G(t)$ for some $C \geq c>0$. In particular, $c\,G(t) \leq G^\ast(g(t)) \leq C G(t)$.
\end{lemma}

Next, we recall several function spaces associated with the growth functions $G$. The Orlicz space $L^G(\Omega)$ is defined as 
\begin{equation*}
	L^{G}(\Omega):=\left\{u:\Omega \to \mathbb{R}~\text{measurable}:\int_{\Omega}G(|u|)\dx< \infty \right\},
\end{equation*}
endowed with the Luxemburg norm 
\begin{equation*}
	\|u\|_{L^G(\Omega)}:=\inf\left\{\lambda>0:\int_{\Omega}G \left(\frac{|u|}{\lambda}\right)\dx\leq 1\right\}.
\end{equation*}
We define the Orlicz--Sobolev space $W^{1,G}(\Omega)$ by
\begin{equation*}
	W^{1,G}(\Omega)=\left\{u\in W_{\rm loc}^{1,1}(\Omega): u,|\nabla u|\in L^G(\Omega)\right\},
\end{equation*}
where the gradient is understood in the distributional sense. The norm on this space is given by  
\begin{equation*}
	\|u\|_{W^{1,G}(\Omega)}:=\inf\left\{\lambda>0:\int_{\Omega}G \left(\frac{|u|}{\lambda}\right)\dx+\int_{\Omega}G \left(\frac{|\nabla u|}{\lambda}\right)\dx\leq 1\right\}.
\end{equation*}
It is well known that the spaces $L^G(\Omega)$ and $W^{1,G}(\Omega)$ are separable and reflexive under assumption~\eqref{eq:Gg}. The space $W_0^{1,G}(\Omega)$ denotes the closure of $C^{\infty}_0(\Omega)$ with respect to this norm, and $(W_0^{1,G}(\Omega))'$ denotes its dual space. The space $W^{1, G}_{\mathrm{loc}}(\Omega)$ is defined as usual.

\subsection{\texorpdfstring{$\cA$}{A}-supersolutions and \texorpdfstring{$\cA$}{A}-superharmonic functions}

Let us define $\cA$-supersolutions and $\cA$-superharmonic functions, and summarize several of their key properties.

\begin{definition}\label{def-supersolution}
A function $u\in W^{1,G}_{\rm loc}(\Omega)$ is called an \emph{$\cA$-supersolution} to
\begin{equation}\label{eq-sec2-new}
-\DIV\cA(x,\nabla u)=0\quad\text{in }\Omega,
\end{equation}
if
\begin{equation}\label{eq-supersolution}
\int_{\Omega}\cA(x,\nabla u)\cdot\nabla \varphi\dx\geq 0  \quad \text{for all nonnegative }\varphi\in C^{\infty}_c(\Omega).
\end{equation}
Similarly, $u$ is called an \emph{$\cA$-subsolution} to \eqref{eq-sec2-new} if the above inequality holds with $\leq$ in place of $\geq$. Moreover, $u$ is called an \emph{$\cA$-solution} if it is both an $\cA$-supersolution and an $\cA$-subsolution. A function $u$ is said to be \emph{$\cA$-harmonic} in $\Omega$ if it is an $\cA$-solution in $\Omega$ and $u \in C(\Omega)$.    
\end{definition}

Note that neither $-u$ nor $\lambda u$ (for $\lambda>0$, $\lambda \neq 1$) is in general an $\cA$-sub- or $\cA$-supersolution when $u$ is an $\cA$-supersolution. A similar observation holds for the $\cA$-superharmonic and $\cA$-subharmonic functions discussed below. This stems from the non-homogeneous nature of the operator $\cA$, i.e., the lack of any counterpart of a typical assumption in the $p$-growth case recalled in~\eqref{eq-homogeneity}, cf.~\cite{HKM06,Mikkonen}.

\medskip
Recall that the definition of an $\cA$-superharmonic function is introduced in Definition~\ref{def:Ash}. An $\cA$-subharmonic function can be defined in a similar way: an upper semicontinuous function $u:\Omega\to \mathbb{R}\cup\{-\infty\}$ is said to be \emph{$\cA$-subharmonic} in $\Omega$ if $u\not\equiv -\infty$ in each component of $\Omega$, and for every open set $U\Subset \Omega$ and every $h\in C(\overline{U})$ that is $\cA$-harmonic in $U$, the inequality $u\leq h$ on $\partial U$ implies $u\leq h$ in $U$.

It is straightforward to see from definition that the truncation $\min\{u, k\}$ of an $\cA$-superharmonic function $u$ with level $k \in \mathbb{R}$ is also an $\cA$-superharmonic function.

\medskip

$\cA$-supersolutions and $\cA$-superharmonic functions are closely related to each other, which can be summarized as follows.

\begin{lemma}[{\cite[Lemma~4.4]{Che22Gen}}]\label{lem-supersol-cS}
If $u$ is an $\cA$-supersolution of \eqref{eq-sec2-new} that satisfies
\begin{equation}\label{eq-u-essliminf}
u(x) = \essliminf_{y \to x} u(y)\quad\text{for every }\ x \in \Omega,
\end{equation}
then $u\in\cS_{\cA}(\Omega)$.
\end{lemma}

\begin{lemma}\label{lem-super-liminf}
If $u$ is an $\cA$-supersolution of \eqref{eq-sec2-new} that is locally essentially bounded below, then $u$ has a lower semicontinuous representative satisfying \eqref{eq-u-essliminf}.
\end{lemma}

The proof of Lemma~\ref{lem-super-liminf} is almost identical to that of \cite[Theorem 3.63]{HKM06}, but we include it here for the sake of completeness.

\begin{proof}
Since $u$ is locally essentially bounded below by hypothesis, we may work locally and assume that $u$ is nonnegative. First, we assume that $M := \esssup_{\Omega} u < \infty$. Fix a point $x \in \Omega$ and a ball $B_r=B_r(x)$ such that $B_{2r} \subset \Omega$. Set $m_r := \essinf_{B_r} u$.   If $M=m_r$, then $u\equiv M$  a.e. on $B_r$, and the desired conclusion is trivial. Hence we assume $M>m_r$ for the following estimates.

 Applying the weak Harnack inequality~\cite[Proposition 3.20]{CGZG24} to $(u-m_r)$, we obtain, for some $s \in (0,1)$,
\begin{align*}
m_{\frac{r}{2}} - m_r+\frac{r}{2}
&\ge c \left( \fint_{B_r} (u(y) - m_r)^s \dy \right)^{1/s}  \ge c \, (M - m_r)^{(s-1)/s} \left( \fint_{B_r} (u(y)- m_r) \dy \right)^{1/s}.
\end{align*}
	This inequality can be rewritten as
	\begin{equation*}
	0 \le \fint_{B_r} u(y) \dy - m_r \le C (M - m_r)^{1-s} \left( m_{\frac{r}{2}} - m_r+\frac{r}{2} \right)^s \le c M^{1-s} \left( m_{\frac{r}{2}} - m_r+\frac{r}{2} \right)^s.
	\end{equation*}
	The right-hand side tends to zero as $r \to 0$. Consequently, for each $x \in \Omega$,
	\begin{equation}\label{eq-qc}
	\essliminf_{y \to x} u(y) = \lim_{r \to 0} \fint_{B_r(x)} u(y) \dy.
	\end{equation}
	By the Lebesgue differentiation theorem, the right-hand side of \eqref{eq-qc} equals $u(x)$ almost everywhere. Therefore, the statement holds for bounded $\cA$-supersolutions.
	
	Next, we treat the general case. For each $k \in \mathbb{N}$, the function $u_k := \min\{u,k\}$ is a bounded $\cA$-supersolution. By the first part of the proof, $u_k$ has a lower semicontinuous representative satisfying \eqref{eq-u-essliminf}. Thus, we may assume that $u$ is lower semicontinuous as the limit of an increasing sequence of lower semicontinuous functions. Let $\tilde{u}(x):=\essliminf_{y \to x} u(y)$. Then for almost every $x$ we have
	\begin{equation*}
	u(x) \le \liminf_{y \to x} u(y) \le \tilde{u}(x)\le \liminf_{r \to 0} \fint_{B_r(x)} u(y) \dy = u(x)
	\end{equation*}
 by the Lebesgue differentiation theorem. Hence $\tilde{u} = u$ a.e.\ and $\tilde{u}(x) = \essliminf_{y \to x} \tilde{u}(y)$.
\end{proof}

\begin{lemma}[{\cite[Lemma 4.6]{Che22Gen}}]\label{lem-cS-supersol}
If $u\in \cS_{\cA}(\Omega)$ is locally bounded above, then $u\in W_{\rm loc}^{1,G}(\Omega)$ and $u$ is an $\cA$-supersolution of \eqref{eq-sec2-new}.
\end{lemma}

As a corollary of Lemmas~\ref{lem-supersol-cS} and~\ref{lem-cS-supersol}, we obtain the following result. For the subharmonic part, one applies the preceding lemmas to the reflected operator $\widetilde{\mathcal A}(x,\xi):=-\mathcal A(x,-\xi)$, which satisfies the same structural assumptions \eqref{eq-monotonicity}--\eqref{ass-op}.

\begin{corollary}
A function $u$ is $\cA$-harmonic in $\Omega$ if and only if $u$ is both $\cA$-superharmonic and $\cA$-subharmonic in $\Omega$.
\end{corollary}

We also present the strong minimum principle for $\cA$-superharmonic functions.

\begin{proposition} \label{prop-min}
Let $\Omega$ be a bounded domain and let $u \in\cS_{\cA}(\Omega)$. If $u$ attains its minimum inside $\Omega$, then $u$ is a constant function.
\end{proposition}

\begin{proof}
This is proved in \cite[Theorem 4]{Che22Gen} when $u$ is finite a.e.\ in $\Omega$. Now, assume that $u \in \cS_{\cA}(\Omega)$ attains its minimum $m$ inside $\Omega$. Then, for any $k > m$, the function $u_k:=\min\{u, k\} \in \cS_{\cA}(\Omega)$ is finite everywhere and attains the same minimum $m$. By the aforementioned result, $u_k \equiv m$. Since this holds for every $k>m$, we conclude that $u \equiv m$ in $\Omega$.
\end{proof}

The following convergence theorem for $\cA$-superharmonic functions plays a fundamental role in the sequel. This result was established in \cite[Theorem~2]{Che22Gen} under the assumption that each $\cA$-superharmonic function in the sequence is finite almost everywhere, but we observe that this assumption is actually redundant for the validity of the proof.

\begin{proposition}
    \label{prop-conv-super}
The pointwise limit of a nondecreasing sequence of $\cA$-superharmonic functions in a domain $\Omega$ is either $\cA$-superharmonic in $\Omega$ or identically infinite in $\Omega$.
\end{proposition}

A similar result holds for $\cA$-harmonic functions as well. See \cite[Corollary~4.9]{Che22Gen} for its proof.

\begin{proposition}
    \label{prop-conv-har}
The pointwise limit of a nondecreasing sequence of $\cA$-harmonic functions in a domain $\Omega$ is either $\cA$-harmonic in $\Omega$ or identically infinite in $\Omega$.
\end{proposition}

\subsection{Capacities}\label{sec2-3}

We recall three capacities that will be used throughout the paper: the (Sobolev) $G$-capacity, the relative $G$-capacity, and a dual capacity associated with the operator $\cA$.

First, for $E \subset \rn$, we define the \emph{$G$-capacity} of Sobolev type by
\begin{equation}\label{G-cap}
C_{G}(E):=\inf \int_{\rn} (G(|\varphi|)+G(|\nabla \varphi|)) \dx,
\end{equation}
where the infimum is taken over all nonnegative functions $\varphi\in W^{1,G}(\rn)$ such that $\varphi\geq 1$ in an open set containing $E$. It is known that $C_G$ is a Choquet capacity; see~\cite[Remark~10]{BHH18}. In particular, every Borel set $E\subset\rn$ satisfies
\begin{equation}\label{eq-Choquet}
C_G(E) = \sup {\{ C_G(K):\ K \subset E\text{ is compact} \}}.
\end{equation}
Moreover, $C_G$ is monotone and countably subadditive; see~\cite[Theorem~9]{BHH18}.

Next, we consider the \emph{relative $G$-capacity}, denoted by $\mathrm{cap}_G$. For a compact set $K\subset \Omega$, we define
\begin{equation*}
	\mathrm{cap}_{G}(K,\Omega):=\inf \int_{\Omega}G(|\nabla \varphi|)\dx,
\end{equation*}
where the infimum is taken over all nonnegative functions $\varphi\in C^\infty_c(\Omega)$ such that $\varphi\geq 1$ on $K$. For an open set $U\subset \Omega$ and an arbitrary set $ E\subset \Omega$, we define
\begin{equation*}
	\mathrm{cap}_{G}(U,\Omega):=\sup{\{\mathrm{cap}_{G}(K,\Omega):\ K \subset U \text{ is compact}\}},
\end{equation*}
and 
\begin{equation}\label{eq-G-U}
	\mathrm{cap}_{G}(E,\Omega):=\inf{\{\mathrm{cap}_{G}(U,\Omega):\
	E \subset U \subset \Omega \text{ and } U \text{ is open}\}},
\end{equation}
respectively. Note that \eqref{eq-G-U} is consistent with the definition for compact sets; see~\cite[Proposition~21]{BHH18}. It follows immediately from the definition that $\mathrm{cap}_G$ is monotone increasing with respect to the first argument and monotone decreasing with respect to the second. Moreover, $\mathrm{cap}_G$ is countably subadditive; see \cite[Theorem~24]{BHH18}. We shall also use the following continuity property.

\begin{lemma}[{\cite[Theorem~24]{BHH18}}]\label{lem-cap-limit}
If $E_1 \subset E_2 \subset \cdots \subset \Omega$, then $\mathrm{cap}_G\left(\bigcup_{i=1}^\infty E_i, \Omega\right) = \lim_{i \to \infty} \mathrm{cap}_G(E_i, \Omega).$
\end{lemma}

We also recall the following results for capacity zero sets. We refer to \cite{BHH18,Chl23,CK21remov,HJ22} for further properties of $G$-capacity and relative $G$-capacity in a more general framework.

\begin{lemma}[{\cite[Theorem~27 and Proposition~29]{BHH18}}]\label{lem-cap-zero}
Let $E \subset \Omega$. Then $C_{G}(E)=0$ if and only if $\mathrm{cap}_{G}(E,\Omega)=0$. In this case, we simply say that $E$ has $G$-capacity zero.
\end{lemma}

\begin{lemma}[{\cite[Lemma~11]{BHH18}}]\label{lem-cap-measure}
If $E \subset \rn$, then $C_G(E) \geq c\, |E|_{\mathrm{outer}}$ for some $c>0$, where $|\cdot|_{\mathrm{outer}}$ denotes the Lebesgue outer measure. In particular, every set of $G$-capacity zero has Lebesgue measure zero.
\end{lemma}

Finally, we introduce a dual capacity associated with $\cA$. Let $0\leq\nu\in (W^{1,G}_0(\Omega))'$. We denote by $u_\nu\in W^{1,G}_0(\Omega)$ the unique function satisfying
\begin{equation}\label{eq-u-nu}
\int_\Omega \cA(x,\nabla u_\nu)\cdot\nabla\varphi\dx= \int_\Omega \varphi \,\mathrm{d}\nu \quad\text{for all }\varphi\in W^{1,G}_0(\Omega),
\end{equation}
whose existence and uniqueness follow from the monotone operator theory; see, for instance,~\cite[Corollary~III.1.8]{KS80}. For $E\subset\Omega$, define
\begin{equation*}
\widetilde{\mathrm{cap}}_{\cA}(E,\Omega)=\sup\left\{\nu(\Omega):\ 0\leq\nu\in (W^{1,G}_0(\Omega))',\ \supp\nu\subset E,\ 0\leq u_\nu\leq1\ \text{a.e.\ in }\Omega\right\}.
\end{equation*}
We will prove in Lemma~\ref{lem-cap-dual} that the dual capacity $\widetilde{\mathrm{cap}}_{\cA}$ is comparable to the relative $G$-capacity $\mathrm{cap}_G$.

\subsection{Renormalized solutions}\label{sec2-4}

Let us analyse the notion of renormalized solutions from Definition~\ref{def-ren}. We consider $-\DIV\cA(x,\nabla u)=\mu$ with a nonnegative Radon measure $\mu$ on $\Omega$. Every nonnegative Radon measure $\mu$ admits a unique decomposition $\mu=\mu_0+\mu_s$, where $\mu_0 \ll C_G$ and $\mu_s \perp C_G$. More precisely, $\mu_0(E)=0$ for every Borel set $E$ of $G$-capacity zero, and there exists a Borel set $F$ of $G$-capacity zero such that $\mu_s(\Omega \setminus F)=0$. This is the standard decomposition of a measure; see \cite[Lemma~2.1]{FST91} and, in the generalized Orlicz setting, \cite[Remark~1.3]{Chl23}. This allows us to define a renormalized solution as in Definition~\ref{def-ren}. All terms appearing in \eqref{eq-renormalized-intro} are finite under the assumption \eqref{eq:renormalizedpro-intro}. Indeed, since $h'=0$ outside $[-k, k]$ for some $k>0$, we have $h'(u) \nabla u=h'(u)\nabla T_k(u)$, allowing us to bound the left-hand side of \eqref{eq-renormalized-intro} by
\begin{equation*}
c_2^{\cA} \int_{\supp \varphi} \left( qG(|\nabla T_k(u)|) \|\varphi\|_{\infty} \|h'\|_{\infty} + g(|\nabla u|) \|\nabla \varphi\|_{\infty} \|h\|_{\infty} \right) \dx < \infty,
\end{equation*}
where we have used \eqref{ass-op} and \eqref{eq:Gg}. The right-hand side of \eqref{eq-renormalized-intro} is finite, due to the boundedness of $h$ and $\varphi$ together with the local finiteness of the Radon measures $\mu_0$ and $\mu_s$, since $\supp\varphi \Subset \Omega$.

\begin{remark}\label{rmk-ren}
As established in \cite[Proposition~6.2]{Chl23}, for any nonnegative bounded measure $\mu$, the Dirichlet problem  $-\DIV\cA(x,\nabla u)=\mu$ with zero boundary data (in the sense that $T_k(u) \in W^{1,G}_0(\Omega)$ for every $k>0$) admits a renormalized solution $u$ such that $\mu_{0}(\{u=\infty\})=\mu_s(\{u<\infty\})=0$. As we shall see in Proposition~\ref{prop-mu-s}, $\cA$-superharmonic functions have similar properties.
\end{remark}

\section{Renormalized solutions have superharmonic representatives}\label{sec3}

In this section, we prove that every renormalized solution admits an $\cA$-superharmonic representative. Our approach is inspired by the classical results of~\cite{Kil11super, KM92, KM94}, as well as recent developments in nonlocal problems~\cite{kim25wolff}. Let us present first auxiliary lemmas needed in the proof that any renormalized solution $u$ of \eqref{eq:main} in $\Omega$ is locally essentially bounded below.
\begin{lemma}\label{lem:au-neg}
    If $u$ is a  renormalized solution of \eqref{eq:main} in $\Omega$, then  $u_-$ satisfies
\begin{equation}\label{eq:rebound}
\int_{\Omega} \widetilde{\cA}(x,\nabla u_{-})\cdot \nabla (h(u_{-})\varphi)\dx \leq 0
\end{equation}
for any nonnegative $\varphi\in C^{\infty}_{c}(\Omega)$ and any nonnegative function $h\in W^{1,\infty}(\mathbb{R})$ such that $h'$ has compact support.
\end{lemma}\begin{proof}
 We decompose $u$ into its positive and negative parts as $u=u_{+}-u_{-}$, where $u_+:=\max\{u, 0\}$ and $u_-:=\max\{-u, 0\}$. Define $\widetilde{\cA}(x, \xi)=-\cA(x, -\xi)$. Note that $\widetilde{\cA}$ satisfies the exact same monotonicity, coercivity, and growth conditions \eqref{eq-monotonicity} and \eqref{ass-op}, with the same constants $c_1^{\cA}$ and $c_2^{\cA}$. We choose
\begin{equation*}
h_{\varepsilon}(t):=\tfrac{1}{\varepsilon}\min\{\varepsilon,t_+\}-1,\quad \varepsilon>0,
\end{equation*}
and define $H_\varepsilon(t) = h(-t)h_\varepsilon(t)$. Notice that $H_\varepsilon \in W^{1,\infty}(\mathbb{R})$ and its derivative $H_\varepsilon'$ has compact support. Plugging $H_\varepsilon$ and $\varphi$ into \eqref{eq-renormalized-intro}, we evaluate both sides. For the right-hand side, since $h_\varepsilon(\infty)=0$ and $h_\varepsilon(t) \leq 0$ everywhere, we have
\begin{equation*}
H_{\varepsilon}(\infty) \int_{\Omega}\varphi\dmu_s=0 \quad\text{and}\quad \int_{\Omega}H_{\varepsilon}(u)\varphi\dmu_0\leq 0.
\end{equation*}
For the left-hand side, we observe that
\begin{equation*}
\int_{\Omega} \cA(x,\nabla u ) \cdot \nabla h_{\varepsilon}(u)h(-u)\varphi\dx \geq c_1^{\cA}\int_{\Omega} G(|\nabla u|) h_\varepsilon'(u)h(-u)\varphi\dx \geq 0
\end{equation*}
by \eqref{ass-op}. Combining these bounds, we obtain
\begin{equation*}
\int_{\Omega}\cA(x,\nabla u)\cdot \nabla (h(-u)\varphi)h_{\varepsilon}(u)\dx \leq 0.
\end{equation*}
By the dominated convergence theorem, passing to the limit as $\varepsilon \to 0$ gives \eqref{eq:rebound}. \end{proof}
We stress that Lemma~\ref{lem:au-neg} implies that $u_{-}$ is a nonnegative $\widetilde{\cA}$-subsolution for which a priori integrability requirements are not necessarily fulfilled. However, $\min\{u_{-},k\}\in W^{1,G}_{\rm loc}(\Omega)$ for all $k>0$. In the following series of lemmas, we use Lemma~\ref{lem:au-neg} to show that $u$ is locally essentially bounded below, or equivalently that $u_{-}$ is locally essentially bounded. Once this is established, the assumption $T_{k}(u)\in W_{\rm loc}^{1,G}(\Omega)$ will allow us, in the proof of Theorem~\ref{thm-ren-sup}, to deduce that $\min\{u,k\}\in W^{1,G}_{\rm loc}(\Omega)$ for all $k>0$.

\begin{lemma}\label{lem:re}
Let $u$ be a renormalized solution to \eqref{eq:main} in $B_R=B_R(x_0)$ with a nonnegative Radon measure $\mu$. For $k\in\mathbb R$, define $v:=\frac{(u_--k)_+}{R}$, and fix $d>0$. Then, for any $\tau>1$, there exists a constant $C=C(p, q, \tau, c^\cA_1, c^\cA_2)>0$ such that
\begin{equation*}
\int_{B_{R/2}} \left( \frac{\bar{g}(d+v)}{\bar{g}(d)} \right)^{\frac{1-\tau}{q-1}} \frac{G(R|\nabla v|)}{d+v} \dx \leq C\bar{g}(d) \int_{B_R \cap \{v>0\}} \left( \frac{\bar{g}(d+v)}{\bar{g}(d)} \right)^{\tau} \dx.
\end{equation*}  
\end{lemma}

\begin{proof}
For $\varepsilon \in (0,1)$, we define the function
\begin{equation*}
\varsigma_{\varepsilon}(t) := 1- \min \left\{ \frac{\bar{g}(d+(t-k)_+/R)}{\bar{g}(d)}, \frac{1}{\varepsilon} \right\}^{\frac{1-\tau}{q-1}}.
\end{equation*}
Let $\varphi \in C^\infty_c(B_R)$ be a nonnegative cutoff function such that $\varphi \equiv 1$ on $B_{R/2}$, $0 \leq \varphi \leq 1$, and $|\nabla \varphi| \leq 4/R$. Note that the inequality \eqref{eq:rebound} is valid for bounded nonnegative $W^{1,G}_0(\Omega)$-functions by approximation. Taking $\varsigma_\varepsilon(u_-)\varphi^q$ as a test function, it follows from \eqref{eq:rebound} and the structural conditions \eqref{ass-op} for $\widetilde{\cA}$ that
\begin{align}\label{eq-energy}
\begin{split}
0 \geq& \int_{B_R} \widetilde{\cA}(x,\nabla u_{-})\cdot \nabla (\varsigma_{\varepsilon}(u_{-})\varphi^q)\dx \\ 
\geq &c^{\cA}_1\int_{B_R} G(|\nabla u_-|) \varsigma_{\varepsilon}'(u_-) \varphi^q \dx - qc^{\cA}_2\int_{B_R} g(|\nabla u_-|) \varsigma_{\varepsilon}(u_-) \varphi^{q-1} |\nabla \varphi| \dx \\
&\to \frac{c^\cA_1}{R} \int_{B_R} G(R |\nabla v|) \varsigma'(v) \varphi^q \dx- qc^\cA_2 \int_{B_R} g(R|\nabla v|) \varsigma(v) \varphi^{q-1} |\nabla\varphi| \dx
\end{split}
\end{align}
as $\varepsilon \to 0$, where
\begin{equation*}
\varsigma(v) = 1- \left( \frac{\bar{g}(d+v)}{\bar{g}(d)} \right)^{\frac{1-\tau}{q-1}}.
\end{equation*}
Using \eqref{eq-barg-der}, we have
\begin{equation*}
\varsigma'(v) = \frac{\tau-1}{q-1} \left( \frac{\bar{g}(d+v)}{\bar{g}(d)} \right)^{\frac{1-\tau}{q-1}-1} \frac{\bar{g}'(d+v)}{\bar{g}(d)} \geq (\tau-1)\frac{p-1}{q-1} \left( \frac{\bar{g}(d+v)}{\bar{g}(d)} \right)^{\frac{1-\tau}{q-1}} \frac{1}{d+v},
\end{equation*}
and hence
\begin{equation}\label{eq-1}
\frac{c^\cA_1}{R} \int_{B_R} G(R|\nabla v|) \varsigma'(v) \varphi^q \dx \geq \frac{c}{R} \int_{B_R} \left( \frac{\bar{g}(d+v)}{\bar{g}(d)} \right)^{\frac{1-\tau}{q-1}} \frac{G(R|\nabla v|)}{d+v} \varphi^q \dx
\end{equation}
for some $c=c(p, q, \tau, c^\cA_1)>0$. 
On the other hand, observing that $\varsigma \leq \chi_{\{v>0\}}$ and using \eqref{eq-barg-comp} and Lemma~\ref{lem-barg}(iii) with $\delta>0$, we obtain
\begin{align}\label{eq-2}
\begin{split}
\int_{B_R} g(R|\nabla v|) \varsigma(v) \varphi^{q-1} |\nabla \varphi| \dx
&\leq q \int_{B_R \cap \{v>0\}} \bar{g}(R|\nabla v|) \frac{|\nabla \varphi|}{\varphi} \varphi^q \dx \\
&\leq q \int_{B_R \cap \{v>0\}} \left( \delta G(R|\nabla v|) + \bar{g}\left( \frac{1}{\delta} \frac{|\nabla \varphi|}{\varphi} \right) \frac{|\nabla \varphi|}{\varphi} \right) \varphi^q \dx.
\end{split}
\end{align}
If we take
\begin{equation*}
\delta := \frac{c}{2q^2 c_2^{\cA} R} \left( \frac{\bar{g}(d+v)}{\bar{g}(d)} \right)^{\frac{1-\tau}{q-1}} \frac{1}{d+v} > 0,
\end{equation*}
then combining \eqref{eq-energy}, \eqref{eq-1}, \eqref{eq-2}, and using Lemma~\ref{lem-barg}(i) and (ii) yields
\begin{align*}
\int_{B_R} &\left( \frac{\bar{g}(d+v)}{\bar{g}(d)} \right)^{\frac{1-\tau}{q-1}} \frac{G(R|\nabla v|)}{d+v} \varphi^q \dx \\
&\leq C \int_{B_R \cap \{v>0\}} \bar{g}\left( \left( \frac{\bar{g}(d+v)}{\bar{g}(d)} \right)^{\frac{\tau-1}{q-1}} (d+v) \frac{R|\nabla \varphi|}{\varphi} \right) \frac{R|\nabla \varphi|}{\varphi} \varphi^q \dx \\
&\leq C \int_{B_R \cap \{v>0\}} \bar{g}\left( \left( \frac{\bar{g}(d+v)}{\bar{g}(d)} \right)^{\frac{\tau-1}{q-1}} (d+v) \right) \left( \left( \frac{R|\nabla \varphi|}{\varphi} \right)^{p} + \left( \frac{R|\nabla \varphi|}{\varphi} \right)^{q} \right) \varphi^q \dx \\
&\leq C \int_{B_R \cap \{v>0\}} \bar{g}(d) \left( \frac{\bar{g}(d+v)}{\bar{g}(d)} \right)^{\tau} (|R\nabla \varphi|^p \varphi^{q-p} + |R\nabla \varphi|^q) \dx
\end{align*}
for some $C=C(p, q, \tau, c^\cA_1, c^\cA_2)>0$. We conclude as $\varphi=1$ on $B_{R/2}$, $0\leq \varphi \leq 1$, and $R|\nabla \varphi| \leq 4$.
\end{proof}

\begin{lemma}\label{lem:bound}
Let $u$ be a renormalized solution to \eqref{eq:main} in $B_R=B_R(x_0)$ with a nonnegative Radon measure $\mu$. For $k \in \R$, define $v:=\frac{(u_--k)_+}{R}$, and fix $d>0$. Let $\tilde p\in(1,\min\{p,n\})$, $\tilde p^\ast:=n\tilde p/(n-\tilde p)$, and
\begin{equation*}
1<\tau<\frac{n}{n-\tilde p+\tilde p/q}.
\end{equation*}
Then there exists a constant $C_0>0$ such that
\begin{equation}\label{eq:b1}
\left( \fint_{B_{R/2}} \left(\frac{\bar{g}(v)}{\bar{g}(d)}\right)^{\tau}\dx\right)^{\tilde{p}/\tilde{p}^\ast}\leq  C_0\theta^{\tilde{p}/\tilde{p}^\ast}_{R}+ C_0\fint_{B_{R}}\left(\frac{\bar{g}(v)}{\bar{g}(d)}\right)^{\tau}\dx,
\end{equation}
where $\theta_{R}=|B_{R}\cap \{v>0\}|/|B_{R}|$. Moreover, there exists a constant $C>0$ such that
\begin{equation}\label{eq:b2}
\left(\fint_{B_{R/2}} \bar{g}^{\tau}(v)\dx\right)^{1/\tau} \leq C\theta^{\rho_1}_{R}\left(\fint_{B_{R}}\bar{g}^{\tau}(v)\dx\right)^{1/\tau},
\end{equation}
where $\rho_1=\frac{1}{\tau}(1-\tilde{p}/\tilde{p}^\ast)>0$.
\end{lemma}

\begin{proof}
Set
\begin{equation*}
\tilde{\tau} := \frac{\tilde{p}(q-1)}{q-\tau}\tau.
\end{equation*}
It is easy to check that $\tilde{p}<\tilde{\tau}<\tilde{p}^\ast$. We define
\begin{equation*}
w := \left( \frac{\bar{g}(d+v)}{\bar{g}(d)} \right)^{\tau/\tilde{\tau}} - 1.
\end{equation*}
Then, we have
\begin{align*}
\left( \frac{\bar{g}(v)}{\bar{g}(d)} \right)^\tau
&\leq \chi_{\{0<v<d\}} + \left( \frac{\bar{g}(d+v)}{\bar{g}(d)} \right)^{\tau \tilde{p}^\ast/\tilde{\tau}} \chi_{\{v\geq d\}} \leq \chi_{\{v>0\}} + (w+1)^{\tilde{p}^\ast} \chi_{\{v>0\}} \leq C \chi_{\{v>0\}} + C w^{\tilde{p}^\ast}
\end{align*}
for some $C=C(p, \tilde{p})>0$, which yields
\begin{equation}\label{eq-gvd}
\fint_{B_{R/2}} \left( \frac{\bar{g}(v)}{\bar{g}(d)} \right)^\tau \dx \leq C\theta_R + C \fint_{B_{R/2}} w^{\tilde{p}^\ast} \dx.
\end{equation}

We apply the Sobolev inequality to obtain
\begin{equation}\label{eq-sobolev}
\left( \fint_{B_{R/2}} w^{\tilde{p}^\ast} \dx \right)^{\tilde{p}/\tilde{p}^\ast} \leq C \fint_{B_{R/2}} w^{\tilde{p}} \dx + CR^{\tilde{p}} \fint_{B_{R/2}} |\nabla w|^{\tilde{p}} \dx =: J_1 + J_2.
\end{equation}
Using Lemma~\ref{lem-barg}(i), we have
\begin{equation*}
w^{\tilde{p}} \leq \left( \frac{\bar{g}(2d)}{\bar{g}(d)} \right)^{\tau \tilde{p}/\tilde{\tau}} \chi_{\{0<v<d\}} + \left( \frac{\bar{g}(2v)}{\bar{g}(d)} \right)^{\tau \tilde{p}/\tilde{\tau}} \chi_{\{v \geq d\}} \leq C \chi_{\{v>0\}} + C \left( \frac{\bar{g}(v)}{\bar{g}(d)} \right)^{\tau},
\end{equation*}
so we deduce that
\begin{equation}\label{eq-J1}
J_1 \leq C\theta_R + C \fint_{B_{R/2}} \left( \frac{\bar{g}(v)}{\bar{g}(d)} \right)^\tau \dx.
\end{equation}

For $J_2$, we observe that $\nabla w=0$ a.e.\ on $\{v=0\}$. Thus,  by \eqref{eq-barg-der} we have
\begin{align*}
|\nabla w|^{\tilde{p}}
&\leq \left| \frac{\tau}{\tilde{\tau} R} \left( \frac{\bar{g}(d+v)}{\bar{g}(d)} \right)^{\tau/\tilde{\tau}-1} \frac{\bar{g}'(d+v)}{\bar{g}(d)} \nabla u_- \right|^{\tilde{p}} \chi_{\{v>0\}} \\
&\leq \frac{C}{R^{\tilde{p}}} \left( \frac{\bar{g}(d+v)}{\bar{g}(d)} \right)^{\tau \tilde{p}/\tilde{\tau}} \frac{|\nabla u_-|^{\tilde{p}}}{(v+d)^{\tilde{p}}} \chi_{\{v>0\}}.
\end{align*}
Let $H(t) = G(t^{1/\tilde{p}})$, then Young's inequality in \eqref{eq-Young} and Lemma~\ref{lem-H} show that on $\{v>0\}$,
\begin{equation*}
\frac{G(d+v)}{(d+v)^{\tilde{p}}} |\nabla u_-|^{\tilde{p}} \leq H^\ast \left( \frac{G(d+v)}{(d+v)^{\tilde{p}}} \right) + H(|\nabla u_-|^{\tilde{p}}) \leq CG(d+v) + G(R|\nabla v|).
\end{equation*}
Combining these estimates gives
\begin{equation*}
|\nabla w|^{\tilde{p}} \leq \frac{C}{R^{\tilde{p}}} \left( \frac{\bar{g}(d+v)}{\bar{g}(d)} \right)^{\tau {\tilde{p}}/\tilde{\tau}} \left( 1 + \frac{G(R|\nabla v|)}{G(d+v)} \right) \chi_{\{v>0\}}.
\end{equation*}
Note that
\begin{equation*}
\left( \frac{\bar{g}(d+v)}{\bar{g}(d)} \right)^{\tau \tilde{p}/\tilde{\tau}} \chi_{\{v>0\}} \leq C \chi_{\{v>0\}} + C \left( \frac{\bar{g}(v)}{\bar{g}(d)} \right)^{\tau}.
\end{equation*}
Using Lemma~\ref{lem:re} and recalling $\tau\tilde{p}/\tilde{\tau}-1=(1-\tau)/(q-1)$ imply that
\begin{align}\label{eq-J2}
\begin{split}
J_2
&\leq C \fint_{B_{R/2}} \left( \frac{\bar{g}(d+v)}{\bar{g}(d)} \right)^{\tau \tilde{p}/\tilde{\tau}} \left( 1 + \frac{G(R|\nabla v|)}{G(d+v)} \right) \chi_{\{v>0\}} \dx\\
&\leq C \theta_R + C \fint_{B_{R/2}} \left( \frac{\bar{g}(v)}{\bar{g}(d)} \right)^\tau \dx+ C \fint_{B_R} \left( \frac{\bar{g}(d+v)}{\bar{g}(d)} \right)^{\tau} \chi_{\{v>0\}} \dx \\
&\leq C \theta_R + C \fint_{B_R} \left( \frac{\bar{g}(v)}{\bar{g}(d)} \right)^\tau \dx.
\end{split}
\end{align}
Therefore, \eqref{eq:b1} follows from \eqref{eq-gvd}, \eqref{eq-sobolev}, \eqref{eq-J1}, \eqref{eq-J2}, and the fact that $\theta_R \leq \theta_R^{\tilde{p}/\tilde{p}^\ast}$.

Next, let us deduce \eqref{eq:b2} from \eqref{eq:b1}. We set $M:=1+C_0$ and 
\begin{equation*}
	A_R:=\left(\fint_{B_R} \bar{g}^{\tau}(v)\dx\right)^{1/\tau}.
\end{equation*}
If $A_{R/2}=0$ or $\theta_R=0$, then there is nothing to prove. Hence, we assume that $A_{R/2}>0$ and $\theta_R>0$. If we choose $d$ so that
\begin{equation*}
	\frac{1}{\bar{g}(d)}=\frac{M^{\tilde{p}^\ast/(\tau \tilde{p})}\theta^{1/\tau}_R}{A_{R/2}},
\end{equation*}
which implies that 
\[\theta^{\tilde{p}/\tilde{p}^\ast}_R=\left(\frac{A_{R/2}}{\bar{g}(d)}\right)^{\tau\tilde{p}/\tilde{p}^\ast}-C_0 \theta_R^{\tilde{p}/\tilde{p}^\ast},\]
then it follows from \eqref{eq:b1} that
\begin{equation*}
\begin{split}
\theta^{\tilde{p}/\tilde{p}^\ast}_R
&=\left(\fint_{B_{R/2}}\left(\frac{\bar{g}(v)}{\bar{g}(d)}\right)^\tau\dx\right)^{\tilde{p}/\tilde{p}^\ast}-C_0 \theta_R^{\tilde{p}/\tilde{p}^\ast}\leq C_0 \left(\frac{A_R}{\bar{g}(d)}\right)^{\tau}= C_0M^{\tilde{p}^\ast/\tilde{p}}\theta_R \left(\frac{A_R}{A_{R/2}}\right)^\tau.
\end{split}
\end{equation*}
This implies that $A_{R/2} \leq C\theta^{\rho_1}_{R}A_R$ for some $C>0$, which proves \eqref{eq:b2}.
\end{proof}

\begin{lemma}\label{lem-int-rem}
If $u$ is a renormalized solution to \eqref{eq:main} in $\Omega$ with a nonnegative Radon measure $\mu$, then $g(|u|)\in L^m_{\rm loc}(\Omega)$ for any $m<\frac{n}{n-1}$.
\end{lemma}

\begin{proof}
Let $U \Subset \Omega$ be open. We choose a cutoff function $\varphi \in C^\infty_c(\Omega)$ such that $\varphi=1$ on $U$ and $0\leq \varphi \leq 1$ in $\Omega$. Let $h(t)=T_k(t)$ be the truncation given in \eqref{eq-truncation}. Clearly, $h \in W^{1,\infty}(\mathbb{R})$ and $h'(t) = \chi_{(-k, k)}(t)$ has compact support. It follows from \eqref{eq-renormalized-intro} that
\begin{equation*}
\int_{\Omega}\cA(x,\nabla u)\cdot \nabla u \chi_{\{|u|<k\}} \varphi\dx = \int_{\Omega}T_k(u)\varphi\dmu_0+ k\int_{\Omega}\varphi\dmu_s - \int_{\Omega}\cA(x,\nabla u)\cdot \nabla \varphi\, T_k(u)\dx.
\end{equation*}
Since $\nabla u=\nabla T_k(u)$ a.e.\ on $\chi_{\{|u|<k\}}$, by  \eqref{ass-op} we obtain 
\begin{equation*}
c_1^{\cA} \int_{U} G(|\nabla T_k(u)|) \dx \leq \int_{\Omega}\cA(x,\nabla u)\cdot \nabla u \chi_{\{|u|<k\}} \varphi\dx.
\end{equation*}
On the other hand, we have
\begin{equation*}
\left| \int_{\Omega}T_k(u)\varphi\dmu_0+ k\int_{\Omega}\varphi\dmu_s \right| \leq k\mu(\supp \varphi)
\end{equation*}
and by Cauchy--Schwartz inequality and  \eqref{ass-op} we get
\begin{equation*}
\left| \int_{\Omega}\cA(x,\nabla u)\cdot \nabla \varphi\, T_k(u)\dx \right| \leq c_2^{\cA} \|\nabla \varphi\|_{\infty} k\int_{\supp \varphi} g(|\nabla u|) \dx.
\end{equation*}
The last integral is finite by the assumption \eqref{eq:renormalizedpro-intro}. Consequently,
\begin{equation}\label{eq-trunc-energy}
\int_{U} G(|\nabla T_k(u)|) \dx \leq Ck \quad\text{for all }k>0.
\end{equation}

A standard argument based on the Sobolev embedding, conducted as in \cite[Theorem~3.2]{CM17} or \cite[Proposition 6.3]{CGZG19}, yields Marcinkiewicz regularity $g(|u|)\in L^{\frac{n}{n-1},\infty}(U)$ from \eqref{eq-trunc-energy}. In particular, this implies that $g(|u|)\in L^m(U)$ for all $m<\frac{n}{n-1}$.
\end{proof}

We are in a position to infer the boundedness from below of a renormalized solution.

\begin{proposition}\label{prop:bound2}
Let $u$ be a renormalized solution to \eqref{eq:main} in $\Omega$ with a nonnegative Radon measure $\mu$. Then $u$ is locally essentially bounded below.
\end{proposition}

\begin{proof}
Let $K$ be a compact subset of $\Omega$. For $x_0 \in K$ and $R<\dist(K, \partial \Omega)$, we define $R_i:=2^{-i-1}R$ and $B^i:=B_{R_i}(x_0)$, $i=0,1,\dots$. Fix $\tilde{p} \in (1, \min\{p, n\})$ and assume
\begin{equation*}
1<\tau<\min\left\{\frac{n}{n-\tilde{p}+\tilde{p}/q},\frac{n}{n-1}\right\}.
\end{equation*}
Let $\varepsilon\in (0,1)$ and $c_0>0$ be constants to be determined later. Set $l_0=0$ and for $i=0,1,\dots$
recursively define
\begin{equation}\label{eq-l-j}
	l_{i+1}:=l_i+ R_i\bar{g}^{-1}\left(\frac{2^{c_0i}}{\varepsilon}\left(\fint_{B^i}\bar{g}^\tau\left(\frac{(u_{-}-l_i)_+}{R_i}\right)\dx\right)^{1/\tau}\right).
\end{equation}
We also define $\theta_{i}:=|B^{i-1}\cap \{u_{-}>l_i\}|/|B^{i-1}|$ for $i=1,2,\dots\,$. Note that due to~\eqref{eq-l-j} it holds
\begin{equation}\label{eq:b3}
	\begin{split}
		\theta_{i}&\leq \frac{1}{|B^{i-1}|}\int_{B^{i-1}\cap\{u_{-}>l_i\}}\left(\frac{\bar{g}((u_{-}-l_{i-1})_+/R_{i-1})}{\bar{g}((l_i-l_{i-1})_+/R_{i-1})}\right)^\tau\dx\\
		&\leq \frac{1}{|B^{i-1}|}\int_{B^{i-1}}\left(\frac{\bar{g}((u_--l_{i-1})_+/R_{i-1})}{\bar{g}((l_i-l_{i-1})_+/R_{i-1})}\right)^\tau\dx=\frac{\varepsilon^\tau}{2^{c_0(i-1)\tau}}.
	\end{split}
\end{equation}
Also, using Lemma~\ref{lem-barg}(i)
\begin{equation*}
\begin{split}
		\bar{g}\left(\frac{l_{i+1}-l_i}{R_i}\right)&= \frac{2^{c_0i}}{\varepsilon}\left(\fint_{B^i}\bar{g}^\tau\left(\frac{(u_{-}-l_i)_+}{R_i}\right)\dx\right)^{1/\tau}\\
	&\leq \frac{2^{q-1}2^{c_0i}}{\varepsilon}\left(\fint_{B^i}\bar{g}^\tau\left(\frac{(u_{-}-l_{i})_+}{R_{i-1}}\right)\dx\right)^{1/\tau}.
\end{split}
\end{equation*}
Applying \eqref{eq:b2} with $v=\frac{(u_{-}-l_{i})_+}{R_{i-1}}$ in $B^{i-1}$ for $i\geq 1$, and then using $(u_--l_i)_+ \leq (u_--l_{i-1})_+$, Lemma~\ref{lem-barg}(ii), and \eqref{eq:b3}, we obtain
\begin{equation*}
	\begin{split}
		\bar{g}\left(\frac{l_{i+1}-l_i}{R_i}\right)
		&\leq \frac{2^{q-1}2^{c_0i}}{\varepsilon}\theta^{\rho_1}_{i}\left(\fint_{B^{i-1}}\bar{g}^\tau\left(\frac{(u_{-}-l_{i})_+}{R_{i-1}}\right)\dx\right)^{1/\tau}\leq \frac{2^{q-1}2^{c_0i}}{\varepsilon}\theta^{\rho_1}_{i}\frac{\varepsilon}{2^{c_0(i-1)}}\bar{g}\left(\frac{l_i-l_{i-1}}{R_{i-1}}\right)\\
		& \leq C\frac{\varepsilon^{\tau\rho_1}}{2^{c_0\tau\rho_1 i}}\bar{g}\left(\frac{l_i-l_{i-1}}{R_{i}}\right),
	\end{split}
\end{equation*}
that is,
\begin{equation}\label{eq-sec3-barg}
	\bar{g}\left(\frac{l_{i+1}-l_i}{R_i}\right)\leq C \frac{\varepsilon^{\tau\rho_1}}{2^{c_0\tau\rho_1 i}}\bar{g}\left(\frac{l_i-l_{i-1}}{R_i}\right)\leq C\varepsilon^{\tau\rho_1}\bar{g}\left(\frac{l_i-l_{i-1}}{R_i}\right),
\end{equation}
where $C$ is independent of $i$. Applying $\bar{g}^{-1}$ to both sides of \eqref{eq-sec3-barg} and using Lemma~\ref{lem-barg}(i) and (ii), we obtain
\begin{equation*}
	l_{i+1}-l_i \leq C\varepsilon^{\frac{\rho_1 \tau}{q-1}}(l_i-l_{i-1}),
\end{equation*}
which implies that for $j\geq 2$,
\begin{equation*}
	l_j=l_1+\sum^{j-1}\limits_{i=1}(l_{i+1}-l_i)\leq l_1 +C\varepsilon^{\frac{\rho_1 \tau}{q-1}}l_{j-1}.
\end{equation*}
Choosing $\varepsilon\in (0,1)$ sufficiently small so that $C\varepsilon^{\frac{\rho_1 \tau}{q-1}} <\frac{1}{2}$, we conclude that
\begin{align*}
l_j
\leq 2l_1&\leq CR \bar{g}^{-1}\left(\frac{1}{\varepsilon}\left(\fint_{B_{R/2}}\bar{g}^\tau\left(\frac{2u_{-}}{R}\right)\dx\right)^{1/\tau}\right) \\
&\leq CR \bar{g}^{-1}\left(\frac{1}{\varepsilon}\sup_{x \in K} \left(\fint_{B_{R/2}(x)}\bar{g}^\tau\left(\frac{2u_{-}}{R}\right)\dx\right)^{1/\tau}\right) =: M.
\end{align*}
Since $\tau<\frac{n}{n-1}$, Lemma~\ref{lem-int-rem} shows that $M$ is finite. Note that $M$ is independent of $x_0 \in K$.

We now prove that $u_- \leq M$ a.e.\ on $K$. To this end, we define for $\delta>0$
\begin{equation*}
E_\delta := K \cap \{u_- > M +\delta\}.
\end{equation*}
Assume to the contrary that $|E_\delta|>0$ for some $\delta>0$. Choose a Lebesgue density point $x_0 \in E_\delta$. Applying the above construction with this center $x_0$, we have $l_i \leq M < M+\delta$ for all $i$, and hence
\begin{equation*}
E_\delta \cap B^{i-1} \subset B^{i-1} \cap \{u_- > l_i\}.
\end{equation*}
Therefore, by using \eqref{eq:b3} we obtain
\begin{equation*}
\frac{|E_\delta \cap B^{i-1}|}{|B^{i-1}|} \leq \theta_i \leq \frac{\varepsilon^\tau}{2^{c_0(i-1)\tau}} \to 0 \quad\text{as }i \to \infty.
\end{equation*}
This contradicts to the fact that $x_0$ is a Lebesgue density point of $E_\delta$. Thus, $|E_\delta|=0$ for every $\delta>0$. Letting $\delta \to 0$ gives $u_- \leq M$ a.e.\ on $K$.
\end{proof}

Now, we are ready to prove that every renormalized solution has an $\cA$-superharmonic representative.

\begin{proof}[Proof of Theorem~\ref{thm-ren-sup}]
By Proposition~\ref{prop:bound2}, $u$ is locally essentially bounded below. Hence, for each  $k>0$, the truncation $u_k:=\min\{u,k\}$ belongs to $W^{1,G}_{\rm loc}(\Omega)$. We first show that $u_k$ is an $\cA$-supersolution of \eqref{eq-sec2-new}.

Let $\varphi\in C^{\infty}_{c}(\Omega)$ be nonnegative. For $\varepsilon>0$ and $k>0$, define
\begin{equation*}
    h_{k,\varepsilon}(t):=\tfrac{1}{\varepsilon}\min\{(k+\varepsilon-t)_{+},\varepsilon\}.
\end{equation*}
Since $h'_{k,\varepsilon}(t)\leq 0$, it follows that
\begin{equation*}
    \int_{\Omega}\cA(x,\nabla u) \cdot \nabla u \, h'_{k,\varepsilon}(u)\varphi\dx\leq 0.
\end{equation*}
Moreover, the nonnegativity of $\mu$ and $\varphi$ implies
\begin{equation*}
    \int_{\Omega}h_{k,\varepsilon}(u)\varphi\dmu_0+h_{k,\varepsilon}(\infty)\int_{\Omega}\varphi\dmu_s\geq 0.
\end{equation*}
Testing the equation \eqref{eq-renormalized-intro} with $h_{k, \varepsilon}(u)$ and combining the above inequalities, we deduce that
\begin{equation*}
\int_{\Omega}\cA(x,\nabla u)\cdot \nabla \varphi \, h_{k,\varepsilon}(u) \dx \geq 0.
\end{equation*}
As $\varepsilon\to 0$, we observe that $h_{k, \varepsilon}(u) \to \chi_{\{u\leq k\}}$ everywhere. Applying the dominated convergence theorem, we conclude that
\begin{equation*}
    \int_{\Omega}\cA(x,\nabla(T_k u))\cdot \nabla \varphi\dx \geq 0,
\end{equation*}
By Proposition~\ref{prop:bound2}, $u$ is locally essentially bounded below, which ensures that $T_ku\in W_{\rm loc}^{1,G}(\Omega)$. This shows that $T_ku$ is an $\cA$-supersolution of \eqref{eq-sec2-new}.

It follows from Lemma~\ref{lem-super-liminf} that $T_ku$ has a lower semicontinuous representative $\tilde{u}_k$ such that $\tilde{u}_k(x)=\essliminf_{y \to x} \tilde{u}_k(y)$ for every $x\in \Omega$. Since $\{T_ku\}$ is an increasing sequence, so is $\{\tilde{u}_k\}$. Since $\tilde{u}_k=T_ku$ a.e., we have $\tilde u:=\lim_{k\to\infty}\tilde u_k=u$ a.e. In particular, by Lemma~\ref{lem-int-rem}, the limit is not identically infinite. Thus, Proposition~\ref{prop-conv-super} shows that $\tilde{u} \in \cS_{\cA}(\Omega)$. Therefore, $\tilde{u}$ serves as the desired representative of $u$.
\end{proof}

By virtue of Theorem~\ref{thm-ren-sup}, the solvability of the Dirichlet problem for renormalized solutions naturally translates into the solvability of the problem for $\cA$-superharmonic functions as a corollary. For the Dirichlet problem $\mu$ is required to be bounded.

\begin{corollary}\label{cor-DP}
Let $\mu$ be a nonnegative bounded Radon measure on $\Omega$. Then there exist a nonnegative $\cA$-superharmonic function $u$ in $\Omega$ and a nonnegative renormalized solution $v$ to \eqref{eq:main} such that $u=v$ a.e.\ in $\Omega$ and $T_k(u) \in W^{1, G}_0(\Omega)$ for all $k>0$.
\end{corollary}

\begin{proof}
By~\cite[Proposition~6.2]{Chl23}, there exists a renormalized solution $v$ to \eqref{eq:main} such that $T_k(v) \in W^{1,G}_0(\Omega)$ for all $k>0$. Since $\mu \geq 0$, $v$ may be chosen nonnegative. Indeed, choosing nonnegative bounded approximations of $\mu$, the corresponding weak solutions are nonnegative by the weak maximum principle, and the a.e.\ limit remains nonnegative.

Applying Theorem~\ref{thm-ren-sup} to $v$, we obtain an $\cA$-superharmonic representative $u$ with $u=v$ a.e. In particular, $u$ is nonnegative. Finally, $T_k(u)=T_k(v)$ a.e., so $T_k(u)\in W^{1,G}_0(\Omega)$ for every $k>0$.
\end{proof}

\section{Nonlinear potential theory and obstacle problems}\label{sec4}

Prior to proving that any $\cA$-supeharmonic function is a renormalized solution (Theorem~\ref{thm-sup-ren}), we develop a scale-free nonlinear potential theory. Main goals of this section are the $G$-quasicontinuity of $\cA$-superharmonic functions and the fact that $\cA$-polar sets have zero $G$-capacity. To this end, we develop the necessary potential-theoretic tools, including obstacle problems, balayage, and capacitary estimates.

\subsection{Obstacle problems and boundary regularity}\label{sec4-1}

We introduce the set
\[
\mathcal{K}_{\psi, \vartheta}(\Omega) := \left\{ v \in W^{1,G}(\Omega) : v \ge \psi \text{ a.e.\ in } \Omega \text{ and } v - \vartheta \in W^{1,G}_0(\Omega) \right\},
\]
where $\psi: \Omega \to \overline{\mathbb{R}}$ is an obstacle and $\vartheta \in W^{1,G}(\Omega)$ is a boundary datum. A function $u \in \mathcal{K}_{\psi,\vartheta}(\Omega)$ is called a solution to the \emph{$\mathcal{K}_{\psi, \vartheta}(\Omega)$-obstacle problem} if it satisfies the variational inequality
\begin{equation}\label{eq-obs}
\int_{\Omega} \mathcal{A}(x,\nabla u) \cdot \nabla (v-u) \dx \ge 0 \qquad \text{for all } v \in \mathcal{K}_{\psi,\vartheta}(\Omega).
\end{equation}
The following theorem summarizes the fundamental results on the existence, uniqueness, and interior regularity of the $\mathcal{K}_{\psi, \vartheta}(\Omega)$-obstacle problem.

\begin{proposition}
    \cite[Theorem 2]{CK21remov}\label{prop-known}
Let $\psi, \vartheta \in W^{1,G}(\Omega)$ be such that $\mathcal{K}_{\psi, \vartheta}(\Omega)\neq \emptyset$. Then, the $\mathcal{K}_{\psi, \vartheta}(\Omega)$-obstacle problem is solvable, and the solution is unique up to a set of measure zero. Moreover, the solution is an $\cA$-supersolution of \eqref{eq-sec2-new}. If, in addition, $\psi \in C(\Omega)$ ($\psi \in C^\alpha(\Omega)$ for some $\alpha \in (0,1)$, resp.), then the solution $u$ is continuous (H\"older continuous with exponent $\alpha$, resp.) in $\Omega$ and is an $\cA$-solution to \eqref{eq-sec2-new} in the open set $\{x \in \Omega: u(x)>\psi(x)\}$.
\end{proposition}

The main goal of this section is to prove H\"older regularity of the $\mathcal{K}_{\psi, \vartheta}(\Omega)$-obstacle problem up to the boundary. For this purpose, we develop the De Giorgi--Nash--Moser theory. We begin with the local boundedness up to the boundary.

\begin{lemma}\label{lem-loc-bdd-obs}
Assume $B_R=B_R(x_0)$ satisfy $\Omega \cap B_R \neq \emptyset$. Let $\vartheta \in W^{1,G}(\Omega) \cap C(\overline{\Omega} \cap B_R)$. Let $u$ be a solution to the $\mathcal{K}_{\psi, \vartheta}(\Omega)$-obstacle problem.
\begin{enumerate}[\textup{(}i\textup{)}]
\item
If
$M_{\psi,\vartheta}:= \max \left\{ \esssup_{\Omega \cap B_R} \psi, \esssup_{\partial \Omega \cap B_R} \vartheta \right\} < \infty$,
then for any $k \geq M_{\psi,\vartheta}$ it holds
\begin{equation*}
\esssup_{\Omega \cap B_{R/2}} {(u-k)_+} \leq CR G^{-1} \left( \frac{1}{|B_R|} \int_{\Omega \cap B_R} G\left( \frac{(u-k)_+}{R} \right) \dx \right).
\end{equation*}
\item
If
$m_{\psi,\vartheta}:= \essinf_{\partial \Omega \cap B_R} \vartheta > -\infty$,
then for any $k \leq m_{\psi,\vartheta}$ it holds
\begin{equation*}
\esssup_{\Omega \cap B_{R/2}} {(u-k)_-} \leq CR G^{-1} \left( \frac{1}{|B_R|} \int_{\Omega \cap B_R} G\left( \frac{(u-k)_-}{R} \right) \dx \right).
\end{equation*}
\end{enumerate}
\end{lemma}

\begin{proof}
{\it (i)} Let $R/2 \leq \rho < r \leq R$. Let $\eta \in C^\infty_c(B_r)$ satisfy $0 \leq \eta \leq 1$ in $B_r$ and $\eta=1$ on $B_\rho$. Then, for any $k \geq M$, the function $v:=u-(u-k)_+\eta^q \in \mathcal{K}_{\psi, \vartheta}(\Omega)$ is admissible for the variational inequality \eqref{eq-obs}. Thus, a standard argument shows that $u$ satisfies a Caccioppoli type estimate
\begin{equation*}
\int_{\Omega \cap B_\rho} G(|\nabla (u-k)_+|) \dx \leq C \left( \frac{R}{r-\rho} \right)^q \int_{\Omega \cap B_r} G\left( \frac{(u-k)_+}{R} \right) \dx.
\end{equation*}
Let $w$ be the zero extension of $(u-k)_+$ to $B_R$. Then the above estimate reads as
\begin{equation}\label{eq-Caccioppoli}
\int_{B_\rho} G(|\nabla w|) \dx \leq C \left( \frac{R}{r-\rho} \right)^q \int_{B_r} G\left( \frac{w}{R} \right) \dx,
\end{equation}
which is sufficient to deduce
\begin{equation*}
\esssup_{B_{R/2}} w \leq CR G^{-1} \left( \fint_{B_R} G\left( \frac{w}{R}\right) \dx \right),
\end{equation*}
by using a standard iteration argument; see for instance \cite[Proposition~5.5]{HHL21}.

{\it (ii)} The proof follows the same lines upon the choice of an admissible function $v:=u+(u-k)_-\eta^q \in \mathcal{K}_{\psi, \vartheta}(\Omega)$ for any $k \leq m$.
\end{proof}

We now state the main theorem of this section, which establishes the global H\"older regularity of solutions to the $\mathcal{K}_{\psi, \vartheta}(\Omega)$-obstacle problem.

\begin{theorem}\label{thm-bdry-reg}
Assume that $\rn \setminus \Omega$ satisfies the measure density condition, i.e.\ there exists $R_0>0$ such that
\begin{equation}\label{eq-MDC}
\inf_{x_0 \in \partial \Omega} \inf_{R \in (0, R_0)} \frac{|B_R(x_0) \setminus \Omega|}{|B_R(x_0)|} >0.
\end{equation}
Suppose that either $\psi$ is H\"older continuous on $\overline{\Omega}$ or $\psi \equiv -\infty$, that $\vartheta \in \mathcal{K}_{\psi, \vartheta}(\Omega)$ is H\"older continuous on $\overline{\Omega}$, and that $\psi \leq \vartheta$ on $\partial \Omega$. Then the solution to the $\mathcal{K}_{\psi, \vartheta}(\Omega)$-obstacle problem is H\"older continuous on $\overline{\Omega}$.
\end{theorem}

To establish Theorem~\ref{thm-bdry-reg}, it suffices to prove its  following  localized version near the boundary.

\begin{lemma}
Assume that $\rn \setminus \Omega$ satisfies the measure density condition \eqref{eq-MDC} with $R_0>0$. Let $x_0 \in \partial\Omega$ and $0<R \leq R_0$. Suppose that either $\psi$ is H\"older continuous on $\overline{\Omega} \cap B_R(x_0)$ or $\psi \equiv -\infty$, that $\vartheta \in \mathcal{K}_{\psi, \vartheta}(\Omega)$ is H\"older continuous on $\overline{\Omega} \cap B_R(x_0)$, and that $\psi \leq \vartheta$ on $\partial \Omega \cap B_R(x_0)$. Then the solution to the $\mathcal{K}_{\psi, \vartheta}(\Omega)$-obstacle problem is H\"older continuous on $\overline{\Omega} \cap B_R(x_0)$.
\end{lemma}

\begin{proof}
Let $0<r<R/2$ and write $B_r=B_r(x_0)$. Since $\vartheta$ is H\"older continuous on the closed set $\overline{\Omega} \cap \overline{B}_r$, we extend it to a H\"older continuous function on $\overline B_r$, still denoted by $\vartheta$, by the McShane--Whitney extension theorem. Let $u$ be the solution to the $\mathcal{K}_{\psi, \vartheta}(\Omega)$-obstacle problem. We denote by $\bar{u}$ the function defined by $\bar{u}=u$ in $\Omega \cap B_r$ and $\bar{u}=\vartheta$ in $B_r \setminus \Omega$. Note, however, that $\bar{u}$ does not necessarily belong to $W^{1,G}(B_r)$.

For $0<\rho \leq r$, we set
\begin{equation*}
M_\rho := \esssup_{B_\rho} \bar{u}, \quad m_\rho := \essinf_{B_\rho} \bar{u},
\end{equation*}
which are finite by Lemma~\ref{lem-loc-bdd-obs}, and define an increasing sequence of levels $k_j$ as follows:
\begin{equation*}
k_0: = \max\left\{ \sup_{\Omega \cap B_r} \psi,\ \sup_{B_r} \vartheta \right\} \quad\text{and}\quad k_j := M_r - \frac{M_r-k_0}{2^j} \quad\text{for }j=1, 2, \dots\,.
\end{equation*}
Note that if $\psi \equiv -\infty$, then $\sup_{\Omega \cap B_r} \psi = -\infty$ and hence $k_0=\sup_{B_r}\vartheta$. Also, for $j=0,1,2,\dots$, we define $w_j$ by the zero extension of $(u-k_j)_+$ to $B_r$, which belongs to $W^{1,G}(B_r)$, and
\begin{equation*}
A_j := B_r \cap \{w_j>0\}, \quad \widetilde{A}_j := B_{r/2} \cap \{w_j>0\},\quad D_j := \widetilde{A}_j \setminus \widetilde{A}_{j+1}.
\end{equation*}
We notice that $\widetilde{A}_{j+1} = B_{r/2} \cap \{w_j > k_{j+1} - k_j\}$ and $D_j=B_{r/2} \cap \{0 <w_j \leq k_{j+1} - k_j\}$.

Applying De Giorgi's isoperimetric lemma (see for instance \cite[Chapter~I, Lemma~2.2]{Dib93}), we obtain
\begin{equation}\label{eq-isoperimetric}
(k_{j+1} - k_j) |\widetilde{A}_{j+1}| |B_{r/2} \setminus \widetilde{A}_{j}| \le C r^{n+1} \int_{D_j} |\nabla w_j| \dx.
\end{equation}
Since $B_{r/2} \setminus \Omega \subset B_{r/2} \setminus \widetilde{A}_{j}$, it follows from the measure density condition that
\begin{equation}\label{eq-MDC-A}
|B_{r/2} \setminus \widetilde{A}_{j}| \geq |B_{r/2} \setminus \Omega| \geq \delta_0 |B_{r/2}|
\end{equation}
for some $\delta_0>0$. 
Combining \eqref{eq-isoperimetric} and \eqref{eq-MDC-A} yields
\begin{equation*}
|\widetilde{A}_{j+1}| \leq C\frac{r}{k_{j+1} - k_j} \int_{D_j} |\nabla w_j| \dx.
\end{equation*}

By using Jensen's inequality, the Caccioppoli estimate \eqref{eq-Caccioppoli}, and Lemma~\ref{lem-barg}, we have
\begin{align*}
\int_{D_j} |\nabla w_j| \dx
&\leq G^{-1} \left( \fint_{D_j} G(|\nabla w_j|) \dx \right) |D_j| \leq G^{-1} \left( \frac{C}{|D_j|} \int_{A_j} G\left( \frac{w_j}{r} \right) \dx \right) |D_j| \\
&\leq C G^{-1} \left( \frac{|A_j|}{|D_j|} G\left( \frac{M_r-k_j}{r} \right) \right) |D_j| \leq C \frac{M_r-k_j}{r} \left( \frac{|A_j|}{|D_j|} \right)^{\frac 1p} |D_j|.
\end{align*}
Since $|A_j| \leq |B_r|=|B_1|r^n$, it follows from $k_{j+1}-k_j=2^{-j-1}(M_r-k_0)$ and $M_r-k_j=2^{-j}(M_r-k_0)$ that
\begin{equation*}
|\widetilde{A}_{j+1}| \leq C r^{\frac np} |D_j|^{1-\frac 1p}.
\end{equation*}
Raising both sides to the power of $p/(p-1)$ and noting that $|D_j|=|\widetilde{A}_j| - |\widetilde{A}_{j+1}|$, we arrive at
\begin{equation*}
|\widetilde{A}_{j+1}|^{\frac{p}{p-1}} \leq C r^{\frac{n}{p-1}} \left( |\widetilde{A}_j| - |\widetilde{A}_{j+1}| \right).
\end{equation*}
We sum this inequality from $j=0$ to $N-1$ to obtain
\begin{equation*}
N |\widetilde{A}_N|^{\frac{p}{p-1}} \leq \sum_{j=0}^{N-1} |\widetilde{A}_{j+1}|^{\frac{p}{p-1}} \leq Cr^{\frac{n}{p-1}} \sum_{j=0}^{N-1} \left( |\widetilde{A}_j| - |\widetilde{A}_{j+1}| \right) \leq C |B_{r/2}|^{\frac{p}{p-1}},
\end{equation*}
in other words,
\begin{equation}\label{eq-ratio}
\frac{|\widetilde{A}_N|}{|B_{r/2}|} \leq \frac{C}{N^{1-1/p}}.
\end{equation}

Now, Lemma~\ref{lem-loc-bdd-obs}(i), applied to the level $k_N$, together with Lemma~\ref{lem-barg} and \eqref{eq-ratio}, shows that
\begin{align*}
\esssup_{\Omega \cap B_{r/4}} {(u-k_N)_+}
&\leq Cr G^{-1} \left( \fint_{B_{r/2}} G\left( \frac{w_N}{r} \right) \dx \right) \leq Cr G^{-1} \left( \frac{|\widetilde{A}_N|}{|B_{r/2}|} G\left( \frac{M_r-k_N}{r} \right) \right) \\
&\leq C \left( \frac{1}{N^{1-1/p}} \right)^{1/q} (M_r-k_N).
\end{align*}
We choose $N$ sufficiently large so that $2C \leq N^{\frac{p-1}{pq}}$. This yields $\esssup_{\Omega \cap B_{r/4}} {(u-k_N)_+} \leq (M_r-k_N)/2$. Moreover, since $\bar{u} =\vartheta \leq k_0 \leq k_N$ in $B_{r/4} \setminus \Omega$, for  $\gamma:=2^{-N-1} \in (0,1)$ we obtain
\begin{equation*}
M_{r/4} \leq (1-\gamma)M_r + \gamma k_0\,.
\end{equation*}

A symmetric argument utilizing Lemma~\ref{lem-loc-bdd-obs}(ii) yields the corresponding lower bound decay $m_{r/4} \geq (1-\gamma) m_r + \gamma \inf_{B_r} \vartheta$. Subtracting these two inequalities produces the standard oscillation estimate
\begin{equation*}
\essosc_{B_{r/4}} \bar{u} \le (1-\gamma) \essosc_{B_{r}} \bar{u} + \gamma \left( \max\left\{ \sup_{\Omega \cap B_r} \psi,\ \sup_{B_r} \vartheta \right\} - \inf_{B_r} \vartheta \right).
\end{equation*}
If $\psi \equiv -\infty$, then
\begin{equation*}
\max\left\{ \sup_{\Omega \cap B_r} \psi,\ \sup_{B_r} \vartheta \right\} - \inf_{B_r} \vartheta = \osc_{B_r} \vartheta.
\end{equation*}
Otherwise, since $\max\{a, b\} = b+(a-b)_+$ for all $a, b \in \mathbb{R}$ and $\psi(x_0) \leq \vartheta(x_0)$, we have
\begin{align*}
\max\left\{ \sup_{\Omega \cap B_r} \psi,\ \sup_{B_r} \vartheta \right\} - \inf_{B_r} \vartheta
&\leq \osc_{B_r} \vartheta + \left( \sup_{\Omega \cap B_r} \psi - \sup_{B_r} \vartheta \right)_+ \leq \osc_{B_r} \vartheta + \left( \psi(x_0) + \osc_{\Omega \cap B_r} \psi - \vartheta(x_0) \right)_+ \\
&\leq \osc_{B_r} \vartheta + \osc_{\Omega \cap B_r} \psi,
\end{align*}
and hence
\begin{equation*}
\essosc_{B_{r/4}} \bar{u} \le (1-\gamma) \essosc_{B_{r}} \bar{u} + \gamma \left( \osc_{B_r} \vartheta + \osc_{\Omega \cap B_r} \psi \right).
\end{equation*}
A standard iteration argument finishes the proof of H\"older continuity at $x_0$. Finally, global H\"older regularity in $\overline{\Omega} \cap B_R$ follows from H\"older continuity at boundary points and the interior H\"older regularity in Proposition~\ref{prop-known}.
\end{proof}

\subsection{Auxiliary facts for bounded supersolutions}\label{sec4-2}

Before developing the balayage theory, we record two auxiliary facts for bounded $\cA$-supersolutions. Let us introduce the Riesz measure of an $\cA$-supersolution $u$ as follows: since the distribution $-\DIV\cA(x,\nabla u)$ is positive, it is represented by a unique nonnegative Radon measure, denoted by $\mu[u]$, as
\begin{equation}\label{eq-Riesz-supersoln}
\int_\Omega \cA(x,\nabla u) \cdot \nabla\varphi \dx=\int_\Omega \varphi \dmu[u] \quad\text{for all }\varphi\in C^\infty_c(\Omega).
\end{equation}
The Riesz measure of a general $\cA$-superharmonic function will be introduced later in Section~\ref{sec4-5}.

We first prove a uniform local energy estimate for bounded supersolutions.

\begin{lemma}\label{lem-uniform-bound}
Let $\{u_i\}$ be a sequence of locally uniformly bounded $\cA$-supersolutions of \eqref{eq-sec2-new}. Then $\{\nabla u_i\}$ is uniformly bounded in $L^G_{\mathrm{loc}}(\Omega)$. In particular, $\{\cA(x, \nabla u_i)\}$ is uniformly bounded in $L^{G^{\ast}}_{\mathrm{loc}}(\Omega)$.
\end{lemma}

\begin{proof}
Let $\Omega' \Subset \Omega$ be an open set, and choose a cutoff function $\eta \in C^\infty_c(\Omega)$ such that $0\leq \eta\leq 1$ in $\Omega$ and $\eta=1$ on $\Omega'$. Set
\begin{equation*}
M:=\textstyle\sup_{i}\|u_i\|_{L^\infty(\supp\eta)}<\infty.
\end{equation*}
Fix $j\in\mathbb N$. Then the function $(M-u_j)\eta^q$ is an admissible nonnegative test function for \eqref{eq-supersolution} for $u_j$, by a standard approximation argument. Hence, using the structural conditions \eqref{ass-op} and Lemma~\ref{lem-barg}(iii) with $\delta \in (0,1)$, we obtain
\begin{align*}
	c_1^{\cA} \int_\Omega G(|\nabla u_j|)\eta^q \dx
	&\le \int_\Omega \cA(x,\nabla u_j)\cdot \eta^q \nabla u_j \dx \le q\int_\Omega \cA(x,\nabla u_j)\cdot \eta^{q-1} \nabla \eta (M - u_j) \dx \\
	&\le 2qc_2^{\cA} M \int_\Omega g(|\nabla u_j|) \eta^{q-1} |\nabla \eta| \dx \\
	&\leq 2q^2c_2^{\cA} M \int_{\Omega \cap \{\eta>0\}} \left( \delta G(|\nabla u_j|) + \delta^{1-q} G(\eta^{-1}|\nabla \eta|) \right) \eta^q \dx.
\end{align*}
By taking $\delta$ sufficiently small, we can absorb the term containing $G(|\nabla u_j|)$ in to the left-hand side. Furthermore, observing that $\eta^qG(\eta^{-1}|\nabla \eta|) \leq G(|\nabla \eta|)$ by Lemma~\ref{lem-barg}(i), we deduce that
\begin{equation*}
\int_{\Omega'} G(|\nabla u_j|) \dx \leq \int_{\Omega} G(|\nabla u_j|) \eta^q \dx \leq C \int_{\Omega \cap \{\eta>0\}} G(\eta^{-1} |\nabla \eta|)\eta^q \dx \leq C \int_\Omega G(|\nabla \eta|) \dx < \infty,
\end{equation*}
which implies that the sequence $\{\nabla u_i\}$ is uniformly bounded in $L^G(\Omega')$. Finally, it follows from \eqref{ass-op} and Lemma~\ref{lem-H} that the sequence $\{\cA(x,\nabla u_i)\}$ is uniformly bounded in $L^{G^*}(\Omega')$.
\end{proof}

The next lemma is a convergence theorem for uniformly bounded $\cA$-supersolutions. It is the main compactness result used below to pass to the limit in Riesz measures of bounded balayage and obstacle-type functions. The following proof is partly inspired by the ideas from \cite{KKP-JEE-10,LM-Diff-07}.

\begin{lemma}\label{lem-mea-weakcon}
Let $\{u_i\}$ be a sequence of locally uniformly bounded $\cA$-supersolutions of \eqref{eq-sec2-new} such that $u_i\to u$ a.e.\ in $\Omega$. Then $u$ is an $\cA$-supersolution of \eqref{eq-sec2-new} and $\mu[u_i] \rightharpoonup\mu[u]$ weakly as $i\to \infty$, that is,
\begin{equation*}
\int_\Omega \varphi \dmu[u_i] \to \int_\Omega \varphi \dmu[u] \quad\text{for all }\varphi \in C_c(\Omega).
\end{equation*}
\end{lemma}

\begin{proof}
We shall prove that the sequence $\{\nabla u_i\}$ is Cauchy in measure, which implies that it converges in measure to its distributional limit $\nabla u$. Once this is established, an arbitrary subsequence $\{\nabla u_{i_j}\}$ of $\{\nabla u_i\}$ has a further subsequence $\{\nabla u_{i_{j_k}}\}$ converging to $\nabla u$ a.e.\ in $\Omega$. Since $\cA(x, \cdot)$ is continuous for almost every $x \in \Omega$, it follows that $\cA(x, \nabla u_{i_{j_k}}) \to \cA(x, \nabla u)$ a.e.\ in $\Omega$. Combined with the uniform bound from Lemma~\ref{lem-uniform-bound}, this a.e.\ pointwise convergence implies that $\cA(x,\nabla u_{i_{j_k}})\rightharpoonup\cA(x,\nabla u)$ weakly in $L^{G^*}_{\mathrm{loc}}(\Omega)$. By Urysohn's subsequence principle, the full sequence also converges weakly, so that
\[
 \lim_{i\to\infty}\int_\Omega \cA(x,\nabla u_{i})\cdot\nabla\varphi\dx=\int_\Omega \cA(x,\nabla u)\cdot\nabla\varphi\dx,
\]
for any $\varphi\in C_c^{\infty}(\Omega)$. This proves that $u$ is an $\cA$-supersolution of \eqref{eq-sec2-new} and that $\mu[u_i] \rightharpoonup\mu[u]$ weakly.

Now, we show that the sequence $\{\nabla u_i\}$ is Cauchy in measure. Fix an open set $\Omega' \Subset \Omega$. For every $\rho, \theta >0$, we need to find $N \in \mathbb{N}$ such that $|E_{jk}| < \theta$ for all $j, k \geq N$, where
\begin{equation*}
E_{jk} = \{x \in \Omega': |\nabla u_j(x) - \nabla u_k(x)| \geq \rho\}.
\end{equation*}
Since the a.e.\ pointwise convergence $u_i \to u$ implies that $\{u_i\}$ is Cauchy in measure, for $\delta>0$ to be determined later, there exists $N \in \mathbb{N}$ such that for all $j, k \geq N$ it holds
\begin{equation*}
|\{x \in \Omega':\ |u_j(x) - u_k(x)| \ge \delta\}| < \theta/3.
\end{equation*}
Furthermore, since the sequence $\{\nabla u_i\}$ is uniformly bounded in $L^G(\Omega')$ by Lemma~\ref{lem-uniform-bound}, there exists a large constant $\lambda>0$ such that for all $j, k \geq N$ it holds
\begin{equation*}
|\{x \in \Omega':\ |\nabla u_j(x)| > \lambda\ \text{ or }\ |\nabla u_k(x)| > \lambda\}| < \theta/3.
\end{equation*}  Therefore, it suffices to derive for all $j, k \geq N$ that
\begin{equation*}
|E_{jk} \cap U_{jk}| < \theta/3,
\end{equation*} where
\begin{equation*}
U_{jk} := \{x\in\Omega':\ |u_j(x)-u_k(x)|<\delta,\ |\nabla u_j(x)| \leq \lambda,\ |\nabla u_k(x)| \leq \lambda\}.
\end{equation*}

Let us choose a nonnegative cutoff function $\eta \in C_c^\infty(\Omega)$ such that $\eta = 1$ on $\Omega'$. For  $T_\delta$ being the truncation defined in \eqref{eq-truncation}, we define the nonnegative test functions
	\[
	\eta_1:= (\delta - T_\delta(u_j - u_k))\eta \quad\text{and}\quad \eta_2 := (\delta + T_\delta(u_j - u_k))\eta.
	\]
	 Testing \eqref{eq-sec2-new} for $u_j$ and $u_k$ against $\eta_1$ and $\eta_2$, respectively, yields
	\[
	\int_\Omega \cA(x,\nabla u_j)\cdot\nabla\eta_1 \dx \ge 0 \quad\text{and}\quad
	\int_\Omega \cA(x,\nabla u_k)\cdot\nabla\eta_2 \dx \ge 0.
	\]
	Adding these two inequalities and rearranging the terms, we obtain
	\begin{align*}
	&\int_{\Omega}\left(\cA(x,\nabla u_j)-\cA(x,\nabla u_k)\right)\cdot \eta\nabla  T_{\delta}(u_j-u_k) \dx\\
	&\quad\leq  \delta\int_{\Omega}\left(  \cA(x,\nabla u_j)+\cA(x,\nabla u_k)   \right)\cdot\nabla \eta\dx  -\int_{\Omega}\left(\cA(x,\nabla u_j)-\cA(x,\nabla u_k)\right)\cdot T_{\delta}(u_j-u_k) \nabla\eta\dx.
	\end{align*}
Therefore, it follows from the uniform boundedness of $\{\cA(x, \nabla u_i)\}$ in $L^{G^{\ast}}(\supp\eta)$ (Lemma~\ref{lem-uniform-bound}) and  \eqref{eq-monotonicity} -- the monotonicity of $\cA(x, \cdot)$, that 
for some $C>0$ independent of $j$ and $k$ it holds
 \begin{equation}\label{eq-de}
 	\int_{E_{jk} \cap U_{jk}}\left(\cA(x,\nabla u_j)-\cA(x,\nabla u_k)\right)\cdot \nabla (u_j-u_k)\dx \leq C\delta.
\end{equation}
Let us define
\[
\gamma(x):=\inf\left\{\left(\cA(x,\xi)-\cA(x,\zeta)\right)\cdot(\xi-\zeta):\ |\xi|,|\zeta|\leq\lambda,\ |\xi-\zeta|\ge\rho\right\}.
\]
We then have from \eqref{eq-de} that
\begin{equation*}
\int_{E_{jk}\cap U_{jk}} \gamma(x) \dx\leq C\delta.
\end{equation*}
Since $\cA$ is a Carath\'eodory function, for almost every $x \in \Omega$, the infimum is taken over a compact set. By  \eqref{eq-monotonicity} -- the strict monotonicity of $\cA(x, \cdot)$, this infimum is attained and strictly positive for a.e.\ $x \in \Omega'$. Thus, we can choose a sufficiently small $\gamma_0>0$ such that
\begin{equation*}
|\{x \in \Omega': \gamma(x) \leq \gamma_0\}|<\theta/6.
\end{equation*}
Then, it follows that
\begin{equation*}
\gamma_0 |E_{jk} \cap U_{jk} \cap \{\gamma(x) > \gamma_0\}| \leq \int_{E_{jk}\cap U_{jk} \cap \{\gamma>\gamma_0\}} \gamma \dx\leq C\delta.
\end{equation*}
Now, we choose $\delta>0$ small enough so that $C\delta \leq \gamma_0 \theta/6$. Then, we are in a position to conclude as
\begin{equation*}
|E_{jk} \cap U_{jk}| \leq |\{x \in \Omega': \gamma(x) \leq \gamma_0\}| + |E_{jk}\cap U_{jk} \cap \{\gamma>\gamma_0\}| < \theta/3\,. \qedhere
\end{equation*}
\end{proof}

\subsection{Balayage}\label{sec4-3}

For a function $u: E \to [-\infty, \infty]$, its \emph{lsc-regularization} $\hat{u}$ is defined by
\begin{equation*}
\hat{u}(x) = \lim_{r \to 0} \inf_{E \cap B_r(x)} u.
\end{equation*}
Clearly, $\hat{u} \leq u$ on $E$. If $u$ is locally bounded below, then $\hat{u}$ is lower semicontinuous; indeed, $\hat{u}$ is the greatest lower semicontinuous minorant of $u$.

Let $\psi: \Omega \to (-\infty, \infty]$ be a function that is locally bounded below, and let
\begin{equation*}
\Phi^\psi := \{v \in \cS_{\cA}(\Omega) : v \geq \psi \text{ in } \Omega\}.
\end{equation*}
The function $R^\psi := R^\psi(\Omega) = \inf \Phi^\psi$
is called the \emph{r\'{e}duite} and its lsc-regularization
$\hat{R}^\psi := \hat{R}^\psi(\Omega)$ 
is called the \emph{balayage} of $\psi$ in $\Omega$. If $u$ is a nonnegative function on a set $E \subset \Omega$ and $\psi = u \chi_E$, we write
$$
\Phi_E^u = \Phi^\psi, \quad R_E^u = R^\psi, \quad \text{and} \quad \hat{R}_E^u = \hat{R}^\psi.
$$
We refer to $\hat{R}_E^u$ as the \emph{balayage of $u$ relative to $E$ in $\Omega$}. If $u \equiv \ell$ for some constant $\ell>0$, we write $\hat{R}_E^\ell$ accordingly. In particular, the function $\hat{R}_E^1$ is called the \emph{$\cA$-potential} of $E$ in $\Omega$. Later, in Proposition~\ref{prop57}, this potential will be identified with the capacitary potential.

As stated in the following lemmas, the balayage $\hat{R}^\psi$ is $\cA$-superharmonic in $\Omega$ and coincides with the r\'eduite $R^\psi$ a.e.\ in $\Omega$. The proof of \cite[Theorem~7.4]{HKM06} for the standard growth case applies almost verbatim to the present setting of Lemma~\ref{lem-inf-F}.

\begin{lemma}\label{lem-inf-F}
Let $\cF\subset \cS_{\cA}(\Omega)$ be a non-empty family that is locally uniformly bounded below. Then, the lsc-regularization of $\inf \cF$ belongs to $\cS_{\cA}(\Omega)$.
\end{lemma}

\begin{lemma}\label{lem-ae}
If $\Phi^\psi \neq \emptyset$, then $\hat{R}^\psi=R^\psi$ a.e.\ in $\Omega$. If, in addition, $\hat{R}^\psi$ is locally bounded above in $\Omega$, then it is an $\cA$-supersolution of \eqref{eq-sec2-new}.
\end{lemma}

To prove Lemma~\ref{lem-ae}, we need a series of lemmas. We first recall the following Choquet's topological lemma; see~\cite[Lemma~8.3]{HKM06}. A family of functions $\cF$ is called \emph{downward directed} if for each $u_1, u_2 \in \cF$ there is $v \in \cF$ with $v \leq \min\{u_1, u_2\}$.

\begin{lemma}[Choquet's topological lemma]\label{lem:choquet}
Suppose that $E \subset \rn$ and that $\cF = \{u_\gamma : \gamma \in I\}$ is a family of functions $u_\gamma : E \to [-\infty, \infty]$. Let $u: = \inf \cF$. If $\cF$ is downward directed, then there is a decreasing sequence of functions $\{v_i\} \subset \cF$ with the pointwise limit $v$ such that the lsc-regularizations $\hat{u}$ and $\hat{v}$ coincide.
\end{lemma}

We are now ready to prove Lemma~\ref{lem-ae}.

\begin{proof}[Proof of Lemma~\ref{lem-ae}]
Since the minimum of two functions in $\Phi^\psi$ belongs to $\Phi^\psi$, the family $\Phi^\psi$ is downward directed. Thus, by Choquet's topological Lemma~\ref{lem:choquet}, there is a decreasing sequence of functions $\{v_i\} \subset \Phi^\psi$ with the pointwise limit $v$ such that $\hat{v}=\hat{R}^\psi$. Since $\hat{v} = \hat{R}^\psi \leq R^\psi \leq v$, it suffices to show that $v = \hat{v}$ a.e.\ in $\Omega$.

Notice that $v_i \geq v \geq \psi$ in $\Omega$. Since $\psi$ is locally bounded below, for any integer $k>0$, the truncated functions $w_{i,k} := \min\{v_i, k\}$ are locally bounded uniformly in $i$. By Lemma~\ref{lem-cS-supersol}, each $w_{i,k}$ is an $\cA$-supersolution of \eqref{eq-sec2-new}. Since $\{w_{i,k}\}_i$ is a locally bounded decreasing sequence of $\cA$-supersolutions, its pointwise limit $w_k := \min\{v, k\}$ is also an $\cA$-supersolution by Lemma~\ref{lem-mea-weakcon}. Thus, by Lemmas~\ref{lem-supersol-cS} and \ref{lem-super-liminf}, $w_k$ admits an $\cA$-superharmonic representative, which implies that $w_k$ coincides with its lsc-regularization $\hat{w}_k$ a.e.\ in $\Omega$. Since $w_k = \hat{w}_k$ a.e.\ for all $k>0$, letting $k \to \infty$ yields $v = \hat{v}$ a.e.\ in $\Omega$.

The last assertion follows from Lemma~\ref{lem-cS-supersol}.
\end{proof}

We further investigate the properties of the balayage $\hat{R}^u_E$ of $u$ relative to $E$. Proposition~\ref{prop-bal-u-E} establishes its $\cA$-harmonicity in $\Omega \setminus \overline{E}$ and pointwise behavior. Before presenting this lemma, we provide a couple of auxiliary results.

By definition, an $\cA$-superharmonic function is not identically $\infty$. In fact, it is finite on a dense subset of $\Omega$ as shown in the following lemma. This result will be improved in Lemma~\ref{lem-super-int} and Theorem~\ref{th-1}.

\begin{proposition}
    \label{prop-dense}
If $u\in \cS_{\cA}(\Omega)$, then the set $\{x \in \Omega : u(x) < \infty\}$ is dense in $\Omega$.
\end{proposition}

\begin{proof}
We may assume that $\Omega$ is a domain. Suppose that there is a ball $B \Subset \Omega$ such that $u = \infty$ in $\overline{B}$. Since $u \not\equiv \infty$ in $\Omega$, we can choose a smooth domain $U \Subset \Omega$ containing $\overline{B}$ and a point $y \in \Omega$ with $u(y)<\infty$. Since the boundary of $U \setminus \overline{B}$ is smooth, Proposition~\ref{prop-known} and Theorem \ref{thm-bdry-reg} with $\psi \equiv -\infty$ show that there exists a function $u_k$ that is $\cA$-harmonic in $U \setminus \overline{B}$ and continuous in $\overline{U} \setminus B$, taking the value 0 on $\partial U$ and $k$ on $\partial B$. We define $G_k(t) := k^{-1}G(kt)$ and $g_k(t):= g(kt)$ so that $G_k(t) = \int_0^t g_k(\tau)\dtau$. Then $G_k$ satisfies \eqref{eq:Gg} with the same constants $p$ and $q$, and the operator $\cA_k(x,\xi):=\cA(x,k\xi)$ satisfies the structural assumptions \eqref{eq-monotonicity} and \eqref{ass-op} with $G$ replaced by $G_k$ and with the same constants. Let $v_k=u_k/k$. Then, $v_k$ is $\cA_k$-harmonic in $U \setminus \overline{B}$ and takes the boundary value 0 on $\partial U$ and 1 on $\partial B$ continuously. Note that $0\leq v_k\leq 1$ by the weak maximum principle (see, e.g., \cite[Lemma~4.3]{CK21remov}).

We claim that $v_k \geq c$ for some universal constant $c>0$ independent of $k$. Indeed, by the H\"older estimates up to the boundary in Theorem~\ref{thm-bdry-reg}, there exist constants $\alpha \in (0, 1)$ and $C>0$, depending only on $n$, $p$, $q$, $c_1^{\cA}$, and $c_2^{\cA}$, such that
\begin{equation*}
1-v_k(x) = |v_k(x)-v_k(x_0)| \leq C|x-x_0|^\alpha
\end{equation*}
for all $x \in U$ and $x_0 \in \partial B$. Note that the constants $\alpha$ and $C$ are independent of $k$. Thus, we can choose a point $z \in U$ such that $v_k(z) \geq 1/2$. Now, a standard Harnack chain argument shows that $v_k(y) \geq c$ with $c>0$ independent of $k$. See, e.g., \cite[Theorem~5]{Che22Gen} for the Harnack inequality. This implies that $u_k(y) \geq ck$.

Since $u$ is locally bounded below, we have $m := \inf_{\overline{U}} u>-\infty$. Then, the comparison principle shows that $u_k \leq u - m$ in $U \setminus \overline{B}$ for all $k>0$. In particular, $ck \leq u_k(y) \leq u(y)-m < \infty$ for all $k>0$, which is a contradiction.
\end{proof}

For $u \in \cS_{\cA}(\Omega)$ and an open set $U \Subset \Omega$ with smooth boundary, the \emph{Poisson modification} $(u)_U$ of $u$ in $U$ is defined by
\begin{equation*}
(u)_U=
\begin{cases}
\inf\{v: v \in \cS_{\cA}(U),\ \liminf_{y \to x}v(y) \geq u(x) \text{ for each } x \in \partial U\} &\text{in }U, \\
u &\text{in }\Omega \setminus U.
\end{cases}
\end{equation*}
The following standard properties of the Poisson modification hold.

\begin{lemma}\label{lem-Poisson}
The Poisson modification $(u)_U$ of $u \in \cS_{\cA}(\Omega)$ in $U$ is $\cA$-superharmonic in $\Omega$, $\cA$-harmonic in $U$, and $(u)_U \leq u$ in $\Omega$.
\end{lemma}

\begin{proof}
This proposition was proved in \cite[Theorem~3]{Che22Gen} under the assumption that $u \in \cS_{\cA}(\Omega)$ is finite a.e.\ in $\Omega$. In their proof, the limit function $h$ of the sequence of $\cA$-harmonic functions $\{h_i\}$ is either $\cA$-harmonic or identically infinite, and the a.e.\ finiteness of $u$ is then used to rule out the possibility that $h \equiv \infty$. Alternatively, this can be done by using Proposition~\ref{prop-dense}.
\end{proof}

The balayage $\hat{R}^u_E$ enjoys the following properties. 

\begin{proposition}\label{prop-bal-u-E}
Let $u$ be a nonnegative function on a set $E \subset \Omega$. Assume that $\Phi^u_E \neq \emptyset$. The balayage $\hat{R}_E^u$ is $\cA$-harmonic in $\Omega \setminus \overline{E}$ and coincides with $R_E^u$ there. If, in addition, $u \in \cS_{\cA}(\Omega)$, then $\hat{R}_E^u = u$ in the interior of $E$.
\end{proposition}
	
\begin{proof}
As in the proof of Lemma~\ref{lem-ae}, Choquet's topological lemma (Lemma~\ref{lem:choquet}) yields a decreasing sequence of functions $\{v_i\} \subset \Phi^u_E$ converging pointwise to a limit $v$ such that $\hat{v}=\hat{R}^u_E$.  Choose a ball $B \subset \Omega \setminus \overline{E}$ and let $w_i := (v_i)_B$ be the Poisson modification of $v_i$ in $B$. Then $w_i$ is $\cA$-harmonic in $B$, $w_i \in \Phi^u_E$, and $w_{i+1} \leq w_i \leq v_i$ by Lemma~\ref{lem-Poisson}. Thus,
\[
R_E^u \leq w := \lim_{i \to \infty} w_i \leq \lim_{i \to \infty}v_i = v,
\]
and therefore $\hat{R}_E^u = \hat{w}$ in $B$. This establishes the first part, because $\hat{w} = w$ is $\cA$-harmonic in $B$ by Harnack's convergence theorem (Proposition~\ref{prop-conv-har} applied to the reflected operator $\widetilde{\cA}(x, \xi):=-\cA(x, -\xi)$ and the sequence $\{-w_i\}$). The second part is immediate as $u$ is lower semicontinuous and $u \in \Phi_E^u$.
\end{proof}

When $u$ is a positive constant, we obtain the following proposition. The proof of \cite[Lemma~8.5]{HKM06} for the standard growth case applies almost verbatim here, as we are equipped with the necessary ingredients for that proof: the pasting lemma, Harnack's convergence theorems, the fact that any $\cA$-harmonic function is a quasiminimizer of the energy with Orlicz growth $G$ (see \cite[Lemma~4.3, Corollary~4.9, Theorem~2, and Lemma~3.7]{Che22Gen}), and Proposition~\ref{prop-bal-u-E}. We thus omit its proof.

\begin{proposition}
    \label{prop-bal-K}
Let $K \subset \Omega$ be compact and let $u = \hat{R}_K^\ell(\Omega)$ for some constant $\ell>0$. Let $\varphi \in C_c^\infty(\Omega)$ be such that $\varphi = \ell$ on $K$. Then $u$ is the unique $\cA$-harmonic function in $\Omega \setminus K$ with $u - \varphi \in W_0^{1,G}(\Omega \setminus K)$. Moreover, $0 \leq u \leq \ell$ in $\Omega$, and in particular $u \in W^{1,G}_0(\Omega)$. 
\end{proposition}

The above proposition can be used to show that the dual capacity associated with $\cA$ is comparable to the relative $G$-capacity. Our proof follows the ideas of \cite[Theorem~3.5]{KM94}, but do not make use of homogeneity~\eqref{eq-homogeneity}.

\begin{lemma}\label{lem-cap-dual}
If $E\subset\Omega$ is a Borel set, then
\begin{equation}\label{eq-cap-dual}
C_1\,\mathrm{cap}_G(E,\Omega)\leq\widetilde{\mathrm{cap}}_{\cA}(E,\Omega)\leq C_2\,\mathrm{cap}_G(E,\Omega)
\end{equation}
for some constants $C_1,C_2>0$ depending only on $p$, $q$, $c_1^\cA$, and $c_2^\cA$.
\end{lemma}

\begin{proof}
	 	We  assume that $E=K$ is compact without loss of generality. 	Let $u=\hat{R}_{K}^{1}$ be the $\cA$-potential of $K$ in $\Omega$. By Proposition~\ref{prop-bal-K} and Lemma~\ref{lem-cS-supersol}, $u \in W^{1,G}_0(\Omega)$ is an $\cA$-supersolution of \eqref{eq-sec2-new} in $\Omega$, and hence its Riesz measure $\mu=\mu[u]$ belongs to $(W^{1,G}_0(\Omega))'$ and has support on $K$; see Section~\ref{sec4-2} for the definition of the Riesz measure of an $\cA$-supersolution.  Note that $u$ can be used as a test function for \eqref{eq-Riesz-supersoln} by approximation. Since $0 \leq u \leq 1$ in $\Omega$, we have from the definition of $\widetilde{\mathrm{cap}}_{\cA}$ that
	 	\begin{equation}\label{eq-tilde-cap-lower}
			 			\widetilde{\mathrm{cap}}_{\cA}(K,\Omega) \geq \mu(\Omega)\geq \int_{\Omega} u \dmu= \int_{\Omega}\cA(x,\nabla u)\cdot \nabla u\dx\geq c^{\cA}_1 \int_{\Omega}G(|\nabla u|)\dx.
		 	\end{equation}
We claim that
\begin{equation}\label{eq-claim}
\mathrm{cap}_G(K,\Omega) \leq \int_\Omega G(|\nabla u|) \dx.
\end{equation}
Indeed, let $\varphi \in C^\infty_c(\Omega)$ be such that $0 \leq \varphi \leq 1$ in $\Omega$ and $\varphi = 1$ in a neighborhood of $K$. By Proposition~\ref{prop-bal-K}, $u-\varphi \in W^{1,G}_0(\Omega \setminus K)$. Thus there exists a sequence $\{\varphi_j\} \subset C^\infty_c(\Omega \setminus K)$ such that
\begin{equation*}
\varphi_j \to u-\varphi \quad\text{in }W^{1,G}(\Omega).
\end{equation*}
Let $u_j:=\varphi+\varphi_j$. Then $u_j \in C^\infty_c(\Omega)$, $u_j=1$ in a neighborhood of $K$, and $u_j \to u$ in $W^{1,G}(\Omega)$. We may assume that $u_j$ is nonnegative by taking its positive part and considering a mollification if necessary. Therefore, $u_j$ is admissible for $\mathrm{cap}_G(K, \Omega)$, yielding that
\begin{equation*}
\mathrm{cap}_G(K,\Omega) \leq \int_\Omega G(|\nabla u_j|) \dx.
\end{equation*}
Passing to the limit as $j \to \infty$ verifies the claim \eqref{eq-claim}. Combining this with \eqref{eq-tilde-cap-lower} proves the first inequality in \eqref{eq-cap-dual}.

	 	For the second inequality, let $\nu$ be any nonnegative measure in $(W_{0}^{1,G}(\Omega))'$ with $\supp\nu\subset K$ and let $u_{\nu}$ be the unique function satisfying \eqref{eq-u-nu} and $0 \leq u_\nu \leq 1$ a.e.\ in $\Omega$. Take any nonnegative function $v\in C_{c}^{\infty}(\Omega)$ such that $v=1$ on $K$. Then
	 	\begin{equation}\label{eq-upp1}
		 		\nu(K)\leq \int_{\Omega}v\,\mathrm{d}\nu=\int_{\Omega}\cA(x,\nabla u_{\nu})\cdot\nabla v\dx.
		 	\end{equation}
	 By the structural assumptions \eqref{ass-op} and  Lemma~\ref{lem-barg}, 
	 	\begin{equation}\label{eq-upp2}
	 \int_{\Omega}\cA(x,\nabla u_{\nu})\cdot\nabla v\dx \le q c^{\cA}_2 \int_{\Omega} \big(G(|\nabla u_{\nu}|)+ G(|\nabla v|)\big) \dx.
	 	\end{equation}
	 
	We next estimate $\int_{\Omega}G(|\nabla u_{\nu}|) \dx$. Set $v_{1}=\max\{v,u_{\nu}\}$. Then $v_{1}-u_{\nu}\in W_{0}^{1,G}(\Omega)$ is nonnegative and $v-v_{1}\in W_{0}^{1,G}(\Omega\setminus K)$. Note that  $u_{\nu}$ is an $\cA$-solution to \eqref{eq-sec2-new} in $\Omega \setminus K$ and an $\cA$-supersolution to  \eqref{eq-sec2-new} in $\Omega$. Since $v-v_1=0$  on $K$, we have
	 	\begin{equation*}
		 		\begin{split}
			 			\int_{\Omega}\cA(x,\nabla u_{\nu})\cdot \nabla (v-u_{\nu})\dx
			 			&= \int_{\Omega}\cA(x,\nabla u_{\nu})\cdot \nabla (v_1-u_{\nu})\dx+\int_{\Omega\setminus K}\cA(x,\nabla u_{\nu})\cdot \nabla (v-v_1)\,\dx\\
			 			&= \int_{\Omega}\cA(x,\nabla u_{\nu})\cdot \nabla (v_1-u_{\nu})\dx \geq 0.
			 		\end{split}
		 	\end{equation*}
	 	 Hence, for $\delta\in (0,1)$,  by   \eqref{ass-op} and  Lemma~\ref{lem-barg}, we obtain
	 	\begin{equation*}
		 		\begin{split}
			 			c^{\cA}_1\int_{\Omega} G(|\nabla u_{\nu}|)\dx &\leq \int_{\Omega} \cA(x,\nabla u_{\nu})\cdot \nabla u_{\nu}\dx  \leq \int_{\Omega} \cA(x,\nabla u_{\nu})\cdot \nabla v\dx \\
			 			&\leq q c^{\cA}_2 \int_{\Omega}\left( \delta G(|\nabla u_{\mu}|)+ \frac{q}{p}\delta^{1-q}G(|\nabla v|)\right)\dx.
			 		\end{split}
		 	\end{equation*} 
	 	 By taking $\delta $ sufficiently small, it follows that 
	 	\begin{equation*}
		 		\int_{\Omega} G(|\nabla u_{\nu}|)\dx \leq C\int_{\Omega} G(|\nabla v|)\dx.
		 	\end{equation*}
	 Combining this with \eqref{eq-upp1} and \eqref{eq-upp2}, we deduce that 
	 	\[
	 	\nu(K) \le C_2 \int_{\Omega} G(|\nabla v|)\dx
	 	\]
	 	for some $C_2=C_2(p, q, c_1^{\cA}, c_2^{\cA})>0$. Taking the infimum over all $0\leq v\in C_0^\infty(\Omega)$ with $v=1$ on $K$, we obtain $\nu(K)\le C_2\,\mathrm{cap}_G(K,\Omega)$. Since $\nu$ was arbitrary,
	 	\[
	 	\widetilde{\mathrm{cap}}_{\cA}(K,\Omega) = \sup \nu(K) \le C_2\,\mathrm{cap}_G(K,\Omega).
	 	\]
	 	This completes the proof.
	 \end{proof}

\subsection{Balayage and obstacle problems}\label{sec4-4}

The theory of obstacle problems presented in Section~\ref{sec4-1} is intimately connected with balayage theory. This connection is particularly useful when the obstacle is supported on an open set as in the following theorem. In contrast, the corresponding statement is generally false for compact sets, where the r\'eduite and its lsc-regularization may differ pointwise.

\begin{theorem}\label{thm-obs-open}
Suppose that $u \in \cS_{\cA}(\Omega)$ is nonnegative and bounded. Let $U \Subset \Omega$ be an open set. Then $R^u_U \equiv \hat{R}^u_U$ is the unique lsc-regularized solution to the $\mathcal{K}_{u\chi_U, 0}(\Omega)$-obstacle problem.
\end{theorem}

We need a series of lemmas to prove Theorem~\ref{thm-obs-open}. We first solve the obstacle problem with a nonnegative smooth obstacle with compact support.

\begin{lemma}\label{lem-R-obs}
Let $\psi \in C^\infty_c(\Omega)$ be nonnegative. Then, $R^\psi \equiv \hat{R}^\psi$ is the unique continuous solution to the $\mathcal{K}_{\psi, 0}(\Omega)$-obstacle problem.
\end{lemma}

\begin{proof}
To rule out the trivial case, we may assume that $\psi \not\equiv 0$. Let $v$ be the unique continuous solution to the $\mathcal{K}_{\psi, 0}(\Omega)$-obstacle problem obtained in Proposition~\ref{prop-known}. We shall prove that $R^\psi \equiv v$.

First, we approximate $v$ by a sequence of solutions to obstacle problems. We exhaust $\Omega$ by smooth domains $\Omega_1 \Subset \Omega_2 \Subset \cdots \Subset \Omega$ such that $\supp\psi \Subset \Omega_1$ and $\bigcup_i \Omega_i=\Omega$. Let $h_i$ be the continuous solution to the $\mathcal{K}_{\psi,0}(\Omega_i)$-obstacle problem obtained in Proposition~\ref{prop-known}. Since $\Omega_i$ is smooth and the obstacle $\psi$ is smooth, Theorem~\ref{thm-bdry-reg} ensures that $h_i \in C_0(\overline{\Omega}_i)$, so we can extend $h_i$ continuously to $\Omega$ by setting $h_i=0$ on $\Omega \setminus \Omega_i$. Then, $h_i \geq h_j$ in $\Omega$ if $i \geq j$. Indeed, since $h_j-\min\{h_i, h_j\} \in W^{1,G}_0(\Omega_i)$ is nonnegative and $h_i$ is an $\cA$-supersolution in $\Omega_i$, we have
\begin{equation*}
\int_{\Omega_i} \cA(x, \nabla h_i) \cdot \nabla (h_j-\min\{h_i, h_j\}) \dx \geq 0.
\end{equation*}
Moreover, since $\min\{h_i, h_j\} \in \mathcal{K}_{\psi, 0}(\Omega_j)$, the variational inequality \eqref{eq-obs} for $h_j$ yields
\begin{equation*}
\int_{\Omega_j} \cA(x, \nabla h_j) \cdot \nabla (\min\{h_i, h_j\}-h_j) \dx \geq 0.
\end{equation*}
We thus have
\begin{align*}
0
&\leq \int_\Omega (\cA(x, \nabla h_i)-\cA(x, \nabla h_j)) \cdot \nabla (h_j - \min\{h_i, h_j\}) \dx \\
&= - \int_{\Omega \cap \{h_i<h_j\}} (\cA(x, \nabla h_i)-\cA(x, \nabla h_j)) \cdot \nabla (h_i - h_j) \dx.
\end{align*}
By the strict monotonicity \eqref{eq-monotonicity} of $\cA(x, \cdot)$, this implies that $\nabla (h_j-h_i)_+=0$ a.e.\ in $\Omega$. Since $(h_j-h_i)_+ \in W^{1,G}_0(\Omega) \cap C(\Omega)$, it follows that $(h_j-h_i)_+=0$, i.e., $h_i \geq h_j$ in $\Omega$. Similarly, $v \geq h_i$ in $\Omega$. We denote the pointwise limit of $\{h_i\}$ by $h:=\lim_{i\to \infty}h_i$. Our aim is to prove that $h \equiv v$.

Using $\psi \in \mathcal{K}_{\psi, 0}(\Omega_i)$ as a test function in the variational inequality \eqref{eq-obs} for $h_i$ and applying the structural conditions \eqref{ass-op}, we obtain
\begin{align*}
c_1^{\cA} \int_{\Omega_i} G(|\nabla h_i|) \dx
&\leq \int_{\Omega_i} \cA(x, \nabla h_i) \cdot \nabla h_i \dx \\
&\leq \int_{\Omega_i} \cA(x, \nabla h_i) \cdot \nabla \psi \dx \leq c_2^{\cA} \int_{\Omega_i} g(|\nabla h_i|)|\nabla \psi| \dx.
\end{align*}
Applying Lemma~\ref{lem-barg}(iii) with $\varepsilon > 0$ and using \eqref{eq:Gg}, we have
\begin{equation*}
\int_{\Omega_i} G(|\nabla h_i|) \dx \leq \frac{c_2^{\cA}}{c_1^{\cA}} \left( q\varepsilon \int_{\Omega_i} G(|\nabla h_i|) \dx + q\varepsilon^{1-q} \int_{\Omega} G(|\nabla \psi|) \dx \right).
\end{equation*}
Taking sufficiently small $\varepsilon$ and recalling $h_i=0$ outside $\Omega_i$ yields the uniform energy bound
\begin{equation}\label{eq-energy-bound-i}
\int_{\Omega} G(|\nabla h_i|) \dx \leq C \int_{\Omega} G(|\nabla \psi|) \dx,
\end{equation}
where $C>0$ is independent of $i$. Thus, $\{h_i\}$ converges weakly to $h$ in $W^{1,G}_0(\Omega)$. By the interior H\"older regularity for obstacle problem solutions (Proposition~\ref{prop-known}), the sequence $\{h_i\}$ is locally equicontinuous in $\Omega$. Hence the limit $h$ is continuous in $\Omega$.

Next, we show that $h \equiv v$. Let $\varphi \in \mathcal{K}_{\psi, 0}(\Omega)$ be any function with compact support in $\Omega$. For sufficiently large $i$, we have $\supp \varphi \Subset \Omega_i$, which implies $\varphi \in \mathcal{K}_{\psi, 0}(\Omega_i)$. Since $h_i$ solves the $\mathcal{K}_{\psi, 0}(\Omega_i)$-obstacle problem, we have
\begin{equation*}
\int_{\Omega_i} \cA(x, \nabla h_i) \cdot \nabla (\varphi - h_i) \dx \geq 0.
\end{equation*}
By applying the monotonicity \eqref{eq-monotonicity} of the operator $\cA(x, \cdot)$ and extending the integral over $\Omega$, we obtain
\begin{equation*}
\int_\Omega \cA(x, \nabla \varphi) \cdot \nabla (\varphi - h_i) \dx \geq 0.
\end{equation*}
Notice that $\cA(x, \nabla \varphi) \in L^{G^\ast}(\Omega)$ by \eqref{ass-op} and Lemma~\ref{lem-H}. Therefore, we pass to the weak limit $h_i \rightharpoonup h$ in $W^{1,G}(\Omega)$ to get
\begin{equation}\label{eq-var-h}
\int_\Omega \cA(x, \nabla \varphi) \cdot \nabla (\varphi - h) \dx \geq 0.
\end{equation}
By a standard density argument, \eqref{eq-var-h} holds for all $\varphi \in \mathcal{K}_{\psi, 0}(\Omega)$. Let $\eta \in \mathcal{K}_{\psi, 0}(\Omega)$ be arbitrary. For any $t \in (0, 1]$, the convex combination $\varphi_t := h + t(\eta - h)$ belongs to $\mathcal{K}_{\psi, 0}(\Omega)$. Substituting $\varphi_t$ into \eqref{eq-var-h} and dividing by $t$ yields
\begin{equation*}
\int_\Omega \cA(x, \nabla (h + t(\eta - h))) \cdot \nabla (\eta - h) \dx \geq 0.
\end{equation*}
Letting $t \to 0^+$, the continuity of $\cA(x, \cdot)$ implies that $h$ solves the variational inequality for the $\mathcal{K}_{\psi, 0}(\Omega)$-obstacle problem. Since $v$ is the unique continuous solution to this problem, we conclude that $h \equiv v$.

Now, we prove that $R^\psi \equiv v$. To prove $R^\psi \leq v$, we observe that $v \in \cS_{\cA}(\Omega)$ by Lemma~\ref{lem-supersol-cS}. Moreover, since $v \geq \psi$ a.e.\ in $\Omega$, it follows from the continuity of $v$ and $\psi$ that $v \geq \psi$ everywhere in $\Omega$. This implies $v \in \Phi^\psi$. Therefore, $R^\psi \leq v$ in $\Omega$.

To prove $v \leq R^\psi$, we fix $w \in \Phi^\psi$. By a similar argument used above, we can deduce $h_i \leq w$ a.e.\ in $\Omega_i$, and by zero extension, $h_i \leq w$ a.e.\ in $\Omega$. Taking the pointwise limit $i \to \infty$ gives $h \leq w$ a.e.\ in $\Omega$. Since $h\equiv v$, we obtain $v \leq w$ a.e.\ in $\Omega$. 
Since $v$ is continuous and $w$ is lower semicontinuous, we have $v \leq w$ everywhere in $\Omega$. Taking the pointwise infimum over $w \in \Phi^\psi$ yields $v \leq R^\psi$ everywhere in $\Omega$. Therefore, $R^\psi \equiv v$.

Finally, we observe that, since $v \leq R^\psi$ everywhere, $\hat{v} \leq \hat{R}^\psi$ everywhere in $\Omega$. Therefore, we arrive at
\begin{equation*}
v = \hat{v} \leq \hat{R}^\psi \leq R^\psi \leq v \quad \text{everywhere in } \Omega.
\end{equation*}
Consequently, $R^\psi \equiv \hat{R}^\psi \equiv v$, completing the proof.
\end{proof}

Next, we prove that the balayage $\hat{R}^u_U$ equals the r\'eduite $R^u_U$ and can be approximated by solutions to obstacle problems with smooth obstacles. Note that we do not need the boundedness of $u$ in the following lemma. We refer the reader to Lemma~\ref{lem-bal-cpt} for the analogous approximation of $R^1_K$ when $K$ is a compact set.

\begin{lemma}\label{lem-bal-open}
Suppose that $u \in \cS_{\cA}(\Omega)$ is nonnegative. Let $U \Subset \Omega$ be an open set. If $\{\psi_i\} \subset C^\infty_c(U)$ is a sequence of functions such that $\psi_i \nearrow u\chi_U$ pointwise in $\Omega$ as $i \to \infty$, then $\hat{R}^u_U=R^u_U=\lim_{i \to \infty} R^{\psi_i}$ pointwise in $\Omega$.
\end{lemma}

\begin{proof}
By the definition of the r\'eduite, the condition $\psi_i \leq \psi_{i+1} \leq u\chi_U$ implies that $\{R^{\psi_i}\}$ is an increasing sequence bounded above by $R^u_U$. Let $v:= \lim_{i \to \infty} R^{\psi_i}$. It is clear that $v \leq R^u_U$ everywhere in $\Omega$.

To prove the reverse inequality, note that each $R^{\psi_i}=\hat{R}^{\psi_i}$ is a continuous $\cA$-superharmonic function in $\Omega$ by virtue of Lemmas~\ref{lem-supersol-cS}, \ref{lem-R-obs}, and Proposition~\ref{prop-known}. By Proposition~\ref{prop-conv-super}, their increasing limit $v$ is also an $\cA$-superharmonic function in $\Omega$. Since $v \geq R^{\psi_i} \geq \psi_i$ for all $i$, we have $v \geq u\chi_U$ in $\Omega$, which implies that $v \in \Phi^u_U$. Consequently, we obtain $v \geq R^u_U$.

Finally, since $R^u_U$ is the supremum of a sequence of continuous functions, it is lower semicontinuous in $\Omega$. This immediately yields $\hat{R}^u_U=R^u_U$, completing the proof.
\end{proof}

We are in a position to prove the main theorem of this section.

\begin{proof}[Proof of Theorem~\ref{thm-obs-open}]
We may assume that $u \geq 0$. Let $\{\psi_i\} \subset C^\infty_c(U)$ be a sequence of functions such that $0\leq \psi_i \nearrow u\chi_U$ pointwise in $\Omega$ as $i \to \infty$. Then, Lemma~\ref{lem-R-obs} shows that each $v_i := R^{\psi_i} = \hat{R}^{\psi_i}$ is a continuous solution to the $\mathcal{K}_{\psi_i, 0}(\Omega)$-obstacle problem. Moreover, by Lemma~\ref{lem-bal-open},
\begin{equation*}
v:= \hat{R}^{u}_U = R^{u}_U = \lim_{i\to\infty} v_i \quad\text{pointwise in }\Omega.
\end{equation*}

Take a cutoff function $\eta \in C^\infty_c(\Omega)$ such that $0 \leq \eta \leq 1$ in $\Omega$ and $\eta=1$ on $U$. Define $\psi= u \eta$. By Lemma~\ref{lem-cS-supersol}, the boundedness of $u$ ensures that $u \in W^{1,G}_{\mathrm{loc}}(\Omega)$, and hence $\psi \in W^{1,G}_0(\Omega)$. Since $\psi_i \leq u\chi_U \leq u\eta=\psi$ everywhere in $\Omega$, it follows that $\psi \in \mathcal{K}_{\psi_i, 0}(\Omega)$. Using $\psi$ as a test function in the variational inequality \eqref{eq-obs} for $v_i$, and proceeding exactly as in the derivation of \eqref{eq-energy-bound-i} in Lemma~\ref{lem-R-obs}, we obtain
\begin{equation*}
\int_{\Omega} G(|\nabla v_i|) \dx \leq C \int_{\Omega} G(|\nabla \psi|) \dx.
\end{equation*}
This uniform energy bound implies that the sequence $\{v_i\}$ is bounded in $W^{1,G}_0(\Omega)$. Therefore, $v_i$ converges weakly to $v$ in $W^{1,G}_0(\Omega)$, which ensures $v \in W^{1,G}_0(\Omega)$. Also, $v=\lim_{i \to \infty}v_i \geq \lim_{i \to \infty} \psi_i=u\chi_U$ in $\Omega$. Thus, it suffices to prove that $v$ satisfies the variational inequality \eqref{eq-obs}.

Let $\varphi \in \mathcal{K}_{u\chi_U, 0}(\Omega)$. Then $\varphi \geq u\chi_U \geq \psi_i$, which means $\varphi \in \mathcal{K}_{\psi_i, 0}(\Omega)$. Thus, the variational inequality \eqref{eq-obs} for $v_i$ yields
\begin{equation*}
\int_{\Omega} \cA(x,\nabla v_i) \cdot \nabla (\varphi - v_i) \dx \geq 0.
\end{equation*}
By the monotonicity of the operator $\cA(x, \cdot)$, we have
\begin{equation*}
\int_{\Omega} \cA(x, \nabla \varphi) \cdot \nabla (\varphi - v_i) \dx \geq \int_{\Omega} \cA(x, \nabla v_i) \cdot \nabla (\varphi - v_i) \dx \geq 0.
\end{equation*}
Passing to the weak limit as $i \to \infty$, we obtain
\begin{equation}\label{eq-var-v}
\int_\Omega \cA(x, \nabla \varphi) \cdot \nabla (\varphi - v) \dx \geq 0 \quad \text{for all } \varphi \in \mathcal{K}_{u\chi_U, 0}(\Omega).
\end{equation}
Arguing in exactly the same way as in the proof of Lemma~\ref{lem-R-obs}, we can derive  the variational inequality for $v$ from \eqref{eq-var-v}, as desired.
\end{proof}

For unbounded $u \in \cS_{\cA}(\Omega)$, the balayage $\hat{R}^u_U$ may not be the solution to the $\mathcal{K}_{u\chi_U, 0}(\Omega)$-obstacle problem in general, as it does not necessarily belong to $W^{1,G}_0(\Omega)$. However, its truncations do belong to $W^{1,G}_0(\Omega)$.

\begin{lemma}\label{lem-R-boundary}
Suppose $u\in \cS_{\cA}(\Omega)$ is nonnegative and $U\Subset\Omega$ is an open set. Then $T_\ell(\hat{R}^u_U) \in W_0^{1,G}(\Omega)$ for all $\ell>0$.
\end{lemma}

\begin{proof}
For each integer $k>0$, we define $u_k:=\min\{u, k\} \in \cS_{\cA}(\Omega)$. Let $v_k:=\hat{R}^{u_k}_U$ and $v:= \hat{R}^u_U$. Since $u_k$ is bounded, Theorem~\ref{thm-obs-open} shows that $v_k$ is the solution to the $\mathcal{K}_{u_k \chi_U, 0}(\Omega)$-obstacle problem. In particular, $v_k \in W^{1,G}_0(\Omega)$, and hence $T_\ell(v_k) \in W^{1,G}_0(\Omega)$ for all $\ell>0$.

We claim that the functions $T_\ell(v_k)$ satisfy the uniform energy bound
\begin{equation}\label{eq-energy-bd-vk}
\int_{\Omega} G(|\nabla T_\ell(v_k)|) \dx \leq C
\end{equation}
for some $C>0$ independent of $k$. Indeed, if we take $\eta \in C^\infty_c(\Omega)$ such that $\eta=1$ in a neighborhood $V$ of $\overline{U}$ and $0 \leq \eta \leq 1$ in $\Omega$, then $\varphi_k:= (1-\eta)T_\ell(v_k) \in W^{1,G}_0(\Omega \setminus \overline{U})$. Since $v_k$ is $\cA$-harmonic in $\Omega \setminus \overline{U}$ by Proposition~\ref{prop-bal-u-E}, we have
\begin{align*}
\int_{\Omega \setminus \overline{U}} (1-\eta) \cA(x, \nabla v_k) \cdot \nabla T_\ell(v_k) \dx
&= \int_{\Omega \setminus \overline{U}} T_\ell(v_k) \cA(x, \nabla v_k) \cdot \nabla \eta \dx \\
&\leq \ell \|\nabla \eta\|_{\infty} \int_{\supp \eta \setminus V} |\cA(x, \nabla v_k)| \dx.
\end{align*}

Note that $\supp\eta \setminus V$ is a compact subset of $\Omega \setminus \overline{U}$. Let $W$ be any open set such that $\supp\eta\setminus V \Subset W \Subset \Omega \setminus \overline{U}$. Since $v$ is continuous in $\Omega \setminus \overline{U}$, we have $M:= \max_{\overline{W}}v < \infty$. In particular, $v_k \leq v \leq M$ in $W$ for all $k$, i.e.,\ $\{v_k\}$ is a sequence of uniformly bounded $\cA$-harmonic functions in $W$. It thus follows from Lemma~\ref{lem-uniform-bound} that $\{\cA(x, \nabla v_k)\}$ is uniformly bounded in $L^{G^{\ast}}(\supp \eta \setminus V)$. This implies that
\begin{equation}\label{eq-energy-bd-vk1}
c_1^{\cA} \int_{\Omega \setminus \supp\eta} G(|\nabla T_\ell(v_k)|) \dx \leq \int_{\Omega \setminus \overline{U}} (1-\eta) \cA(x, \nabla v_k) \cdot \nabla T_\ell(v_k) \dx \leq C
\end{equation}
for some $C>0$ independent of $k$. On the other hand, since $T_\ell(v_k)$ is bounded by $\ell$, Lemma~\ref{lem-cS-supersol} shows that it is an $\cA$-supersolution of \eqref{eq-sec2-new}. Thus, again by Lemma~\ref{lem-uniform-bound}, $\{\nabla T_\ell(v_k)\}_k$ is uniformly bounded in $L^G_{\mathrm{loc}}(\Omega)$. In particular,
\begin{equation}\label{eq-energy-bd-vk2}
\int_{\supp \eta} G(|\nabla T_\ell(v_k)|) \dx \leq C
\end{equation}
with $C>0$ independent of $k$. The claim \eqref{eq-energy-bd-vk} now follows from \eqref{eq-energy-bd-vk1} and \eqref{eq-energy-bd-vk2}.

Next, we prove that $v_k \nearrow v$. It is clear by definition that $v_k \leq v_{k+1} \leq v$ for all $k$. Let $w:= \lim_{k \to \infty}v_k$, then $w \leq v$. For the reverse inequality, we observe that $w$ is $\cA$-superharmonic in $\Omega$ by Proposition~\ref{prop-conv-super}. Moreover, since $v_k \geq u_k$ in $U$ for all $k$ by Proposition~\ref{prop-bal-u-E}, taking the limit $k \to \infty$ yields $w \geq u$ in $U$. This implies that $v \leq R^u_U \leq w$. Consequently, $w\equiv v$, and therefore $v_k \nearrow v$ pointwise in $\Omega$.

The uniform bound \eqref{eq-energy-bd-vk} implies that $\{T_\ell(v_k)\}$ is bounded in $W^{1,G}_0(\Omega)$. Since $v_k\nearrow v$ pointwise, we have $T_\ell(v_k)\to T_\ell(v)$ pointwise in $\Omega$. Passing to a weakly convergent subsequence in $W^{1,G}_0(\Omega)$, the weak limit must coincide with $T_\ell(v)$. Hence $T_\ell(v)\in W^{1,G}_0(\Omega)$.
\end{proof}

\subsection{Riesz measures of \texorpdfstring{$\cA$}{A}-superharmonic functions}\label{sec4-5}

An $\cA$-superharmonic function is defined by lower semicontinuity and comparison principles, and therefore no a priori Sobolev regularity or integrability is available from the definition. In particular, its distributional gradient may not be a well-defined function. To handle this, we use the generalized gradient introduced in \eqref{eq-generalized-grad}, as we did in the definition of renormalized solution. If $u \in \cS_{\cA}(\Omega)$, then $\min\{u, k\} \in W^{1,G}_{\rm loc}(\Omega)$ for every $k>0$ by Lemma~\ref{lem-cS-supersol}. Hence $T_k(u)\in W^{1,G}_{\rm loc}(\Omega)$. Therefore, $\nabla u$ can be understood as the generalized gradient $Z_u$.

The next lemma provides the integrability and Sobolev regularity of $\cA$-superharmonic functions.

\begin{lemma}\label{lem-super-int}
If $u\in\cS_{\cA}(\Omega)$, then $g(|\nabla u|)\in L^1_{\rm loc}(\Omega)$ and $g(|u|)\in L^m_{\rm loc}(\Omega)$ for any $m<\frac{n}{n-1}$.
\end{lemma}

\begin{proof}
Let $U, V$ be open sets such that $U \Subset V \Subset \Omega$. Since $u$ is locally bounded below in $\Omega$, we may assume without loss of generality that $u \geq 0$ in $V$. By Lemma~\ref{lem-R-boundary}, the balayage $v:= \hat{R}^u_U(V)$ satisfies $v_k:= T_{k}(v) \in W^{1,G}_0(V)$ for all $k > 0$.

We set, for any positive integer $i$,
\begin{equation*}
a_i := \int_{V \cap \{i-1 \leq v<i\}} \cA(x, \nabla v) \cdot \nabla v \dx.
\end{equation*}
Since $v_{i+1} \in W^{1,G}_0(V)$ is an $\cA$-supersolution in $V$ and $\varphi_i:=(1-|v-i|)_+ \in W^{1,G}_0(V)$, we have
\begin{align*}
0 &\leq \int_{V} \cA(x, \nabla v_{i+1}) \cdot \nabla \varphi_i \dx \\
&= \int_{V \cap \{i-1 \leq v < i\}} \cA(x, \nabla v) \cdot \nabla v \dx - \int_{V \cap \{i \leq v < i+1\}} \cA(x, \nabla v) \cdot \nabla v \dx \\
&= a_i - a_{i+1},
\end{align*}
namely, the sequence $\{a_i\}$ is decreasing. By using the structural assumption \eqref{ass-op}, we have
\begin{equation*}
c_1^{\cA} \int_{V \cap \{v<k\}} G(|\nabla v|) \dx \leq \sum_{i=1}^k a_i \leq k a_1.
\end{equation*}
Since $v=u$ in $U$ by Proposition~\ref{prop-bal-u-E}, this proves
\begin{equation}\label{eq-grad-k}
\int_{U \cap \{u<k\}} G(|\nabla u|)\dx \leq \int_{V \cap \{v<k\}} G(|\nabla v|) \dx \leq Ck \quad\text{for all integers }k \geq 1.
\end{equation}
The standard level-set argument used in Lemma~\ref{lem-int-rem} yields $g(|\nabla u|)\in L^1_{\rm loc}(\Omega)$ and $g(|u|)\in L^m(U)$ for all $m<\frac{n}{n-1}$; see \cite[Lemmas~4.1, 4.5, and Remark~4.3]{CM17}.
\end{proof}

Lemma~\ref{lem-super-int} allows us to associate a Riesz measure with an $\cA$-superharmonic function.

\begin{proposition}
    \label{prop-Riesz}
Let $u$ be an $\cA$-superharmonic function in $\Omega$. Then there exists a unique nonnegative Radon measure $\mu[u]$ satisfying \eqref{eq:main} in the sense of distributions, that is,
\begin{equation}\label{eq-riesz}
\int_\Omega \cA(x,\nabla u) \cdot \nabla \varphi \dx=\int_\Omega \varphi \dmu[u] \quad\text{for all }\varphi\in C^\infty_c(\Omega).
\end{equation}
The measure $\mu[u]$ is called the \emph{Riesz measure} of $u$.
\end{proposition}

\begin{proof}
By Lemma~\ref{lem-super-int}, $u$ is finite a.e. and $g(|\nabla u|)\in L^1_{\rm loc}(\Omega)$. Hence $\cA(x,\nabla u)\in L^1_{\rm loc}(\Omega;\rn)$ by the growth condition \eqref{ass-op}. Thus the functional
\begin{equation*}
F(\varphi):=\int_\Omega \cA(x,\nabla u)\cdot\nabla\varphi\dx, \quad \varphi\in C^\infty_c(\Omega),
\end{equation*}
is well defined.

We claim that $F$ is a positive distribution. Let $\varphi\in C^\infty_c(\Omega)$ be nonnegative, and choose an open set $U\Subset\Omega$ such that $\supp\varphi\Subset U$. Since $u$ is locally bounded below, there exists a constant $C_U>0$ such that $v:=u+C_U \geq 0$ in $U$. Moreover, $v$ is $\cA$-superharmonic in $U$. By \cite[Proposition~3.14]{CGZG24}, there exists a nonnegative Radon measure $\mu_U$ on $U$ such that
\begin{equation*}
\int_U \cA(x,\nabla v)\cdot\nabla\varphi\dx = \int_U \varphi \dmu_U.
\end{equation*}
Since $\nabla v=\nabla u$, we obtain
\begin{equation*}
F(\varphi)=\int_U \varphi \dmu_U \geq 0.
\end{equation*}
Thus $F$ is a positive distribution.

By the Riesz representation theorem for positive distributions, there exists a unique nonnegative Radon measure $\mu[u]$ on $\Omega$ such that \eqref{eq-riesz} holds.
\end{proof}

\subsection{Riesz measure estimates and capacitary potentials}\label{sec4-6}

This section is devoted to the Riesz measure estimates needed for the capacitary analysis of balayage functions. We first compare the total masses of Riesz measures under pointwise ordering, and then identify the $\cA$-potential of a compact set with the capacitary potential.

Throughout this section, $\mu[u]$ denotes the Riesz measure of $u$, either in the sense of Proposition~\ref{prop-Riesz} when $u$ is $\cA$-superharmonic, or in the standard weak sense when $u \in W^{1,G}_{\rm loc}(\Omega)$ is an $\cA$-supersolution.

\begin{lemma}\label{lem-total-mass1}
Assume that $\Omega$ is a bounded domain. Let $v_1, v_2 \in \cS_{\cA}(\Omega) \cap W^{1,G}_0(\Omega) \cap L^\infty(\Omega)$ be such that $0\leq v_1 \leq v_2$ in $\Omega$. Then
\begin{equation*}
\mu[v_1](\Omega) \leq \mu[v_2](\Omega).
\end{equation*}
\end{lemma}

\begin{proof}
For $\varepsilon>0$, we define $w_\varepsilon:= \min\{v_1+\varepsilon, v_2\}$, which is an $\cA$-supersolution of \eqref{eq-sec2-new}. Since $w_\varepsilon \searrow v_1$ pointwise in $\Omega$ as $\varepsilon \to 0$, Lemma~\ref{lem-mea-weakcon} gives the weak convergence $\mu[w_\varepsilon] \rightharpoonup \mu[v_1]$ as $\varepsilon \to 0$. Hence, by the lower semicontinuity of mass on open sets,
\begin{equation}\label{eq-v1-w-eps}
\mu[v_1](\Omega) \leq \liminf_{\varepsilon \to 0} \mu[w_\varepsilon](\Omega).
\end{equation}
Now, for $0<k<\varepsilon$, we define the function
\begin{equation*}
\varphi_k = \tfrac{1}{k} \min\{v_2, k\} \in W^{1, G}_0(\Omega).
\end{equation*}
On the set $\{v_2<k\}$, since $v_1+\varepsilon \geq \varepsilon> k>v_2$, we have $w_\varepsilon=v_2$, and consequently $\nabla w_\varepsilon = \nabla v_2$ a.e.\ on $\{v_2<k\}$. Testing the equations \eqref{eq-riesz} for $w_\varepsilon$ and $v_2$ with $\varphi_k$ yields
\begin{align*}
\int_\Omega \varphi_k \dmu[w_\varepsilon]
&= \int_\Omega \cA(x, \nabla w_\varepsilon) \cdot \nabla \varphi_k \dx = \frac{1}{k} \int_{\Omega \cap \{v_2<k\}} \cA(x, \nabla w_\varepsilon) \cdot \nabla v_2 \dx \\
&= \frac{1}{k} \int_{\Omega \cap \{v_2<k\}} \cA(x, \nabla v_2) \cdot \nabla v_2 \dx = \int_{\Omega} \cA(x, \nabla v_2) \cdot \nabla \varphi_k \dx = \int_\Omega \varphi_k \dmu[v_2].
\end{align*}

By the strong minimum principle (Proposition~\ref{prop-min}), $v_2 \equiv 0$ or $v_2>0$ in $\Omega$. The former case is trivial, so we may assume that $v_2>0$ in $\Omega$. Then, $\varphi_k \nearrow 1$ in $\Omega$ as $k \searrow 0$. By the monotone convergence theorem,
\begin{equation}\label{eq-w-eps}
\mu[w_\varepsilon](\Omega) = \mu[v_2](\Omega).
\end{equation}
The desired inequality follows from \eqref{eq-v1-w-eps} and \eqref{eq-w-eps}.
\end{proof}

\begin{lemma}\label{lem-total-mass2}
Assume that $\Omega$ is a bounded domain. Let $v \in \cS_{\cA}(\Omega)$ be nonnegative in $\Omega$ and $\cA$-harmonic in $\Omega \setminus K$ for some compact subset $K$ of $\Omega$. Suppose that $T_k(v) \in W^{1,G}_0(\Omega)$ for any $k>0$. Then for any $\lambda>0$ it holds
\begin{equation*}
\mu[v](\Omega) = \mu[\min\{v, \lambda\}](\Omega).
\end{equation*}
\end{lemma}

\begin{proof}
We may assume that $v \not\equiv 0$. By the strong minimum principle (Proposition~\ref{prop-min}) and the compactness of $K$, we have $v \geq \min_K v =: \delta >0$ on $K$. Let $w:= \min\{v, \lambda\} \in \cS_{\cA}(\Omega) \cap W^{1,G}_0(\Omega)$. Note that $w \geq \min\{\delta, \lambda\}>0$ on $\supp\mu[w] \subset K \cup \{v \geq \lambda\}$. For $0<k<\min\{\delta, \lambda\}$, we define
\begin{equation*}
\varphi_k = \tfrac{1}{k} \min\{v, k\} = \tfrac{1}{k} T_k(v) \in W^{1,G}_0(\Omega).
\end{equation*}
Then, $\varphi_k=1$ on the open set $\{v>k\}$. Since $\supp \mu[w] \subset K \cup \{v \geq \lambda\} \subset \{v>k\}$ and $\mu[w] \in (W^{1,G}_0(\Omega))'$, we have
\begin{equation*}
\mu[w](\Omega) = \int_\Omega \cA(x, \nabla w) \cdot \nabla \varphi_k \dx = \frac{1}{k} \int_{\Omega \cap \{v<k\}} \cA(x, \nabla w) \cdot \nabla v \dx.
\end{equation*}
Note, however, that $\mu[v] \notin (W^{1,G}_0(\Omega))'$ in general. Instead, we take a cutoff function $\eta \in C^\infty_c(\Omega)$ with $\eta=1$ in an open neighborhood of $K$. Since $\supp\mu[v] \subset K$, we obtain
\begin{equation*}
\mu[v](\Omega) = \int_\Omega \eta \dmu[v] = \int_\Omega \cA(x, \nabla v) \cdot \nabla \eta \dx = \int_{\Omega \setminus K} \cA(x, \nabla v) \cdot \nabla \eta \dx.
\end{equation*}
Since $\eta=1$ in a neighborhood of $K$ and $\varphi_k=1$ in the open set $\{v>k\} \supset K$, the difference $\varphi_k-\eta$ vanishes in a neighborhood of $K$. This implies $\varphi_k-\eta \in W^{1,G}_0(\Omega \setminus K)$. Since $v$ is $\cA$-harmonic in $\Omega \setminus K$, we obtain
\begin{equation*}
\int_{\Omega \setminus K} \cA(x, \nabla v) \cdot \nabla \eta \dx = \int_{\Omega \setminus K} \cA(x, \nabla v) \cdot \nabla \varphi_k \dx = \frac{1}{k} \int_{\Omega \cap \{v<k\}} \cA(x, \nabla v) \cdot \nabla v \dx.
\end{equation*}
Since $v=w$ on $\Omega \cap \{v<k\}$, we deduce that $\mu[w](\Omega) = \mu[v](\Omega)$.
\end{proof}

In the standard $p$-growth case, superharmonic functions can be rescaled to obtain suitable test functions for the capacity. However, since the class of $\cA$-superharmonic functions lacks scale invariance, we bypass this difficulty by deriving estimates for rescaled obstacles instead. Our approach here is inspired by the techniques elaborated for a parabolic problem \cite{Para-polar}.

\begin{proposition}
    \label{prop-64}
Let $K\subset \Omega$ be a compact set. Then, there exists a constant $C>0$ such that 
for any $\lambda>1$ we have
\begin{equation*}
\mu[\hat{R}^1_{K}](\Omega) \leq C \left( \lambda^{1-p}+\lambda^{\frac{1+q}{1-q}} \right) \mu[\hat{R}^{\lambda}_{K}](\Omega).
\end{equation*}
\end{proposition}

\begin{proof}
For $\ell>0$, we define the function $u_\ell:= \hat{R}^\ell_K$. Let $\eta \in C^\infty_c(\Omega)$ be such that $\eta=1$ on $K$. By Proposition~\ref{prop-bal-K}, $u_\ell$ is $\cA$-harmonic in $\Omega \setminus K$ and $u_\ell - \ell\eta \in W^{1,G}_0(\Omega \setminus K)$. Thus,
\begin{equation*}
\int_{\Omega \setminus K} \cA(x, \nabla u_\ell) \cdot \nabla (u_\ell - \ell \eta) \dx = 0.
\end{equation*}
On the other hand, it follows from Lemma~\ref{lem-ae} that $u_\ell = \ell = \ell \eta$ a.e.\ on $K$, and hence the above integral can be extended to all of $\Omega$. Moreover, since $\supp\mu[u_\ell] \subset K$,  by using \eqref{eq-riesz} we have
\begin{equation*}
\ell \mu[u_\ell](\Omega) = \int_\Omega \ell \eta \dmu[u_\ell] = \int_\Omega \cA(x, \nabla u_\ell) \cdot \nabla (\ell \eta) \dx = \int_\Omega \cA(x, \nabla u_\ell) \cdot \nabla u_\ell \dx.
\end{equation*}
Setting $\ell=1$ and $\ell=\lambda$, and applying the assumptions \eqref{ass-op} and \eqref{eq:Gg}, we obtain
\begin{equation}\label{eq-mu}
\mu[u_1](\Omega) \leq qc_2^{\cA} \int_\Omega G(|\nabla u_1|) \dx \quad\text{and}\quad c_1^{\cA} \int_\Omega G(|\nabla u_\lambda|) \dx \leq \lambda \mu[u_\lambda](\Omega).
\end{equation}

Let
\begin{align*}
\varphi_1&:=\lambda u_\lambda - \lambda^2 u_1 = \lambda (u_\lambda - \lambda \eta) - \lambda^2 (u_1-\eta) \in W^{1,G}_0(\Omega \setminus K), \\
\varphi_\lambda&:=u_1 - \lambda^{-1} u_\lambda = u_1-\eta - \lambda^{-1}(u_\lambda - \lambda \eta) \in W^{1,G}_0(\Omega \setminus K).
\end{align*}
Since $u_1$ and $u_\lambda$ are $\cA$-harmonic in $\Omega \setminus K$, and since $\varphi_1=\varphi_\lambda=0$ a.e.\ on $K$, we get
\begin{equation*}
\int_{\Omega} \cA(x, \nabla u_1) \cdot \nabla \varphi_1 \dx = 0 \quad\text{and}\quad \int_{\Omega} \cA(x, \nabla u_\lambda) \cdot \nabla \varphi_\lambda \dx = 0.
\end{equation*}
Adding these inequalities and applying the structural assumption \eqref{ass-op}, we obtain
\begin{align*}
		c^\cA_1 \lambda^2 \int_{\Omega}G(|\nabla u_1|) \dx &\leq \lambda^2 \int_{\Omega}\cA(x, \nabla u_1)\cdot \nabla u_1 \dx \\	
		& = \lambda \int_{\Omega}\cA(x,\nabla u_1)\cdot \nabla u_{\lambda} \dx+\int_{\Omega} \cA(x,\nabla u_{\lambda})\cdot \nabla u_1 \dx-\frac{1}{\lambda}\int_{\Omega} \cA(x,\nabla u_\lambda)\cdot \nabla u_\lambda \dx
		\\
        & \le \lambda c^{\cA}_{2} \int_{\Omega} g( |\nabla u_1|) |\nabla u_\lambda| \dx +c^{\cA}_2\int_{\Omega}g( |\nabla u_{\lambda}|) |\nabla u_1| \dx.
	\end{align*}
	By using Lemma~\ref{lem-barg} and \eqref{eq:Gg}, we have for $\varepsilon \in (0,1)$
    \begin{equation*}
\lambda^2 g(|\nabla u_1|) \frac{|\nabla u_\lambda|}{\lambda} \leq \lambda^2 \left( q\varepsilon G(|\nabla u_1|) + q\varepsilon^{1-q} \lambda^{-p} G(|\nabla u_\lambda|) \right)
    \end{equation*}
    and
    \begin{align*}
    \lambda^{-\frac{2}{q-1}} g(|\nabla u_\lambda|) \lambda^{\frac{2}{q-1}} |\nabla u_1|
    &\leq \lambda^{-\frac{2}{q-1}} \left( q\varepsilon^{-\frac{1}{p-1}} G(|\nabla u_\lambda|) + g\left( \varepsilon^{\frac{1}{p-1}} \lambda^{\frac{2}{q-1}} |\nabla u_1| \right) \lambda^{\frac{2}{q-1}} |\nabla u_1| \right) \\
    &\leq q\varepsilon^{-\frac{1}{p-1}} \lambda^{-\frac{2}{q-1}} G(|\nabla u_\lambda|) + q\varepsilon \lambda^2 G(|\nabla u_1|).
    \end{align*}
By taking $\varepsilon$ sufficiently small, we can absorb the terms involving $G(|\nabla u_1|)$ into the left-hand side. It thus follows that
\begin{equation}\label{eq-mu-energy}
\int_{\Omega} G(|\nabla u_1|) \dx \le C \left(\lambda^{-p}+\lambda^{-\frac{2q}{q-1}} \right)\int_{\Omega} G(|\nabla u_\lambda|) \dx.
\end{equation}
Combining \eqref{eq-mu} and \eqref{eq-mu-energy} yields the desired inequality.
\end{proof}

The next lemma shows that $R^1_K$ can be approximated by solutions to obstacle problems with smooth obstacles.

\begin{lemma}\label{lem-bal-cpt}
Let $K$ be a compact subset of $\Omega$. Let $\{\psi_i\} \subset C_c^\infty(\Omega)$ be a decreasing sequence such that $0\leq \psi_i\searrow \chi_K$ pointwise in $\Omega$ and $\|\psi_i\|_\infty \to 1$ as $i \to \infty$.
Then, $R^1_K = \lim_{i \to \infty} R^{\psi_i}$ pointwise in $\Omega$.
\end{lemma}

\begin{proof}
By the definition of the r\'eduite, the condition $\psi_i \geq \psi_{i+1} \geq \chi_K$ implies that $\{R^{\psi_i}\}$ is a decreasing sequence bounded below by $R^1_K$. Let $v:=\lim_{i \to \infty} R^{\psi_i}$. Then, it is clear that $v \geq R^1_K$.

To prove the reverse inequality, let $w \in \Phi^1_K$. Since $w$ is lower semicontinuous, for any given $\varepsilon \in (0,1/2)$ the superlevel set $U_\varepsilon=\{x \in \Omega: w(x)>1-\varepsilon\}$ is an open subset of $\Omega$ that contains $K$. Since $\psi_i \searrow 0$ on the compact set $K_1:=\supp\psi_1 \setminus U_\varepsilon$, Dini's theorem shows that this convergence is uniform. Thus, there exists an integer $N$ such that $\psi_i<\varepsilon$ on $K_1$ for all $i \geq N$. Moreover, we may assume that $\psi_i \leq 1+\varepsilon$ in $\Omega$ for all $i \geq N$.

Now, we define $w_\varepsilon:=w+2\varepsilon \in \cS_{\cA}(\Omega)$. Then, $w_\varepsilon>1+\varepsilon \geq \psi_i$ in $U_\varepsilon$ and $w_\varepsilon \geq 2\varepsilon > \psi_i$ in $\Omega \setminus U_\varepsilon$ (which includes $K_1$) for all $i \geq N$. This implies $w_\varepsilon \in \Phi^{\psi_i}$, and hence $R^{\psi_i} \leq w_\varepsilon$ everywhere in $\Omega$. Taking the limit as $i \to \infty$, we obtain $v \leq w+2\varepsilon$ in $\Omega$. Since $\varepsilon \in (0, 1/2)$ was arbitrary, we deduce that $v \leq w$. Therefore, taking the infimum over all $w \in \Phi^1_K$ yields $v \leq R^1_K$.
\end{proof}

The following result gives a characterization of capacity of compact sets through the $\cA$-potentials.

\begin{proposition}\label{prop57}
Assume that $\Omega$ is a bounded domain. Let $K$ be a compact subset of $\Omega$. Then 
\begin{equation*}
\widetilde{\mathrm{cap}}_{\cA}(K,\Omega)=\mu\lbrack\hat{R}^1_{K}\rbrack(\Omega).
\end{equation*}
\end{proposition}
	
\begin{proof}
By Proposition~\ref{prop-bal-K}, the Riesz measure $\mu:=\mu[\hat{R}^1_K]$ of $\hat{R}^1_K$ is supported on $K$ and belongs to $(W^{1,G}_0(\Omega))'$. Moreover, by the weak maximum principle (see, e.g., \cite[Corollary~4.16]{Che22Gen}), we have $0 \leq \hat{R}^1_K \leq 1$ and $\mu \geq 0$. Thus, $\mu$ is admissible for the capacity $\widetilde{\mathrm{cap}}_{\cA}(K, \Omega)$, which leads to the inequality
\begin{equation*}
\widetilde{\mathrm{cap}}_{\cA}(K,\Omega) \geq \mu(K).
\end{equation*}

To prove the reverse inequality, we fix an open set $U$ such that $K \subset U \Subset \Omega$. Let $\varepsilon_i \to 0$ and choose a decreasing sequence $\{\psi_i\} \subset C_c^\infty(U)$ such that $\psi_i\to\chi_K$ pointwise in $\Omega$, $\psi_i = 1+\varepsilon_i$ on $K$, and $\psi_i \leq 1+\varepsilon_i$ in $\Omega$. It then follows from Lemma~\ref{lem-bal-cpt} that $R^{\psi_i}$ converges to $R^1_K$ pointwise in $\Omega$. Moreover, Lemma~\ref{lem-ae} shows that $\hat{R}^{\psi_i}=R^{\psi_i}$ and $\hat{R}^1_K=R^1_K$ a.e.\ in $\Omega$, and they are $\cA$-supersolutions of \eqref{eq-sec2-new}. Since $\{\psi_i\}$ is uniformly bounded, so is $\{\hat{R}^{\psi_i}\}$. Therefore, by Lemma~\ref{lem-mea-weakcon}, the sequence of measures $\{\mu[\hat{R}^{\psi_i}]\}$ converges weakly to $\mu$ as $i \to \infty$. Choose $\eta \in C^\infty_c(\Omega)$ such that $\eta=1$ on $\overline{U}$. Since $\supp\mu[\hat{R}^{\psi_i}] \subset \overline{U}$ and $\supp\mu \subset K$ by Propositions~\ref{prop-bal-u-E} and \ref{prop-bal-K}, we have
\begin{equation}\label{eq-conv-mass}
\mu[\hat{R}^{\psi_i}](\Omega) = \mu[\hat{R}^{\psi_i}](\overline{U}) = \int_{\Omega} \eta \dmu[\hat{R}^{\psi_i}] \to \int_\Omega \eta \dmu = \mu(K).
\end{equation}

Let $\nu$ be an admissible measure for $\widetilde{\mathrm{cap}}_{\cA}(K, \Omega)$ associated with $v:=u_\nu$; that is, $v \in W^{1,G}_0(\Omega)$ solving
\begin{equation*}
-\DIV\cA(x,\nabla v)=\nu \quad\text{in }\Omega
\end{equation*}
in the distributional sense, $0\leq v \leq 1$ a.e.\ in $\Omega$, and $0\leq \nu \in (W^{1,G}_0(\Omega))'$ is supported on $K$. Due to Lemmas~\ref{lem-supersol-cS} and \ref{lem-super-liminf}, we may assume that $v \in \cS_{\cA}(\Omega)$ and $0\leq v \leq 1$ in $\Omega$. Since $\hat{R}^{\psi_i}$ is continuous in $\Omega$, $\hat{R}^{\psi_i} >1$ in a neighborhood of $K$, and hence $(v-\hat{R}^{\psi_i})_+ \in W^{1,G}_0(\Omega \setminus K)$. Since $\nu$ is supported on $K$, the function $v$ is an $\cA$-solution in $\Omega \setminus K$. On the other hand, $\hat{R}^{\psi_i}$ is an $\cA$-supersolution in $\Omega \setminus K$. Hence the comparison principle in \cite[Lemma~4.3]{CK21remov} shows that $v \leq \hat{R}^{\psi_i}$ a.e.\ in $\Omega$. It thus follows from the continuity of $\hat{R}^{\psi_i}$ and the lower semicontinuity of $v$ that $v \leq \hat{R}^{\psi_i}$ pointwise in $\Omega$. By Lemma~\ref{lem-total-mass1} and Proposition~\ref{prop-bal-K}, along with the convergence \eqref{eq-conv-mass}, we obtain
\begin{equation*}
\nu(K) \leq \mu[\hat{R}^{\psi_i}](\Omega) \to \mu(K)
\end{equation*}
as $i \to \infty$. To conclude it is enough to take the supremum over all such admissible $\nu$ yielding
\begin{equation*}
\widetilde{\mathrm{cap}}_{\cA}(K,\Omega)\leq \mu(K).\qedhere
\end{equation*}
\end{proof}

Combining Propositions~\ref{prop57} and~\ref{prop-64},
we establish capacity estimates for the super-level sets of $\cA$-superharmonic functions in terms of the Riesz measure of the balayage.

\begin{lemma}\label{lem-cap-pot}
Assume that $\Omega$ is a bounded domain. Let $u \in \cS_{\cA}(\Omega)$ be nonnegative in $\Omega$. Then, there exists a constant $C>0$ such that
\begin{equation}\label{eq-cap-pot-K}
\widetilde{\mathrm{cap}}_{\cA}(\{u>\lambda\} \cap K, \Omega) \leq C \left( \lambda^{1-p}+\lambda^{\frac{1+q}{1-q}} \right) \mu[\hat{R}^u_U](\Omega)
\end{equation}
for any compact set $K$ and any open set $U$ such that $K \subset U \Subset \Omega$, and any $\lambda>1$.
\end{lemma}

\begin{proof}
By the strong minimum principle (Proposition~\ref{prop-min}), we may assume that $u>0$ in $\Omega$. Since $u$ is lower semicontinuous in $\Omega$, there exists an increasing sequence of smooth functions $\{\phi_i\}$ such that $\phi_i \to u$ pointwise in $\Omega$. Let $\eta \in C^\infty_c(U)$ be such that $\eta=1$ in a neighborhood of $K$ and $0\leq \eta \leq 1$, and define $\psi_i=\phi_i \eta$. Then, $\{\psi_i\}$ is an increasing sequence of smooth functions such that $\psi_i=\phi_i \nearrow u$ on $K$ and $\psi_i \leq u \chi_{U}$. Let $u_i=\hat{R}^{\psi_i}$ be the unique continuous solution to the $\mathcal{K}_{\psi_i, 0}(\Omega)$-obstacle problem; see Lemma~\ref{lem-R-obs}. Observe that
\begin{equation*}
\{u>\lambda\} \cap K \subset \textstyle\bigcup_{i=1}^\infty K_i, \quad\text{where }K_i:=K\cap \{u_i \geq \lambda\}\text{ is a compact subset of }\Omega.
\end{equation*}
Indeed, if $x \in \{u>\lambda\} \cap K$, then $\psi_i(x)>\lambda$ for sufficiently large $i$. Since $R^{\psi_i} \geq \psi_i$ and $\psi_i$ is smooth, taking the lsc-regularization gives $u_i(x) \geq \psi_i(x)$, which implies $x \in K_i$. Moreover, the sequence $\{u_i\}$ is increasing by construction. It thus follows from Lemmas~\ref{lem-cap-limit} and \ref{lem-cap-dual} that
\begin{equation*}
\widetilde{\mathrm{cap}}_{\cA}(\{u>\lambda\} \cap K, \Omega) \leq C \lim_{i \to \infty} \widetilde{\mathrm{cap}}_{\cA}(K_i, \Omega).
\end{equation*}
Furthermore, Propositions~\ref{prop57} and~\ref{prop-64} show that
\begin{equation*}
\widetilde{\mathrm{cap}}_{\cA}(K_i, \Omega) = \mu[\hat{R}^1_{K_i}](\Omega) \leq C \left( \lambda^{1-p}+\lambda^{\frac{1+q}{1-q}} \right) \mu[\hat{R}^{\lambda}_{K_i}](\Omega).
\end{equation*}

Since $\psi_i\leq u\chi_U$, the monotonicity of the r\'eduite and Lemma~\ref{lem-bal-open} imply
\begin{equation*}
u_i=R^{\psi_i}\leq R_U^u=\hat R_U^u.
\end{equation*}
Hence $\hat R_U^u\geq\lambda$ on $K_i=K \cap \{u_i\geq\lambda\}$, so $\hat R_U^u$ is admissible for $R^\lambda_{K_i}$ with the constant function $\lambda$. Therefore, we have
\begin{equation*}
\hat R^\lambda_{K_i}\leq R^\lambda_{K_i}\leq \min\{\hat R_U^u,\lambda\} = T_\lambda(\hat{R}^u_U).
\end{equation*}
Thus Lemma~\ref{lem-total-mass1}, together with Proposition~\ref{prop-bal-K} and Lemma~\ref{lem-R-boundary}, gives
\begin{equation*}
\mu[\hat R^\lambda_{K_i}](\Omega)\leq \mu[\min\{\hat R_U^u,\lambda\}](\Omega).
\end{equation*}

Combining these estimates, we arrive at
\begin{equation*}
\widetilde{\mathrm{cap}}_{\cA}(\{u>\lambda\} \cap K, \Omega) \leq C \left( \lambda^{1-p}+\lambda^{\frac{1+q}{1-q}} \right) \mu[\min\{\hat{R}^u_U, \lambda\}](\Omega).
\end{equation*}
The desired estimate \eqref{eq-cap-pot-K} now follows from Lemma~\ref{lem-total-mass2} applied to $v=\hat{R}^u_U$ with $\overline{U}$ in place of $K$; see Lemma~\ref{lem-R-boundary} and Proposition~\ref{prop-bal-K}.
\end{proof}

\subsection{\texorpdfstring{$\cA$}{A}-polar sets and \texorpdfstring{$G$}{G}-quasicontinuity}\label{sec4-7}

We now prove the capacitary negligibility of $\cA$-polar sets introduced in Definition~\ref{def-polar}, and the $G$-quasicontinuity of $\cA$-superharmonic functions. These results are the main fine potential-theoretic consequences of the capacity estimates obtained in the previous subsection.

\begin{proof}[Proof of Theorem~\ref{th-1}]
Let $E$ be an $\cA$-polar set. Then there exist an open neighborhood $\mathcal{O}$ of $E$ and an $\cA$-superharmonic function $u$ in $\mathcal{O}$ such that $E \subset F$, where $F:=\{x \in \mathcal{O}: u(x)=\infty\}$. Note that $F$ is a Borel set. Since $u$ is lower semicontinuous, we may assume that $u$ is positive in $\mathcal{O}$ (by restricting $\mathcal{O}$ to the open set $\{u>0\}$ if necessary). Take any compact subset $K$ of $F$. Then $K$ intersects only with finitely many components of $\mathcal{O}$, say $\mathcal{O}_1, \dots, \mathcal{O}_N$. Let $K_i:=K \cap \mathcal{O}_i$, then $K_i$ is a compact subset of $\mathcal{O}$. Take a bounded domain $\Omega_i$ and an open set $U_i$ satisfying $K_i \Subset U_i \Subset \Omega_i \subset \mathcal{O}_i$.

Let $\hat{R}^u_{U_i}$ be the balayage of $u$ relative to $U_i$ in $\Omega_i$. Since $K_i \subset \{u>\lambda\}$ for any $\lambda>1$, it follows from Lemmas~\ref{lem-cap-dual} and \ref{lem-cap-pot} that
\begin{equation*}
\mathrm{cap}_{G}(K_i, \Omega_i) \leq C \,\widetilde{\mathrm{cap}}_{\cA}(\{u>\lambda\} \cap K_i, \Omega_i) \leq C \left( \lambda^{1-p}+\lambda^{\frac{1+q}{1-q}} \right) \mu[\hat{R}_{U_i}^u](\Omega_i).
\end{equation*}
Since $\mu[\hat{R}^u_{U_i}]$ is a Radon measure, Proposition~\ref{prop-bal-u-E} ensures that $\mu[\hat{R}^u_{U_i}](\Omega_i)=\mu[\hat{R}^u_{U_i}](\overline{U}_i)<\infty$. Letting $\lambda \to \infty$ shows that $\mathrm{cap}_G(K_i, \Omega_i)=0$, or equivalently, $C_G(K_i)=0$ by Lemma~\ref{lem-cap-zero}. Using the subadditivity of $C_G$, we obtain $C_G(K)=0$. Since $K \subset F$ was arbitrary, it follows from \eqref{eq-Choquet} that $C_G(F)=0$. Therefore, $C_G(E) \leq C_G(F)=0$.
\end{proof}

A property is said to hold \emph{$G$-quasieverywhere} (abbreviated as $G$-q.e.) if it holds everywhere except possibly on a set of $G$-capacity zero.

\begin{definition}\label{def-quasicont}
A function $u$ is said to be \emph{$G$-quasicontinuous} in $\Omega$ if, for every $\varepsilon > 0$, there exists an open set $U \subset \Omega$ with $C_G(U) < \varepsilon$ such that $u$ is finite and continuous on $\Omega \setminus U$.
\end{definition}

It is known \cite[Theorem~17]{BHH18} that every Sobolev function in $W^{1,G}$ admits a $G$-quasicontinuous representative, and that any two such representatives coincide $G$-q.e. We are in a position to prove an analogous result for $\cA$-superharmonic functions.

\begin{proof}[Proof of Theorem~\ref{thm-quasicont}]
It is enough to prove the assertion locally. Let $\Omega' \Subset \Omega$ be a bounded domain and let $V \Subset U \Subset \Omega'$. Since $u$ is locally bounded below, we may assume that $u \geq 0$ in $U$. We first prove that $u$ is $G$-quasicontinuous in $V$.

For any integer $k>0$, the truncated function $u_k:=\min\{u, k\}$ is a bounded $\cA$-superharmonic function, and hence it belongs to $W^{1,G}(V)$ by Lemma~\ref{lem-cS-supersol}. By \cite[Proposition 4.5]{Che22Gen}, we can choose an increasing sequence of continuous $\cA$-superharmonic functions $\{v_i\}$ in $\Omega$ that converges pointwise to $u_k$ as $i \to \infty$.

Since $\{v_i\}$ is a sequence of uniformly bounded $\cA$-supersolutions to \eqref{eq-sec2-new}, Lemma~\ref{lem-uniform-bound} shows that $\{v_i\}$ is uniformly bounded in $W^{1,G}(V)$. Due to \eqref{eq:Gg}, the space $W^{1,G}$ is reflexive, and hence, up to a subsequence, $\{v_i\}$ converges to $u_k$ weakly in $W^{1,G}(V)$. Appealing to Mazur's lemma (see, e.g., \cite[Lemma~1.29]{HKM06}), for each $j \in \mathbb{N}$ we find a convex combination $w_j$ of $\{v_1, \dots, v_j\}$ such that the sequence $\{w_j\}$ converges to $u_k$ strongly in $W^{1,G}(V)$ as $j \to \infty$. Since each $v_i$ is continuous and satisfies $v_i \leq u_k$, their convex combinations $w_j$ are also continuous functions in $V$ satisfying $w_j \leq u_k$. Therefore, by \cite[Lemma~16]{BHH18}, there exists a subsequence of $\{w_j\}$ that converges to $u_k$ uniformly outside an open set of arbitrarily small $G$-capacity in $V$. This precisely means that $u_k$ is $G$-quasicontinuous in $V$.

Now, we show that $u$ is also $G$-quasicontinuous in $V$. To do so, we utilize the estimate in Lemma~\ref{lem-cap-pot}, which was the key ingredient in the proof of Theorem~\ref{th-1}. Let $\hat{R}^u_U$ be the balayage of $u$ relative to $U$ in $\Omega'$. By Lemmas~\ref{lem-cap-dual} and \ref{lem-cap-pot}, we have
\begin{equation*}
\mathrm{cap}_G(\{u>k\} \cap V, \Omega') \leq C\, \widetilde{\mathrm{cap}}_{\cA}(\{u>k\} \cap \overline{V}, \Omega') \leq C \left( k^{1-p}+k^{\frac{1+q}{1-q}} \right) \mu[\hat{R}^u_U](\Omega').
\end{equation*}
Moreover, by \cite[Theorem~27]{BHH18}, we have
\begin{equation*}
C_G(\{u>k\} \cap V) \leq C \mathrm{cap}_G(\{u>k\} \cap V, \Omega').
\end{equation*}
For a given $\varepsilon>0$, we can choose an integer $N$ sufficiently large so that $C_G(\{u>N\} \cap V) \leq \varepsilon/2$. Since $u_N$ is $G$-quasicontinuous in $V$, there exists an open set $U_N \subset V$ with $C_G(U_N)<\varepsilon/2$ such that $u_N$ is continuous on $V \setminus U_N$. Let $U=(\{u>N\} \cap V) \cup U_N$. Since $u$ is lower semicontinuous, $U$ is open. Furthermore, by subadditivity, $C_G(U) \leq C_G(\{u>N\} \cap V) + C_G(U_N)<\varepsilon$. On $V \setminus U$, it is clear to see that $u=u_N$ is continuous. This proves that $u$ is $G$-quasicontinuous in $V$.

Finally, an exhaustion argument with capacities $\varepsilon 2^{-j}$ gives the $G$-quasicontinuity of $u$ in $\Omega$.
\end{proof}

\section{Superharmonic functions are renormalized solutions}\label{sec5}

In this section, we present the proof of Theorem~\ref{thm-sup-ren}. Crucial tools for the proof are the pointwise estimates of $\cA$-superharmonic functions in terms of the Wolff potential of their Riesz measures. We recall the following result, whose proof under the additional assumption that $u$ is finite almost everywhere is given in \cite[Theorem~2]{CGZG24}. This assumption can be relaxed by using Theorem~\ref{th-1} and Lemma~\ref{lem-cap-measure}.

\begin{lemma}[Wolff potential estimates]\label{lem-Wolff}
    Let $u$ be a nonnegative $\cA$-superharmonic function in $B_{2r}(x)$, $r \in (0,1)$, with the Riesz measure $\mu=\mu[u]$. Then
    \begin{equation}\label{eq-sec2-wolff}
        \frac{1}{C_{\cW}}\left(\cW_{\mu,r}(x)-r\right)\leq u(x)\leq C_{\cW} \left(\inf_{B_r(x)}u+\cW_{\mu,r}(x)+r\right)
    \end{equation}
for some constant $C_{\cW} \geq 1$, where $\cW_{\mu, r}$ is the \emph{Wolff potential} of $\mu$ defined by
    \begin{equation*}
       \cW_{\mu,r}(x)=\int^{r}_0g^{-1}\left(\frac{\mu(B_\rho(x))}{\rho^{n-1}}\right) \mathrm{d}\rho.
    \end{equation*}
\end{lemma}

We introduce an auxiliary class of functions characterized by two-sided Wolff potential bounds. For $r\in(0,1)$, $L\geq 0$, and $\Omega'\Subset\Omega$, we define a class
\begin{equation*}
    \cS_{\mu,r,L}(\Omega'):=\left\{u: C_{\cW}^{-1}\left( \cW_{\mu,r}(x)-r\right)\leq u(x) \leq L+C_{\cW}\cW_{\mu,r}(x)\text{ for all } x\in \Omega'\right\},
\end{equation*}
where $C_{\cW}\geq1$ is the constant given in Lemma~\ref{lem-Wolff}. As this class will be utilized repeatedly in subsequent arguments, it is convenient to establish a general lemma concerning its properties here.

\begin{lemma}\label{lem-Wolff-S}
Let $u$ be a nonnegative $\cA$-superharmonic function in $\Omega$ with the Riesz measure $\mu=\mu[u]$. Let $\Omega'\Subset \Omega$ and $0<r<\min\{ \dist(\Omega',\partial\Omega)/2, 1\}$. Then $u\in \cS_{\mu,r,L}(\Omega')$ for some $L>0$.
\end{lemma}

\begin{proof}
The Wolff potential estimate~\eqref{eq-sec2-wolff} immediately yields the first inequality in the definition of $\cS_{\mu,r, L}$. To establish the second inequality, it suffices to find an upper bound for $\inf_{B_r(y)}u$ for any  $y\in \Omega'$. This follows directly from Lemma~\ref{lem-super-int}, which guarantees the existence of a constant $\gamma>0$ such that for all $y\in \Omega'$, it holds that
    \begin{equation*}
        \inf_{B_r(y)} u\leq \left(\fint_{B_r(y)}|u|^\gamma\dx\right)^{1/\gamma}\leq C \left(r^{-n} \int_{\Omega'+B_{r}(0)}|u|^\gamma\dx\right)^{1/\gamma}<\infty.
    \end{equation*}
This completes the proof.
\end{proof}

A crucial ingredient in our argument is the fact that the singular measure $\mu_s$ is concentrated on the set $\{u=\infty\}$ not only for renormalized solutions (as noted in Remark~\ref{rmk-ren}), but also for arbitrary $\cA$-superharmonic functions. We establish this in the following proposition.

\begin{proposition}\label{prop-mu-s}
Let $u$ be an $\cA$-superharmonic function in $\Omega$ with the Riesz measure $\mu=\mu[u]$. Then $\mu_s(\{u<\infty\})=0$.
\end{proposition}

Throughout the proof below and in what follows, for a given measure $\mu$ and a $\mu$-measurable set $E \subset \rn$, we denote the restriction of $\mu$ to $E$ by $\mu\lfloor_{E}$.

\begin{proof}
By the definition of $\mu_s$, there exists a Borel set $E\subset\Omega$ of $G$-capacity zero such that $\mu_{s}(\Omega\setminus E)=0$. For each integer $k\geq 1$, we define $E_{k}=E\cap\{u<k\}$. Since
\begin{equation*}
\mu_s(\{u<\infty\}) \leq \mu_s(\Omega\setminus E)+\sum_{k=1}^{\infty}\mu_s(E_k),
\end{equation*}
it suffices to prove that $\mu_s(E_k)=0$ for all $k \geq 1$. By the inner regularity of Radon measures, this reduces to showing that for a fixed $k\geq1$, $\mu_s(K)=0$ for every compact subset $K \subset E_k$.

Fix a compact subset $K$ of $E_k$. Choose $0<r<\min\{\dist(K, \partial \Omega)/2, 1\}$. Since $u$ is locally bounded below, we may choose $a>0$ such that $u+a \geq 0$ in a neighborhood of $\{x: \dist(x, K)<2r\}$. The first inequality in the Wolff potential estimate \eqref{eq-sec2-wolff} then shows that for every $x \in K$ it holds
\begin{equation*}
\cW_{\mu\lfloor_{K}, r}(x) \leq \cW_{\mu, r}(x) \leq C_{\cW}(u(x)+a)+r \leq C_{\cW}(k+a)+1.
\end{equation*}
Hence
\begin{equation*}
\int_{\Omega} \cW_{\mu\lfloor_{K}, r} \,\mathrm{d}(\mu\lfloor_K) \leq C \mu(K)<\infty.
\end{equation*}
By the Wolff--Hedberg Theorem \cite[Theorem 3]{CGZG24}, this implies $\mu\lfloor_K \in (W^{1,G}_0(\Omega))'$.

Since $C_G(K)=0$, there exists a sequence of functions $\{\varphi_j\} \subset C^\infty_c(\rn)$ such that $\varphi_j \geq 1$ in a neighborhood of $K$ and $\|\varphi_j\|_{W^{1,G}(\rn)} \to 0$ as $j \to \infty$. Take a cutoff function $\eta \in C^\infty_c(\Omega)$ such that $\eta=1$ on $K$. Then $\varphi_j \eta \geq 1$ on $K$, $\varphi_j \eta \in C^\infty_c(\Omega)$, and $\|\varphi_j \eta\|_{W^{1,G}(\Omega)} \to 0$ as $j \to \infty$. Thus,
\begin{equation*}
\mu(K) = \mu\lfloor_K(K) \leq \int_\Omega \varphi_j\eta \,\mathrm{d}(\mu\lfloor_K) \leq \|\varphi_j\eta\|_{W^{1,G}(\Omega)} \|\mu\lfloor_K\|_{(W^{1,G}_0(\Omega))'} \to 0.
\end{equation*}
Consequently, $\mu_s(K)=0$.
\end{proof}

Proposition~\ref{prop-mu-s} has two corollaries.

\begin{corollary}\label{cor-rhs-identification}
Let $u$ be an $\cA$-superharmonic function in $\Omega$ with the Riesz measure $\mu=\mu[u]$. Then, for every $\varphi\in C^\infty_c(\Omega)$ and every $h\in W^{1,\infty}(\mathbb R)$ such that $h'$ has compact support,
\begin{equation*}
\int_\Omega h(u)\varphi\dmu=\int_\Omega h(u)\varphi\dmu_0+h(\infty)\int_\Omega\varphi\dmu_s.
\end{equation*}
Here we use the convention $h(u)=h(\infty)$ on $\{u=\infty\}$.
\end{corollary}

\begin{corollary}\label{cor-ren-unified}
Let $u$ be a renormalized solution to \eqref{eq:main} with a nonnegative Radon measure $\mu$, and let $\tilde{u}$ be the $\cA$-superharmonic representative of $u$ given by Theorem~\ref{thm-ren-sup}. Then
\begin{equation}\label{eq-unified-ren}
\int_\Omega \cA(x,\nabla \tilde{u}) \cdot \nabla(h(\tilde{u})\varphi) \dx = \int_\Omega h(\tilde{u})\varphi \dmu
\end{equation}
for every $\varphi\in C^\infty_c(\Omega)$ and every $h\in W^{1,\infty}(\mathbb R)$ such that $h'$ has compact support.
\end{corollary}

\begin{proof}
We first observe that the Riesz measure of $\tilde{u}$ is $\mu$. Indeed, taking $h\equiv 1$ in the renormalized formulation \eqref{eq-renormalized-intro} for $u$ and observing that $u=\tilde{u}$ a.e.\ and $\nabla u=\nabla \tilde{u}$ a.e., we obtain
\begin{equation*}
\int_\Omega \cA(x, \nabla \tilde{u}) \cdot \nabla \varphi \dx = \int_\Omega \varphi \dmu \quad\text{for all }\varphi \in C^\infty_c(\Omega).
\end{equation*}
Thus, $\mu[\tilde{u}]=\mu$.

Let now $h\in W^{1,\infty}(\mathbb R)$ be such that $h'$ has compact support, and let $\varphi\in C^\infty_c(\Omega)$. Choose $k>0$ such that $\supp h' \subset [-k, k]$. Then $h(t)=h(T_k(t))$ for all $t \in \mathbb{R}$. Since $T_k(u)=T_k(\tilde{u})$ a.e.\ in $\Omega$ and both sides belong to $W^{1,G}_{\rm loc}(\Omega)$, the uniqueness of $G$-quasicontinuous representative of Sobolev functions \cite[Theorem~17]{BHH18} implies that their $G$-quasicontinuous representatives coincide $G$-q.e. Hence $h(u)=h(\tilde{u})$ $G$-q.e.\ in $\Omega$. Since $\mu_0 \ll C_G$, we have
\begin{equation*}
\int_\Omega h(u)\varphi \dmu_0 = \int_\Omega h(\tilde{u})\varphi \dmu_0.
\end{equation*}
Using the renormalized formulation for $u$, together with $u=\tilde{u}$ a.e.\ and $\nabla u= \nabla \tilde{u}$ a.e., we therefore get
\begin{equation*}
\int_\Omega \cA(x,\nabla \tilde{u}) \cdot \nabla(h(\tilde{u})\varphi) \dx = \int_\Omega h(\tilde{u})\varphi \dmu_0 + h(\infty) \int_\Omega \varphi \dmu_s.
\end{equation*}
Applying Corollary~\ref{cor-rhs-identification} to $\tilde{u}$ identifies this split right-hand side with the single integral against $\mu$.
\end{proof}

Before we give the proof of Theorem~\ref{thm-sup-ren}, let us present the following energy bound for solutions to Dirichlet problems with zero boundary data.

\begin{lemma}\label{lem-energy-lambda}
Let $\Omega' \Subset \Omega$ be an open set, and let $\mu$ be a nonnegative bounded Radon measure with $K:=\supp\mu \Subset \Omega'$. Suppose that $u$ is a nonnegative $\cA$-superharmonic renormalized solution of \eqref{eq:main} in $\Omega$ (given by Corollary~\ref{cor-DP}), which satisfies $T_k(u) \in W^{1,G}_0(\Omega)$ for all $k>0$. Then, for any $0<r<\min{\{\dist(\Omega', \partial \Omega)/2, \dist(K, \partial \Omega'), 1\}}$, there exist constants $L, C>0$ such that
\begin{equation*}
u \in \cS_{\mu,r,L}(\Omega),
\end{equation*}
and 
for every $\lambda>L$ it holds
\begin{equation*}
\int_{\Omega} G(|\nabla (\min\{u, 2\lambda\}-\lambda)_+|) \dx\leq C \lambda \mu\left(\left\{\cW_{\mu,r}>\frac{\lambda}{L}\right\} \right).
\end{equation*} 
\end{lemma}

\begin{proof}
It follows from Lemma~\ref{lem-Wolff-S} that there exists $L_1>0$ such that $u\in \cS_{\mu,r,L_1}(\Omega')$. By the Wolff potential estimate, $u$ is locally bounded in $\Omega \setminus K$. Hence, by Lemma~\ref{lem-cS-supersol}, $u \in W^{1,G}_{\rm loc}(\Omega \setminus K)$. Taking $h \equiv 1$ in \eqref{eq-unified-ren} shows that $u$ is an $\cA$-solution in $\Omega \setminus K$. Thus, we may assume that $u$ is $\cA$-harmonic in $\Omega \setminus K$. In particular, $u$ is continuous in a neighborhood of $\partial \Omega'$. Set
\begin{equation*}
L_2:= \sup_{\partial \Omega'}u < \infty.
\end{equation*}
If $m>L_2$, then the function $T_k((u-m)_+)$ vanishes in a neighborhood of $\partial \Omega'$ for every $k>0$. Moreover, since $T_k(u) \in W^{1,G}_0(\Omega)$ for all $k>0$, we have $T_k((u-m)_+) \in W^{1,G}_0(\Omega \setminus \overline{\Omega'})$. Using $T_k((u-m)_+)$ as a test function and applying \eqref{ass-op}, we obtain
\begin{equation*}
0 = \int_{\Omega \setminus \overline{\Omega'}} \cA(x, \nabla u) \cdot \nabla T_k((u-m)_+) \dx \geq c_1^{\cA} \int_{\Omega \setminus \overline{\Omega'}} G\left( |\nabla T_k((u-m)_+)| \right) \dx.
\end{equation*}
This implies that $T_k((u-m)_+)=0$ a.e.\ in $\Omega \setminus \overline{\Omega'}$. Letting $k \to \infty$ and then $m \searrow L_2$ shows that
\begin{equation*}
u \leq L_2 \quad\text{in }\Omega \setminus \Omega'.
\end{equation*}
Now, let $x \in \Omega \setminus \Omega'$. Since $r<\dist(K, \partial \Omega')$, we have $\cW_{\mu, r}(x)=0$. Thus, we obtain
\begin{equation*}
C_{\cW}^{-1} (\cW_{\mu, r}(x) - r) = -\frac{r}{C_{\cW}} \leq u(x) \leq L_2=L_2+C_{\cW} \cW_{\mu, r}(x).
\end{equation*}
If we set $L_0=\max\{L_1, L_2\}$, then $u \in \cS_{\mu, r, L}(\Omega)$ for any $L \geq L_0$. We now fix $L:= 2C_{\cW}\max\{L_0, 1\}$. For $\lambda > L$, we define
\[h(t) = (\min\{t, 2\lambda\}-\lambda)_+,\]
which is clearly admissible in \eqref{eq-unified-ren}, as $h$ is Lipschitz continuous and $h'$ has compact support. Note that $\nabla h(u)=\chi_{\{\lambda<u<2\lambda\}} \nabla u$ a.e.\ in $\Omega$. Let $\varphi \in C^\infty_c(\Omega)$ be a nonnegative function such that $\varphi=1$ in a neighborhood of $K$. Applying Corollary~\ref{cor-ren-unified}, we have
\begin{equation*}
\int_{\Omega} \cA(x,\nabla u) \cdot \nabla(h(u)\varphi) \dx = \int_{\Omega} h(u)\varphi \dmu = \int_{\Omega} h(u) \dmu.
\end{equation*}
Since $1-\varphi$ vanishes in a neighborhood of $K$, the function $h(u)(1-\varphi)$ belongs to $W^{1,G}_0(\Omega \setminus K)$. Since $u$ is $\cA$-harmonic in $\Omega \setminus K$, we obtain
\begin{equation*}
\int_{\Omega} \cA(x,\nabla u) \cdot \nabla(h(u)(1-\varphi)) \dx = 0.
\end{equation*}
Combining these equalities and using \eqref{ass-op}, we have
\begin{equation*}
c_1^{\cA} \int_{\Omega} G(|\nabla h(u)|) \dx \leq \int_{\Omega} \cA(x, \nabla u) \cdot \nabla h(u) \dx = \int_{\Omega} h(u) \dmu.
\end{equation*}

Note that $\{h(u) \neq 0\} = \{u>\lambda\}$. Since $u \in \cS_{\mu, r, L_0}(\Omega)$, we have the inclusion $\{u>\lambda\} \subset \{\cW_{\mu, r}>(\lambda-L_0)/C_{\cW}\}$. However, it is straightforward to see that
\begin{equation*}
\frac{\lambda-L_0}{C_{\cW}} > \frac{\lambda - \lambda/2}{C_{\cW}} \geq \frac{\lambda}{L}.
\end{equation*}
Therefore, we obtain $\{h(u)\neq 0\} \subset \{\cW_{\mu, r}>\lambda /L\}$. This together with $h \leq \lambda$ leads to the conclusion as
\begin{equation*}
\int_{\Omega}h(u)\varphi\dmu \leq \lambda  \mu\left(\left\{\cW_{\mu,r}>\frac{\lambda}{L}\right\} \right).
\end{equation*} 
\end{proof}

Finally, we provide the proof of Theorem~\ref{thm-sup-ren}. We adapt the proof strategy from \cite{Kil11super}. The primary difficulty lies in handling the singular part $\mu_s$ of the measure $\mu$, which is concentrated on the set where the Wolff potential diverges, namely $\{\cW_{\mu, 1} = \infty\}$; this is a consequence of Proposition~\ref{prop-mu-s} and the Wolff potential estimates in Lemma~\ref{lem-Wolff}. To overcome this, we need to isolate these singularities to focus on the regular part $\mu_0$.

\begin{proof}[Proof of Theorem~\ref{thm-sup-ren}]
Let us write $\mu=\mu[u]$. Let $\varphi \in C^\infty_c(\Omega)$ and let $h \in W^{1,\infty}(\mathbb{R}) \subset C^{0,1}_{\rm loc}(\R)$ be such that $h'$ has compact support. We may assume that $\supp \varphi$ is contained in a single component of $\Omega$. Take nested domains $\supp \varphi \Subset \Omega_4 \Subset \Omega_3 \Subset \Omega_2 \Subset \Omega_1 \Subset \Omega$. Since $u$ is locally bounded below, we may also assume that it is nonnegative in $\Omega_1$, by adding a constant to $u$ in $\Omega_1$ and replacing $h(t)$ by a translated test function if necessary.

By Proposition~\ref{prop-mu-s}, $\mu_s$ is concentrated on $\{u=\infty\}$. By inner regularity of measure, for any $\varepsilon>0$ we can choose a compact set $K_\varepsilon \subset \{u=\infty\} \cap \Omega_3$ such that
\begin{equation*}
\mu_s(\Omega_3 \setminus K_\varepsilon) = \mu_s((\{u=\infty\} \cap \Omega_3) \setminus K_\varepsilon) < \varepsilon.
\end{equation*}
If we define the restricted measure
\begin{equation*}
\mu_\varepsilon :=\mu\lfloor_{\Omega_3 \setminus K_\varepsilon},
\end{equation*}
then $\supp\mu_\varepsilon \subset \overline{\Omega}_3 \subset \Omega_2$ and $\mu_{\varepsilon}(E)\leq \mu_0(E \cap \overline{\Omega}_3)+\varepsilon$ for any Borel set $E$.

Let $k>0$ be such that $\supp h' \subset [-k, k]$. This implies that $h(u) = h(u_k)$, where $u_k:=\min\{u, k\}$. Since $u$ is lower semicontinuous, the sublevel set $\{u\leq k\}$ is closed. Thus, the distance between the compact set $K_\varepsilon$ and the disjoint closed set $\{u \leq k\}$ is strictly positive, allowing us to define
\begin{equation*}
r:= \tfrac{1}{2}\min\left\{\dist(\Omega_4, \partial\Omega_3), \dist(\Omega_3, \partial\Omega_2), \mathrm{dist}(\Omega_2,\partial\Omega_1), \dist(\Omega_1, \partial \Omega), \mathrm{dist}(K_\varepsilon,\{u\leq k\}), 1 \right\}>0.
\end{equation*}
Then
\begin{equation*}
S_{\varepsilon}:= \{x\in \mathbb{R}^n:\dist(x,K_{\varepsilon})\leq r\}
\end{equation*}
is compactly contained in $\{u>k\} \cap \Omega_2$. Take $\theta_{\varepsilon}\in C_c^{\infty}(\{u>k\} \cap \Omega_2)$ such that $ 0\leq \theta_\varepsilon \leq 1$ and $\theta_{\varepsilon}=1$ on $S_{\varepsilon}$. Notice that $h(u)=h(k)$ is a constant on the support of $\theta_{\varepsilon}$. 

By the choice of $\theta_\varepsilon$ and \eqref{eq-riesz}, we have
\begin{equation}\label{eq-theta-riesz}
\int_{\Omega}h(u)\varphi \theta_{\varepsilon}\dmu= h(k)\int_{\Omega}\varphi\theta_{\varepsilon}\dmu=h(k)\int_{\Omega}\cA(x,\nabla u)\cdot\nabla (\varphi\theta_{\varepsilon})\dx =\int_{\Omega}\cA(x,\nabla u)\cdot\nabla (h(u)\varphi\theta_{\varepsilon}) \dx.
\end{equation}
Our goal is hence to show that
\begin{equation}\label{eq-theta}
\lim_{\varepsilon\to0} \left|\int_{\Omega}h(u)\varphi(1-\theta_{\varepsilon})\dmu- \int_{\Omega}\cA(x,\nabla u)\cdot\nabla(h(u)\varphi(1-\theta_{\varepsilon}))\dx \right| =0,
\end{equation}
which eventually yields the result of the theorem.

Note that $\mu_\varepsilon$ is a nonnegative bounded Radon measure on $\Omega_3$. By (the proof of) Corollary~\ref{cor-DP}, there exists an $\cA$-superharmonic renormalized solution $w_\varepsilon$ to
\begin{equation*}
-\DIV\cA(x, \nabla w_\varepsilon)=\mu_{\varepsilon} \quad\text{in }\Omega_1
\end{equation*}
such that $T_\ell(w_\varepsilon) \in W^{1,G}_0(\Omega_1)$ for all $\ell>0$. Applying Lemma~\ref{lem-energy-lambda} (with $\Omega=\Omega_1$ and $\Omega'=\Omega_2$), we obtain a constant $L_1>0$ such that $w_\varepsilon \in \cS_{\mu_\varepsilon, r, L_1}(\Omega_1)$, and for every $\lambda>L_1$,
\begin{align}\label{eq-w-eps-energy}
\begin{split}
\int_{\Omega_1}G(|\nabla (\min\{w_\varepsilon, 2\lambda\}-\lambda)_+|) \dx
&\leq C \lambda\mu_\varepsilon\left(\left\{\cW_{\mu_\varepsilon,r}>\frac{\lambda}{L_1}\right\} \right) \\
&\leq C\lambda \left( \mu_0\left(\left\{\cW_{\mu,r}>\frac{\lambda}{L_1}\right\} \cap \overline{\Omega}_3 \right)+\varepsilon\right).
\end{split}
\end{align}
Since $\cW_{\mu_\varepsilon,r}(x)=\cW_{\mu,r}(x)$ for all $x \in \Omega_4 \setminus S_\varepsilon$, we have $w_\varepsilon \in \cS_{\mu, r, L_1}(\Omega_4 \setminus S_\varepsilon)$. 
Moreover, Lemma~\ref{lem-Wolff-S} shows that 
$u \in \cS_{\mu,r,L_2}(\Omega_4)$ for some $L_2>0$. If we set $L=\max\{L_1, L_2\}$, then we have
\begin{equation*}
u, w_\varepsilon \in \cS_{\mu, r, L}(\Omega_4 \setminus S_\varepsilon).
\end{equation*}

To prove \eqref{eq-theta}, we approximate each integral through a regularization procedure. For the first integral in \eqref{eq-theta}, we consider the Riesz measure $\mu_m:= \mu[u_m]$ associated with the truncations $u_m:=\min\{u, m\}$ for $m \in \mathbb{N}$. Note that $\mu_m \rightharpoonup \mu$ as $m \to \infty$. Indeed, for every $\eta\in C^\infty_c(\Omega)$,
\begin{equation*}
\int_\Omega\eta\,d\mu[u_m] = \int_\Omega \cA(x,\nabla u_m)\cdot\nabla\eta\dx = \int_{\{u<m\}}\cA(x,\nabla u)\cdot\nabla\eta\dx.
\end{equation*}
Since $|\cA(x,\nabla u)\cdot\nabla\eta|\leq c_2^{\cA} \|\nabla \eta\|_{\infty}g(|\nabla u|)$ and $g(|\nabla u|)\in L^1_{\rm loc}(\Omega)$ by \eqref{ass-op} and Lemma~\ref{lem-super-int}, the dominated convergence theorem gives
\begin{equation*}
\int_\Omega \eta \dmu[u_m] \to \int_\Omega \eta \dmu.
\end{equation*}
By approximations, this distributional convergence implies the weak convergence against $C_c(\Omega)$, i.e.,
\begin{equation}\label{eq-weak-conv}
\int_\Omega \eta \dmu[u_m] \to \int_\Omega \eta \dmu \quad\text{for all }\eta \in C_c(\Omega).
\end{equation}

Recall that $h(u)=h(u_k)$. Since $u_k$ lacks sufficient regularity to serve directly as a test function for this weak convergence, we need to approximate it. By Lemma~\ref{lem-cS-supersol} and Theorem~\ref{thm-quasicont}, $u_k \in W^{1,G}_{\mathrm{loc}}(\Omega)$ and is $G$-quasicontinuous in $\Omega$. Thus, we can find a sequence of smooth functions $\{u_{k, j}\}_{j}$ that converges to $u_k$ in $W^{1, G}(\Omega_1)$ and $G$-q.e.\ in $\Omega_1$ as $j \to \infty$. This, in particular, implies that $u_{k, j} \to u_k$ $\mu_0$-a.e.\ as $j \to \infty$.

We approximate the first integral in \eqref{eq-theta} by
\begin{equation}\label{eq-theta-kj}
\int_{\Omega}h(u_{k,j})\varphi(1-\theta_{\varepsilon}) \dmu_m,
\end{equation}
which converges to
\begin{equation*}
\int_{\Omega}h(u_{k,j})\varphi(1-\theta_{\varepsilon}) \dmu = \int_{\Omega}h(u_{k,j})\varphi(1-\theta_{\varepsilon}) \dmu_0 + \int_{\Omega}h(u_{k,j})\varphi(1-\theta_{\varepsilon}) \dmu_s
\end{equation*}
as $m \to \infty$, due to the weak convergence \eqref{eq-weak-conv}, since $h(u_{k, j}) \varphi (1-\theta_\varepsilon) \in C_c(\Omega)$. Since $u_{k, j} \to u_k$ $G$-q.e.\ in $\Omega_1$, $\supp \varphi \subset \Omega_1$, and $h(u_k)=h(u)$, the dominated convergence theorem guarantees that
\begin{equation*}
\lim_{j \to \infty} \int_{\Omega}h(u_{k,j})\varphi(1-\theta_{\varepsilon})\dmu_0 = \int_{\Omega}h(u)\varphi(1-\theta_{\varepsilon})\dmu_0.
\end{equation*}
For the integrals with respect to the singular measure $\mu_s$, the choice of $\theta_\varepsilon$ ensures that $1-\theta_\varepsilon$ vanishes on $K_\varepsilon$. Thus, we have the estimate
\begin{equation*}
\left| \int_{\Omega} (h(u_{k,j})-h(u))\varphi(1-\theta_{\varepsilon})\dmu_s \right| \leq 2\|h\|_{\infty}\|\varphi\|_{\infty} \mu_s(\Omega_3 \setminus K_\varepsilon) < 2\varepsilon \|h\|_{\infty}\|\varphi\|_{\infty}.
\end{equation*}
Therefore, we obtain
\begin{equation}\label{eq-sec4-35}
\limsup_{j\to\infty}\lim_{m\to\infty}\left|\int_{\Omega} h(u_{k,j})\varphi(1-\theta_{\varepsilon})\dmu_m-\int_{\Omega}h(u)\varphi(1-\theta_{\varepsilon})\dmu\right|\leq C \varepsilon,
\end{equation}
where $C>0$ is independent of $\varepsilon$.

To further investigate \eqref{eq-theta-kj}, we introduce the function
\begin{equation*}
	\psi_{\lambda}:=\frac{(\min\{w_{\varepsilon}, 2\lambda\}-\lambda)_+}{\lambda},
\end{equation*}
and split the integral in \eqref{eq-theta-kj} into two parts as
\begin{equation}\label{eq-sec4-36}
\int_{\Omega}h(u_{k,j})\varphi(1-\theta_{\varepsilon})\psi_{\lambda}\dmu_m+ \int_{\Omega}h(u_{k,j})\varphi(1-\theta_{\varepsilon})(1-\psi_{\lambda})\dmu_m.
\end{equation}
The first term in \eqref{eq-sec4-36} is estimated as
\begin{equation}\label{eq-sec4-37}
\left| \int_{\Omega}h(u_{k,j})\varphi(1-\theta_{\varepsilon})\psi_{\lambda}\dmu_m \right| \leq \|h\|_{\infty}\int_{\Omega}|\varphi|(1-\theta_{\varepsilon})\psi_{\lambda}\dmu_m.
\end{equation}
Recall from Lemma~\ref{lem-cS-supersol} that $u_m \in W^{1,G}(\Omega_1)$, which ensures that $\mu_m \in (W^{1,G}_0(\Omega_1))'$. Furthermore, since $w_\varepsilon$ is a renormalized solution, it follows that $\psi_\lambda \in W^{1, G}_{\mathrm{loc}}(\Omega)$. Consequently, we have $|\varphi| (1-\theta_\varepsilon)\psi_\lambda \in W^{1,G}_0(\Omega_1) \cap L^\infty(\Omega_1)$. Thus, by the definition of the Riesz measure $\mu_m$, we obtain
\begin{equation}\label{eq-sec4-38}
    \begin{split}
        \int_{\Omega}|\varphi|(1-\theta_{\varepsilon})\psi_\lambda\dmu_m 
        &=\int_{\Omega}\cA(x,\nabla u_m)\cdot\nabla (|\varphi|(1-\theta_{\varepsilon})\psi_{\lambda})\dx\\
        &= \int_{\Omega}\cA(x,\nabla u_m)\cdot (\nabla (|\varphi|(1-\theta_{\varepsilon}))\psi_{\lambda} + |\varphi|(1-\theta_\varepsilon)\nabla \psi_\lambda)\dx.
    \end{split}
\end{equation}

It follows from Lemma~\ref{lem-cap-pot} that  $\widetilde{\mathrm{cap}}_{\cA}( \{ \psi_{\lambda}>0\} \cap \overline{\Omega}_4, \Omega_1) \to 0$  as $\lambda \to \infty$. Hence, by Lemmas~\ref{lem-cap-dual} and \ref{lem-cap-measure}, $|\{\psi_\lambda>0\} \cap \overline{\Omega}_4| \to 0$ as $\lambda \to \infty$. Since $g(|\nabla u|) \in L^1_{\mathrm{loc}}(\Omega)$ by Lemma~\ref{lem-super-int},  by using \eqref{ass-op} and $\supp\varphi \Subset \Omega_4$, we have
\begin{equation}\label{eq-sec4-312}
\begin{split}
& \left| \int_{\Omega}\cA(x,\nabla u_m)\cdot \nabla (|\varphi|(1-\theta_\varepsilon)) \,\psi_\lambda \dx \right|   \leq c_2^{\cA} \|\nabla (|\varphi|(1-\theta_\varepsilon))\|_\infty \int_{\Omega_4\cap \{w_\varepsilon>\lambda\}} g(|\nabla u|) \dx \to 0
\end{split}
\end{equation}
as $\lambda \to \infty$. On the other hand, since $\supp\varphi \Subset \Omega_4$ and $1-\theta_\varepsilon=0$ on $S_\varepsilon$, we have
\begin{equation*}
\left| \int_{\Omega}\cA(x,\nabla u_m)\cdot \nabla \psi_\lambda \, |\varphi|(1-\theta_{\varepsilon}) \dx\right| \leq \|\varphi\|_\infty \int_{\Omega_4 \setminus S_\varepsilon} |\cA(x, \nabla u_m) \cdot \nabla \psi_\lambda| \dx.
\end{equation*}
Since $u, w_\varepsilon \in \cS_{\mu,r,L}(\Omega_4\setminus S_{\varepsilon})$, we have the bound
\begin{equation}\label{eq-sec4-39}
    u\leq L+C_{\cW}\cW_{\mu,r}\leq L+C_{\cW} (r+C_{\cW}w_\varepsilon) <L+C_{\cW}^2(1+2\lambda) \quad \text{in }\{w_\varepsilon<2\lambda\}\cap(\Omega_4 \setminus S_{\varepsilon}),
\end{equation}
which implies that $u\leq c_0\lambda$ for some $c_0>0$ in the intersection of the support of $\nabla \psi_{\lambda}$ and $\Omega_4 \setminus S_{\varepsilon}$ for all sufficiently large $\lambda$. In this region, $u_m=u_{c_0\lambda}$ for all $m\geq c_0\lambda$. By \eqref{ass-op} and Lemma~\ref{lem-barg}(iii), for $\delta\in (0,1)$ it follows 
\begin{align}\label{eq-sec4-310}
    \begin{split}
    \int_{\Omega_4\setminus S_{\varepsilon}} &|\cA(x,\nabla u_m)\cdot\nabla \psi_{\lambda}| \dx  = \int_{\Omega_4 \setminus S_{\varepsilon}} |\cA(x,\nabla u_{c_0\lambda}) \cdot\nabla \psi_{\lambda}| \dx \\
    &\quad \leq \frac{C\delta}{\lambda} \int_{\Omega_4 \cap \{u\leq c_0\lambda\}} G(|\nabla u|) \dx + \frac{C\delta^{1-q}}{\lambda} \int_{\Omega_1} G(|\nabla (\min\{w_{\varepsilon}, 2\lambda\}-\lambda)_+|)\dx.
    \end{split}
\end{align}
By \eqref{eq-grad-k} applied in $\Omega_4$, we have
\begin{equation*}
\int_{\Omega_4 \cap\{u\leq c_0 \lambda\}}G(|\nabla u|)\dx \leq C\lambda.
\end{equation*}
Moreover, by \eqref{eq-w-eps-energy}, we have
\begin{equation*}
\int_{\Omega_1}G(|\nabla (\min\{w_{\varepsilon}, 2\lambda\}-\lambda)_+|)\dx \leq C\lambda\left( \mu_0\left(\left\{\cW_{\mu,r}>\frac{\lambda}{L}\right\}\cap \overline{\Omega}_3 \right)+\varepsilon\right)
\end{equation*}
for every $\lambda>L$, where $C$ is independent of $\lambda$. Taking $\delta=\varepsilon^{1/q}$, we obtain from \eqref{eq-sec4-310} that
\begin{equation*}
\int_{\Omega_4\setminus S_{\varepsilon}} |\cA(x,\nabla u_m)\cdot\nabla \psi_{\lambda}| \dx \leq C\varepsilon^{1/q} + C\varepsilon^{-\frac{q-1}{q}} \mu_0\left(\left\{\cW_{\mu,r}>\frac{\lambda}{L}\right\}\cap \overline{\Omega}_3 \right).
\end{equation*}

Since $\{\cW_{\mu,r}=\infty\}\cap \overline{\Omega}_3$ is contained in $\{u=\infty\}$ by the Wolff potential estimates, Theorem~\ref{th-1} gives $C_G(\{\cW_{\mu,r}=\infty\}\cap \overline{\Omega}_3) \leq C_G(\{u=\infty\})=0$. Moreover, since $\mu_0 \ll C_G$ and the sets $\{\cW_{\mu,r}>\frac{\lambda}{L}\}\cap \overline{\Omega}_3$ decrease to $\{\cW_{\mu,r}=\infty\}\cap \overline{\Omega}_3$, we have
\begin{equation}\label{eq-sec4-3101}
\limsup_{\lambda\to\infty} \limsup_{m\to\infty} \left| \int_{\Omega_1} \cA(x,\nabla u_m)\cdot\nabla\psi_\lambda\, |\varphi|(1-\theta_\varepsilon)\,dx \right| \le C\varepsilon^{1/q}.
\end{equation}
Combining \eqref{eq-sec4-37}, \eqref{eq-sec4-38}, \eqref{eq-sec4-312}, and \eqref{eq-sec4-3101}, we get
\begin{equation*}
\limsup_{\lambda\to\infty}\limsup_{j\to\infty}\limsup_{m\to\infty} \left| \int_\Omega h(u_{k,j})\varphi(1-\theta_\varepsilon)\psi_\lambda\,d\mu_m \right| \le C\varepsilon^{1/q}.
\end{equation*}
Hence by~\eqref{eq-sec4-35} and \eqref{eq-sec4-36}, for $C$  independent of $\varepsilon$, it holds
\begin{equation}\label{eq-sec4-314}
\limsup_{\lambda\to\infty}\limsup_{j\to \infty}\limsup_{m\to\infty} \left|\int_{\Omega} h(u_{k,j})\varphi(1-\theta_{\varepsilon})(1-\psi_{\lambda})\dmu_m-\int_{\Omega}h(u)\varphi(1-\theta_{\varepsilon})\dmu\right|\leq C(\varepsilon + \varepsilon^{1/q}).
\end{equation}

Next, we consider the first term on the left-hand side in \eqref{eq-sec4-314}. By~\eqref{eq-sec4-39} we have that $\varphi(1-\psi_{\lambda})(1-\theta_{\varepsilon})$ vanishes outside $\{u\leq c_0\lambda\}\cap (\Omega_4\setminus S_{\varepsilon})$ for sufficiently large $\lambda$. Hence, for $m\geq c_0\lambda$, it holds
\begin{equation*}
    \begin{split}
      \int_{\Omega} &h(u_{k,j})(1-\psi_{\lambda})(1-\theta_{\varepsilon})\varphi\dmu_m \\
        &\quad=\int_{\Omega_4 \setminus S_{\varepsilon}}\cA(x,\nabla u_{c_0\lambda})\cdot\nabla \varphi h(u_{k,j})(1-\theta_{\varepsilon})(1-\psi_{\lambda})\dx\\
        &\quad\quad+\int_{\Omega_4 \setminus S_{\varepsilon}}\cA(x,\nabla u_{c_0\lambda})\cdot\nabla(1-\theta_{\varepsilon})(1-\psi_{\lambda})\varphi h(u_{k,j})\dx\\
        &\quad\quad+\int_{\Omega_4 \setminus S_{\varepsilon}}\cA(x,\nabla u_{c_0\lambda})\cdot\nabla h(u_{k,j}) (1-\theta_{\varepsilon})(1-\psi_{\lambda})\varphi \dx\\
       &\quad\quad-\int_{\Omega_4 \setminus S_{\varepsilon}}\cA(x,\nabla u_{c_0\lambda})\cdot\nabla \psi_{\lambda} h(u_{k,j}) (1-\theta_{\varepsilon})\varphi \dx.
    \end{split}
\end{equation*}

We now take the limits of the above expression with respect to $m$, $j$, and $\lambda$ in sequence, starting with $m\to\infty$, followed by $j\to \infty$ and finally $\lambda\to \infty$.

The limits of the first two terms follow by the dominated convergence theorem, yielding
\begin{equation*}
    \begin{split}
        &\lim_{\lambda\to\infty}\lim_{j\to \infty}\lim_{m\to\infty}\int_{\Omega_4 \setminus S_{\varepsilon}}\cA(x,\nabla u_{c_0\lambda})\cdot \nabla \varphi h(u_{k,j})(1-\theta_{\varepsilon})(1-\psi_{\lambda})\dx =\int_{\Omega}\cA(x,\nabla u)\cdot\nabla \varphi h(u)(1-\theta_{\varepsilon})\dx
    \end{split}
\end{equation*}
and
\begin{equation*}
    \begin{split}
        &\lim_{\lambda\to\infty}\lim_{j\to \infty}\lim_{m\to\infty}\int_{\Omega_4 \setminus S_{\varepsilon}}\cA(x,\nabla u_{c_0\lambda})\cdot\nabla (1-\theta_{\varepsilon}) \varphi h(u_{k,j})(1-\psi_{\lambda})\dx =\int_{\Omega}\cA(x,\nabla u)\cdot\nabla (1-\theta_{\varepsilon}) \varphi h(u)\dx.
    \end{split}
\end{equation*}

The limit of the third term is justified by the weak convergence of $\nabla h(u_{k,j})$ to $\nabla h(u)$ together with the dominated convergence theorem, giving
\begin{equation*}
    \begin{split}
        &\lim_{\lambda\to\infty}\lim_{j\to \infty}\lim_{m\to\infty}\int_{\Omega_4 \setminus S_{\varepsilon}}\cA(x,\nabla u_{c_0\lambda})\cdot\nabla h(u_{k,j})  (1-\theta_{\varepsilon}) \varphi (1-\psi_{\lambda})\dx =\int_{\Omega}\cA(x,\nabla u)\cdot\nabla h(u) (1-\theta_{\varepsilon}) \varphi \dx.
    \end{split}
\end{equation*}

For the fourth term, we have
\begin{equation*}
\left| \int_{\Omega_4 \setminus S_{\varepsilon}}\cA(x,\nabla u_{c_0\lambda})\cdot\nabla \psi_{\lambda} h(u_{k,j})  (1-\theta_{\varepsilon}) \varphi\dx \right| \leq \|h\|_\infty\|\varphi\|_\infty \int_{\Omega_4\setminus S_\varepsilon} |\cA(x,\nabla u_{c_0\lambda})\cdot\nabla\psi_\lambda| \dx.
\end{equation*}
By the estimate leading to \eqref{eq-sec4-3101}, with $m=c_0\lambda$, we therefore obtain
\begin{equation*}
\limsup_{\lambda\to\infty}\limsup_{j\to\infty} \left| \int_{\Omega_4\setminus S_\varepsilon} \cA(x,\nabla u_{c_0\lambda})\cdot\nabla\psi_\lambda\, h(u_{k,j})(1-\theta_\varepsilon)\varphi \dx \right| \leq C\varepsilon^{1/q}.
\end{equation*}

Inserting the above estimates into \eqref{eq-sec4-314}, we infer that
\begin{equation*}
    \left| \int_{\Omega}h(u)(1-\theta_{\varepsilon})\varphi\dmu-\int_{\Omega}\cA(x,\nabla u)\cdot\nabla (h(u)(1-\theta_{\varepsilon})\varphi)\dx\right|\leq C(\varepsilon+\varepsilon^{1/q}),
\end{equation*}
which is precisely \eqref{eq-theta}. This together with \eqref{eq-theta-riesz} yields
\begin{equation*}
    \left|\int_{\Omega}h(u)\varphi\dmu-\int_{\Omega}\cA(x,\nabla u)\cdot\nabla (h(u)\varphi)\dx
    \right|\leq C(\varepsilon+\varepsilon^{1/q}).
\end{equation*}
Letting $\varepsilon\to0$, we obtain
\begin{equation*}
\int_\Omega\cA(x,\nabla u)\cdot\nabla(h(u)\varphi)\dx = \int_\Omega h(u)\varphi\dmu=\int_\Omega h(u)\varphi\dmu_0+h(\infty)\int_\Omega\varphi\dmu_s,
\end{equation*}
where the last equality holds due to Corollary~\ref{cor-rhs-identification}. Hence $u$ satisfies the renormalized formulation \eqref{eq-renormalized-intro}. Together with $T_k(u)\in W^{1,G}_{\rm loc}(\Omega)$ and $g(|\nabla u|)\in L^1_{\rm loc}(\Omega)$ from Lemma~\ref{lem-super-int}, this proves that $u$ is a renormalized solution.
\end{proof}

\subsection*{Acknowledgment}
IC acknowledges support of Narodowe Centrum Nauki (grant 2019/34/E/ST1/00120). MK is supported 
by the National Research Foundation of Korea (NRF) grant funded by the Korean government (MSIT) (RS-2026-25481961). YL and CZ are supported by  the National Natural Science Foundation of China (No.\ 12471128).

\subsection*{Conflict of interest} The authors declare that there is no conflict of interest. We also declare that this manuscript has no associated data.

\subsection*{Data availability} Data sharing is not applicable to this article as no datasets were generated or analyzed during the current study.

    
\end{document}